\documentclass[12pt]{amsart}
\usepackage[english]{babel}
\usepackage{graphicx}
\usepackage{color}
\usepackage{amsmath,amssymb,hyperref,srcltx}
\usepackage{accents}
\newtheorem{pro}{Proposition}[section]
\newtheorem{teo}{Theorem}[section]
\newtheorem{thmx}{Theorem}

\newtheorem{cor}{Corollary}[section]

\newtheorem{lem}{Lemma}[section]

\theoremstyle{definition}
\newtheorem{defi}{Definition}[section]

\theoremstyle{remark}
\newtheorem{rem}{Remark}[section]

\begin{document}

\title[Fixed points of nilpotent actions]
{Fixed points of nilpotent actions on surfaces of negative Euler characteristic}

\author{Javier Rib\'{o}n}
\address{Instituto de Matem\'{a}tica e Estat\'\i stica \\
Universidade Federal Fluminense\\
Campus do Gragoat\'a\\
Rua Marcos Valdemar de Freitas Reis s/n, 24210\,-\,201 Niter\'{o}i, Rio de Janeiro - Brasil }

\email{jribon@id.uff.br}
 
\thanks{MSC-class. Primary: 37C85, 37E30; Secondary: 37C25, 55M20, 57S25}

\maketitle

\bibliographystyle{plain}
\section*{Abstract}
 We prove that a locally nilpotent group $G$  of $C^{1}$ diffeomorphisms of a compact surface $S$  of 
 non-vanishing 
Euler characteristic has a finite orbit ${\mathcal O}$
whose cardinal is bounded by above by a function of the 
characteristic of Euler of $S$.

We focus on the case of negative Euler characteristic $\chi (S)$. Then we can choose 
${\mathcal O}$ so that it consists of  global contractible fixed points of the
subgroup $G_0$ of $G$ consisting of isotopic to the identity elements.
In particular $G$ has a global contractible fixed point if it consists of isotopic to the identity
elements. 
\section{Introduction}
Let $S$ be a compact surface of non-vanishing Euler characteristic $\chi (S)$.
A homeomorphism $f$  of $S$  
always has a finite orbit. Moreover, if $\chi (S) <0$ and $f$ is isotopic to the identity map, 
$f$ has a contractible fixed point, i.e. a point in $\pi (\mathrm{Fix}(\tilde{f}))$ where
$\pi : \tilde{S} \to S$ is the universal cover of $S$ and 
$\tilde{f}$ is the unique lift of $f$ to $\tilde{S}$  that commutes with 
all deck transformations.
Such properties can be deduced from the behavior of the Lefschetz number of the iterates of $f$.

 It is natural to try to extend the above results 
for cyclic groups of homeomorphisms 
to more general classes of groups even if we can need to replace homeomorphisms with 
flows of vector fields or $C^{1}$ diffeomorphisms.

Let us consider continuous dynamical systems. Lima shows that any continuous action of 
${\mathbb R}^{n}$ on a compact surface $S$ with $\chi (S) \neq  0$ has a global fixed point
\cite{Lima-com}. Such a result was generalized by Plante for actions of 
connected nilpotent Lie groups  $G$ \cite{Plante-fix}. The Lie algebra ${\mathfrak g}$ 
of the group $G$ is a nilpotent
Lie algebra of vector fields on $S$ and to find a global fixed point of $G$ is equivalent to obtain 
a common singular point for the vector fields in ${\mathfrak g}$.

It makes sense to replace vector fields with diffeomorphisms $C^{1}$-close to the identity 
since we can associate to them a sort of a pseudo-flow. 
Following this idea, Bonatti proved that an abelian group of homeomorphisms of $S$ whose
generators are sufficiently $C^{1}$-close to the identity has a global fixed point
\cite{Bonatti-s2, Bonatti-com}.
The analogous theorem for nilpotent groups and $S = {\mathbb S}^{2}$ whas showed by 
Druck, Fang and Firmo \cite{D-F-F}.

It is natural to approach the problem of existence of global fixed points for a subgroup 
$G$ of homeomorphisms of $S$ by considering the subgroups induced by $G$ in the mapping class 
group of certain $G$-invariant subsurfaces of $S$.
The Thurston decomposition is a valuable tool since the fixed point set of homeomorphisms in
Thurston normal form is minimal in some sense (cf. \cite{Jiang-Guo}).
This point of view was developed by Franks, Handel and Parwani to show the existence of 
a finite orbit for any abelian subgroup $G$ consisting of $C^{1}$-diffeomorphisms 
\cite{FHPs, FHP-g}. Moreover, for the case $S = {\mathbb S}^{2}$ they prove that it is possible to
find a finite orbit with at most $2$ elements for groups of orientation-preserving 
$C^{1}$-diffeomorphisms \cite{FHPs} 
whereas for the case $\chi (S) <0$ they show the
existence of a global fixed point if $G$ consists of isotopic 
to the identity diffeomorphisms \cite{FHP-g}.
In order to make the approach work, it is necessary to extend the Thurston classification to 
some infinite type surfaces and they obtain such an improvement for the case of 
abelian groups of $C^{1}$-diffeomorphisms \cite{FHPs, FHP-g}.

Frequently, the problem of existence of global fixed points is lifted to the universal cover of some 
$G$-invariant subsurface. Then  we can apply several results in the literature to find global fixed 
points for subgroups of $C^{1}$-diffeomorphisms of ${\mathbb R}^{2}$:
\cite{saponga-loc, FRV:arxiv} for the case $C^{1}$-close to identity, 
\cite{FHPs, Beg-Cal-Fir-Mie} for the abelian case and 
\cite{JR:arxivsp, Firmo-LeCalvez-Ribon:fixed} for the nilpotent case.

Since  the existence of global fixed points for nilpotent subgroups has been studied for 
the case of continuous actions \cite{Plante-fix} and groups generated by $C^{1}$-diffeomorphisms
close to identity of the sphere  \cite{D-F-F}, it is natural to consider the existence of global fixed
points and finite orbits for nilpotent subgroups of $C^{1}$-diffeomorphisms.
The existence of a finite orbit with at most $2$ elements for a nilpotent subgroup of 
orientation-preserving $C^{1}$-diffeomorphisms of the sphere was proved in \cite{JR:arxivsp},
generalizing the main result in \cite{FHPs}  to the nilpotent setting. 
The case of the torus was treated by Firmo and
the author. There exists a finite orbit if the group contains 
a $C^{1}$-diffeomorphism with non-vanishing
Lefschetz number \cite{Firmo-Ribon:finite}.  We show that 
nilpotent subgroups of irrotational $C^{1}$-diffeomorphisms of the torus  have 
global fixed points \cite{Firmo-Ribon:global}.
This paper completes the program of studying global fixed points of nilpotent subgroups of 
$C^{1}$-diffeomorphisms of compact surfaces by considering the case of negative Euler 
characteristic. 
\begin{thmx}
\label{teo:mainb}
Let $G$ be a locally nilpotent subgroup of $\mathrm{Diff}^{1}(S)$ where $S$ is a compact
surface with $\chi (S) \neq 0$. Then there exists a finite orbit ${\mathcal O}$
whose cardinal is bounded by above by a function of $\chi (S)$.
Moreover, if $\chi (S) <0$, we can choose ${\mathcal O}$ to be contained in 
${\mathrm Fix}_{c}(G_{0})$ where $G_{0}$ is the subgroup of $G$ of diffeomorphisms isotopic to
the identity.
\end{thmx}
\begin{thmx}
\label{teo:maina}
Let $G$ be a nilpotent subgroup of $\mathrm{Diff}_{0}^{1}(S)$ where $S$ is a compact 
surface of negative Euler characteristic. Then $\mathrm{Fix}_{c}(G) \neq \emptyset$.
\end{thmx}
Let us remark that  $\mathrm{Diff}^{1}(S)$ and $\mathrm{Diff}_{0}^{1}(S)$ stand for the 
group of $C^{1}$-diffeomorphisms of $S$ and its subgroup of isotopic to identity elements
respectively. 
Moreover, $\mathrm{Fix}_{c} (G)$ denotes the set of contractible fixed points of $G$ 
(see Definition \ref{def:identity}).

These results are generalizations of theorems of Frank, Handel and Parwani for the abelian case
\cite{FHP-g}.  On the other hand, the existence of the uniform upper bound for the cardinal of the finite orbit 
in Theorem \ref{teo:mainb} is  new. Such a result allows to extend the existence of a finite orbit to
the case of locally nilpotent groups. 
A function on the Euler characteristic providing such an upper bound
can be derived from the proof of 
Theorem \ref{teo:mainb} but is not  sharp.
The outlook of the proofs is discussed in section \ref{sec:outlook}.
  
There are other results of independent interest on the paper. For instance, 
under natural hypothesis we show that 
the groups in the descending central series of a nilpotent subgroup of 
$\mathrm{Diff}^{1}(S)$ 
are isotopically trivial relative to their sets of global fixed points (Theorem \ref{teo:deriso}).
Moreover, we exhibit several results of localization of global fixed points of nilpotent subgroups of
$C^{1}$-diffeomorphisms for the cases of nonnegative Euler characteristic 
(Propositions \ref{pro:clp2},  \ref{pro:clp3} and \ref{pro:ann}).

\section{Outlook of the paper}
\label{sec:outlook}
First, we focus on the proof of Theorem \ref{teo:maina}.
We can suppose that $S$ is orientable and  $G$ is finitely generated without lack of generality.
It suffices to consider finitely generated groups. 
The theorem holds for cyclic groups since the Lefschetz number 
$L(\mathrm{Fix}_{c}(\phi),\phi)$ is equal to $\chi (S)$
and hence non-vanishing for any $\phi \in \mathrm{Homeo}_{0}(S)$. 
Theorem  \ref{teo:maina}
is also satisfied for abelian groups by a theorem of Franks, Handel and Parwani \cite{FHP-g}.

The group $G/G'$ is  finitely generated and abelian.
Section \ref{sec:irr} is devoted to explain that  the derived group $G'$ is isotopically trivial. 
More precisely, any $f \in G'$ satisfies that $f$ is isotopic to the identity map relative to 
$\mathrm{Fix} (G')$ (Theorem \ref{teo:deriso}).
As a consequence, many of the arguments of the abelian case can be adapted to the 
nilpotent setting.  For instance, we will see that given any $f \in G$, we have 
$\mathrm{Fix}_{c} \langle G', f  \rangle \neq \emptyset$.

We follow the approach by Franks, Handel and Parwani in \cite{FHP-g}. We 
consider a sequence $f_1, \hdots, f_q$ of elements of $G$ such that the classes of any  
$e= \mathrm{rank}(G/G')$ of them 
in $G/G'$ generate a finite index subgroup of $G/G'$. 
The number $q$ depends on the topology of the surface $S$. 
Then we study  isotopy classes of elements of $G$ relative to a 
$G$-invariant compact subset of 
$\cup_{j=1}^{q} \mathrm{Fix}_{c} \langle G', f_{j} \rangle$ in order to find
$1 \leq i_1 < \hdots < i_e \leq q$ such that 
$\cap_{j=1}^{e} \mathrm{Fix}_{c} \langle G', f_{j} \rangle \neq \emptyset$. 
Therefore $\langle G', f_{i_1}, \hdots, f_{i_e} \rangle$ is a finite index subgroup of $G$ with 
a global contractible fixed point.
Then, it is easy to use Theorem \ref{teo:plane} to show that $G$ has also a global 
contractible fixed point. 

\strut

Since in particular 
$\cup_{j=1}^{q} \mathrm{Fix}_{c} \langle G', f_{j} \rangle \subset \mathrm{Fix}(G')$, our group 
$G$ is abelian up to isotopy rel 
$\cup_{j=1}^{q} \mathrm{Fix}_{c} \langle G', f_{j} \rangle$. 
As a consequence, we can generalize the approach in  \cite{FHP-g} to obtain
a Thurston decomposition for the group $G$ relative to the 
set $\cup_{j=1}^{q} \mathrm{Fix}_{c} \langle G', f_{j} \rangle$.
Here, it arises a technical issue since the set $\mathrm{Fix}_{c} \langle G', f_{l} \rangle$
is not open in $\cup_{j=1}^{q} \mathrm{Fix}_{c} \langle G', f_{j} \rangle$ in general. 
We replace $\cup_{j=1}^{q} \mathrm{Fix}_{c} \langle G', f_{j} \rangle$ with a $G$-invariant 
compact subset $K$ such that $K \cap \mathrm{Fix}_{c} (f_{l})$ is an open and closed subset
of $K$ for any $1 \leq l \leq q$. By construction, 
such a set is complete in some sense, containing representatives
of the Nielsen-classes that are relevant in our approach (see Definition
\ref{def:good_excellent}).
Another problem is that the set $K$ is not necessarily finite. We generalize the Thurston 
decomposition provided in \cite{FHPs, FHP-g} for the infinite type case to the nilpotent setting.
The main difference is that we need Lemma  \ref{lem:guarantees_thurston} to localize
fixed points of $\mathrm{Fix}_{c} \langle G' , f_{l} \rangle$ for $1 \leq l \leq q$.

\strut

Even if $G$ is abelian up to relative isotopy, such a property does not suffices to find a global
contractible fixed point of $G$. We need to find and localize global fixed points of groups of 
the form $\mathrm{Fix}_{c} \langle G',f \rangle$.  This is a major difficulty to adapt the approach
of Franks, Handel and Parwani and it is 
the subject of sections \ref{sec:annular}, \ref{sec:cyclic} and \ref{sec:loc}.
More precisely, we want to localize global fixed points of subgroups of $G$ in connected 
components $M$ of $S \setminus {\mathcal R}_{e}$ where ${\mathcal R}_{e}$ is the set of 
essential reducing curves of the Thurston decomposition. 
Such a connected component has Euler characteristic less or equal than $0$. 
Once we obtain $M \cap K \cap \mathrm{Fix}_{c} (f_j) \neq \emptyset$ for 
sufficiently many indices, we can obtain 
global contractible fixed points for finite index subgroups of $G$ by using the properties of
the Thurton decomposition.
The results of section \ref{sec:annular} localize fixed points for the case 
where $M$ is an annulus whereas
section \ref{sec:loc} deals with the case of negative Euler characteristic.

\strut

In order to localize points of $\mathrm{Fix}_{c} \langle G',f\rangle$ in section \ref{sec:loc}
we need to show a version 
of Theorem \ref{teo:maina} for the group $\langle G', f \rangle$. 
Since $G'$ is trivial up to isotopy rel $K$, this is a sort of cyclic case and hence simpler
than the general one. It is treated in section \ref{sec:cyclic} and requires enlarging the scope of
our techniques to finite type surfaces with punctures and  boundary components. 
This is the reason behind considering such surfaces in sections 
\ref{sec:thurston}, \ref{sec:annular} and \ref{sec:cyclic}.
Let us remark that the framework in section \ref{sec:thurston}  (see Definition \ref{def:frame})
is larger than the one described so
far in this section; for instance the group $G'$ can be replaced with certain normal subgroups 
$H$ such that $G/H$ is abelian. The reason is having a unified approach for the proofs of
Theorem \ref{teo:maina} and the localization result (Theorem \ref{teo:mainc}) in section 
\ref{sec:cyclic}.
Anyway, we think that the presentation of this section is easier to understand, contains all the
main ideas and avoids technicalities. 

\strut
 
Section \ref{sec:proofa} is devoted to complete the proof of 
the auxiliary Theorem \ref{teo:baux} that implies Theorem \ref{teo:maina} in a straightforward way. 
Theorem \ref{teo:mainb} is proved in section  \ref{sec:mainb}.
Such a proof relies on applying Theorem \ref{teo:baux} to a certain group of homeomorphisms.

\section{Preliminaries}
In this section we are going to introduce some concepts that are going to be important in the 
paper, concerning the properties of nilpotent groups, isotopies and Nielsen classes, and 
localization of global fixed points.

\subsection{Nilpotent groups}
We introduce here the descending central series of a group and some of its properties that will
be used in the paper. 
\begin{defi}
Let $G$ be a group. Consider subgroups $H$ and $K$ of $G$. We define $[H,K]$ as the 
subgroup generated by the commutators $[h,k]:=hkh^{-1}k^{-1}$ where $h \in H$ and $k \in K$. 
\end{defi}
 \begin{defi}
 \label{def:nilpotent}
 Let $G$ be a group. We denote ${\mathcal C}^{(0)} (G) =G$. We define the groups 
 $({\mathcal C}^{(j)} (G))_{j \geq 0}$ in the {\it descending central series} recursively, by using the
 formula ${\mathcal C}^{(j+1)} (G) = [{\mathcal C}^{(j)} (G), G]$ for $j \geq 0$. 
 The group ${\mathcal C}^{(1)} (G)$ is called the {\it derived group} of $G$ and is also denoted
 by $G'$. 
 
 We say that a group $G$ is {\it nilpotent} if there exists $k \in {\mathbb Z}_{\geq 0}$ such that 
 ${\mathcal C}^{(k)} (G) = \{1\}$. Moreover, the minimum $k \geq 0$ such that 
 ${\mathcal C}^{(k)} (G) = \{1\}$ is called the nilpotency class of $G$. 
 \end{defi}
\begin{rem}
The subgroups in the descending central series of $G$ are normal subgroups of $G$. 
\end{rem}
\begin{rem}
\label{rem:des_abe}
The group $\langle {\mathcal C}^{(j)}(G),f \rangle / {\mathcal C}^{(j+1)}(G)$ is abelian for 
all $f \in G$ and $j \geq 0$. 
Indeed, the subgroup  $[{\mathcal C}^{(j)}(G)]$ induced by ${\mathcal C}^{(j)}(G)$ in 
$[G]:=G/ {\mathcal C}^{(j+1)}(G)$ 
is contained in the center of $[G]$. Thus any subgroup $H$ of 
$\langle {\mathcal C}^{(j)}(G),f \rangle$ containing ${\mathcal C}^{(j+1)}(G)$ is a normal 
subgroup of $\langle {\mathcal C}^{(j)}(G),f \rangle$.

Moreover, we get
${\mathcal C}^{(1)} \langle G', f\rangle \subset {\mathcal C}^{(2)} (G)$. 
We deduce ${\mathcal C}^{(j)} \langle G', f\rangle \subset {\mathcal C}^{(j+1)} (G)$
for any $j \geq 1$. 
\end{rem}
The next result is a distortion property of the elements of ${\mathcal C}^{(j)} (G)$
modulo ${\mathcal C}^{(j+1)} (G)$ for $j \geq 1$.  
\begin{lem}
\label{lem:dist}
Let   
\[  \phi = [g_{1},h_{1}] \hdots [g_{m},h_{m}] \in {\mathcal C}^{(j)} (G)  \]
where either $(g_l , h_l)$ or $(h_l, g_l)$ belong to ${\mathcal C}^{(j-1)}(G) \times G$ 
for any $1 \leq l \leq m$.  Then, we obtain 
\begin{equation}
\label{equ:k2}
\phi^{k^{2}}  = 
[g_{1}^{k},h_{1}^{k}] \hdots [g_{m}^{k},h_{m}^{k}] \alpha_{k}
\end{equation}
 for some $\alpha_{k} \in {\mathcal C}^{(j+1)} (G)$ and any $k \in {\mathbb N}$.
\end{lem}
\begin{proof}
Given $\psi \in G$ we denote by $[\psi]$ its class in $[G]:=G/ {\mathcal C}^{(j+1)} (G)$. 
Notice that $[{\mathcal C}^{(j)} (G), G] = {\mathcal C}^{(j+1)} (G)$ implies
that $[{\mathcal C}^{(j)} (G)]$ is a subgroup of the center of $[G]$. Thus, we get
\[ [\phi]^{k^{2}}  =  [[g_{1}],[h_{1}]]^{k^{2}}  \hdots [[g_{m}],[h_{m}]]^{k^{2}}   \]
for any $k \in {\mathbb N}$. Therefore, it suffices to show the result for $m=1$.

Suppose that either $(g,h)$ or $(h,g)$ belongs to  ${\mathcal C}^{(j-1)}(G) \times G$.
It suffices to show $[[g], [h]]^{k} = [[g]^{k}, [h]] = [[g], [h]^{k}]$ for any $k \in {\mathbb N}$. 
Let us show the first equality by induction on $k$. It clearly 
holds for $k=1$. Suppose it holds for $k \geq 1$.  We have
\[ [[g],[h]]^{k+1} = [[g], [h]]  [[g] ,[h]]^{k} =  [g] [h] [g]^{-1} [h]^{-1}   [[g] ,[h]]^{k} =  \]
\[  [g] [[g] ,[h]]^{k}  [h] [g]^{-1} [h]^{-1} =
[g] [g]^{k} [h] [g]^{-k} [h]^{-1}  [h] [g]^{-1} [h]^{-1} = [[g]^{k+1}, [h]], \]
where the third equality holds 
since $[[g], [h]]^{k}$ belongs to the center of $[G]$ and the fourth equality is a consequence
of the hypothesis of induction. 
The proof of  $[[g], [h]]^{k} = [[g], [h]^{k}]$ for any $k \in {\mathbb N}$ is analogous. 
\end{proof}

\subsection{Groups of homeomorphisms}

We denote by $\mathrm{Homeo} (S)$ and $\mathrm{Homeo}_{0} (S)$
 the group of homeomorphisms of a manifold $S$ and its subgroup of 
isotopic to the identity elements respectively. Their analogues for diffeomorphisms of class $C^1$
will be denoted by $\mathrm{Diff}^{1} (S)$ and $\mathrm{Diff}_{0}^{1} (S)$.
If $S$ is orientable, we can define the subgroup $\mathrm{Homeo}_{+} (S)$
of $\mathrm{Homeo} (S)$ of orientation-preserving elements and its analogue 
$\mathrm{Diff}_{+}^{1} (S)$ for diffeomorphisms of class $C^{1}$.
\begin{defi}
Given a subset $C$ of $S$, we
denote by $\mathrm{MCG}(S, C)$ the mapping class group rel $C$, 
i.e. the group of isotopy classes relative to $C$
of homeomorphisms, of $S$. The mapping class group $\mathrm{MCG}(S)$ of $S$ is just
$\mathrm{MCG}(S, \emptyset)$. 
\end{defi}
\subsubsection{Nielsen classes}
\begin{defi}
Given $f \in \mathrm{Homeo} (S)$ we denote by $\mathrm{Fix} (f)$ the set of fixed points of $f$. 
Given a subgroup $G$ of $ \mathrm{Homeo} (S)$  we define the set 
$\mathrm{Fix} (G)$ of global fixed points as $\cap_{f \in G} \mathrm{Fix}(f)$.
\end{defi}
\begin{rem}
\label{rem:inv_nor}
Let $f, g \in  \mathrm{Homeo} (S)$. We have 
$\mathrm{Fix} (g f g^{-1}) = g (\mathrm{Fix}(f))$.  
As a consequence, given $G <  \mathrm{Homeo} (S)$ and a normal subgroup $H$ of $G$, 
the set $\mathrm{Fix}(H)$ is $G$-invariant.
\end{rem}
\begin{defi}
\label{def:identity}
Let $f \in \mathrm{Homeo}_{0} (S)$ where $S$ is a surface of negative Euler characteristic. 
Let $\pi: \tilde{S} \to S$ be the universal covering map of $S$. 
There exists a isotopy $(f_t)_{t \in [0,1]}$ in  $\mathrm{Homeo} (S)$ such that 
$f_0 = \mathrm{Id}_{S}$ and $f_1= f$. It can be lifted
to an isotopy $(\tilde{f}_{t})_{t \in [0,1]}$ in  $\mathrm{Homeo} (\tilde{S})$ such that 
$\tilde{f}_{0} = \mathrm{Id}_{\tilde{S}}$. The homeomorphism $\tilde{f}:= \tilde{f}_{1}$ is called
the {\it identity lift} of $f$ to $\tilde{S}$. 
Note that the identity lift does not depend on the isotopy since $\chi (S) <0$.
We define the set $\mathrm{Fix}_{c}(f)$  of contractible fixed points of $f$ as 
$\mathrm{Fix}_{c}(f) = \pi (\mathrm{Fix}(\tilde{f}))$.

We define the set $\mathrm{Fix}_{c}(G) = \cap_{f \in G} \mathrm{Fix}_{c}(f)$
of global contractible  fixed points for any subgroup $G$ of 
$\mathrm{Homeo}_{0} (S)$. 
\end{defi}
\begin{rem}
\label{rem:iso_nor}
Suppose $\chi (S) <0$. 
Let $\tilde{f}$ be the identity lift of $f \in \mathrm{Homeo}_{0}(S)$.
Consider a lift $\tilde{g}$ of $g \in  \mathrm{Homeo}(S)$.
Since $(g f_t g^{-1})_{t \in [0,1]}$ and 
$(\tilde{g} \tilde{f}_{t} \tilde{g}^{-1})_{t \in [0,1]}$ are isotopies issued from the identity map,  
it follows that  $g f g^{-1} \in \mathrm{Homeo}_{0}(S)$ and 
$\tilde{g} \tilde{f} \tilde{g}^{-1}$ is the identity lift of $g f g^{-1}$. 

Let $\tilde{\mathrm H}(S)$ be the group of all lifts of elements of $\mathrm{Homeo}(S)$.
Let $\tilde{\mathrm H}_{0}(S)$ be the group of identity lifts of elements of $\mathrm{Homeo}_{0}(S)$.
The previous discussion implies that 
$\mathrm{Homeo}_{0}(S)$ and  $\tilde{\mathrm H}_{0}(S)$  are normal subgroups of 
$\mathrm{Homeo}(S)$ and  $\tilde{\mathrm H}(S)$ respectively.
In particular, given a subgroup $G$ of $\mathrm{Homeo}(S)$, its subgroup 
$G_0 = G \cap \mathrm{Homeo}_{0}(S)$, of isotopic to the identity elements, is normal.
\end{rem}
\begin{rem}
\label{rem:inv_nor0}
Let $f \in \mathrm{Homeo}_{0} (S)$, $g \in  \mathrm{Homeo} (S)$ and $\chi (S) <0$. 
We obtain $\mathrm{Fix}_{c} (g f g^{-1}) = g (\mathrm{Fix}_{c}(f))$
by Remark \ref{rem:iso_nor}.  
In particular, given $G <  \mathrm{Homeo} (S)$ and a normal subgroup $H$ of $G$
contained in $\mathrm{Homeo}_{0}(S)$, the set $\mathrm{Fix}_{c} (H)$ is $G$-invariant. 
\end{rem}
Contractible fixed points are an example of a Nielsen class.
\begin{defi}
\label{def:Nielsen}
Let $f \in \mathrm{Homeo} (S)$. We say that $z, w \in S$ are $f$-Nielsen equivalent if 
there exists a path $\gamma:[0,1] \to S$ with $\gamma (0)= z$, $\gamma (1)=w$ and 
such that $\gamma$ and $f \circ \gamma$ are homotopic relative to ends.
\end{defi}
\begin{rem}
The $f$-Nielsen class ${\mathcal N}_{f,z}$ 
of $z \in \mathrm{Fix}(f)$ is $\pi (\mathrm{Fix}(\tilde{f}))$ where 
$\pi: \tilde{S} \to S$ is the universal covering map and $\tilde{f}$ is a lift of $f$ to $\tilde{S}$
such that $z \in \pi (\mathrm{Fix}(\tilde{f}))$.
\end{rem} 
\begin{defi}
We say that a $f$-Nielsen class ${\mathcal N}$ is {\it compactly covered}
if there exists a lift (and hence for every lift) $\tilde{f}$ of $f$ to $\tilde{S}$   such that 
$\pi (\mathrm{Fix}(\tilde{f})) = {\mathcal N}$ and $\mathrm{Fix}(\tilde{f})$ is compact.
\end{defi}
\begin{defi}
\label{def:Nielsen_iso}
Let $(f_{t})_{t \in [0,1]}$ be an isotopy  in $\mathrm{Homeo} (S)$. 
Let ${\mathcal N_0}$ be a $f_0$-Nielsen class and consider a lift 
$\tilde{f}_0$ of $f_0$ to $\tilde{S}$ with
$\pi (\mathrm{Fix}(\tilde{f}_0)) = {\mathcal N}_0$. 
Consider an isotopy $(\tilde{f}_t)_{t \in [0,1]}$ in $\mathrm{Homeo} (S)$
issued from $\tilde{f}_0$. We say that $\pi (\mathrm{Fix}(\tilde{f}_1))$
is the $f_1$-Nielsen class induced by ${\mathcal N}_0$ and the isotopy $(f_{t})_{t \in [0,1]}$.
\end{defi}
\begin{rem}
The definition does not depend on the choice of the lift 
$\tilde{f}_{0}'$ of $f_0$ to $\tilde{S}$ with 
$\pi (\mathrm{Fix}(\tilde{f}_{0}')) = {\mathcal N}_0$. Indeed, 
$\tilde{f}_{0}'$  is of the form 
$T \tilde{f}_0 T^{-1}$ for some deck transformation $T$.
We deduce that $\tilde{f}_{1}' = T \tilde{f}_{1} T^{-1}$ and hence 
$\pi (\mathrm{Fix}(\tilde{f}_{1}')) = \pi (T(\mathrm{Fix}(\tilde{f}_1))))= \pi (\mathrm{Fix}(\tilde{f}_1))$.
\end{rem}
\begin{defi}
Let $G$ be a subgroup of $\mathrm{Homeo}(S)$ and let $\varpi: S' \to S$ be a covering map. 
We say that a subgroup $G'$ of $\mathrm{Homeo}(S')$ is a lift of $G$ if 
any $f' \in G'$ is a lift of some $f \in G$ and the map 
$\tau: G' \to G$, defined by $\tau (f')=f$,  is an isomorphism of groups.
\end{defi}
\begin{rem}
Let us point out that the set ${G}''$
of all lifts of elements of $G$ to $S'$ {\it is not a lift} of $G$ if 
$\varpi$ is not a homeomorphism. Indeed, the kernel of $\tau: G'' \to G$ is the group of
deck transformations. A lift of a group 
is useful since it is nilpotent if $G$ is nilpotent. On the other hand
$G''$ is not nilpotent even if $G$ is nilpotent if $\chi (S) <0$, since the group of deck transformations 
of a surface of negative Euler characteristic is non-nilpotent.
\end{rem}


\subsection{Fixed points}
In this section we introduce some results about the existence of global fixed points in the 
compactly covered case.
\begin{defi}
Let $f \in \mathrm{Homeo}(S \setminus K)$ where $S$ is a compact surface 
(maybe with boundary) and $K$ is a closed subset of $S \setminus \partial S$. 
Consider $z \in \mathrm{Fix} (f)$. We say  that $z$ {\it peripherally contains} a puncture 
$w$ rel $K$ if  there exists a path $\gamma:[0,1) \to S \setminus K$
such that $\gamma (0)=z$, $\lim_{t \to 1} \gamma (t)=w$ and 
$\gamma$ is properly homotopic to $f \circ \gamma$ relative to $0$. 
If $\partial$ is a boundary component of $S$, we say that $z$ peripherally contains $\partial$
if $z$ peripherally contains the puncture induced by $\partial$ in the surface 
$S \setminus \partial S$. 
\end{defi}
\begin{defi}
Let $f \in \mathrm{Homeo}_{+} ({\mathbb R}^{2})$ that preserves a non-empty compact set
$K$ contained in ${\mathbb R}^{2} \setminus \mathrm{Fix}(f)$.
We denote by $P(K,f)$ the subset of $\mathrm{Fix}(f)$ of points that peripherally do not contain 
$\infty$ rel $K$.
\end{defi}
In order to find compactly covered 
 $f_{|{\mathbb R}^{2} \setminus K}$-Nielsen classes it is interesting to consider $P(K,f)$
 since all such Nielsen classes are contained in $P(K,f)$.
 
  Next, we introduce some results about global fixed points of subgroups of 
 $\mathrm{Diff}_{+}^{1}({\mathbb R}^{2})$ that will be useful in the sequel. 
 The first one is a generalization of a theorem of Franks, Handel and Parwani \cite{FHPs}
 for finitely generated abelian groups to the nilpotent setting. 
\begin{teo} \cite[Theorem 1]{Firmo-LeCalvez-Ribon:fixed}
\label{teo:plane}
Let $G$ be  a   nilpotent subgroup of $\mathrm{Diff}_{+}^{1}({\mathbb R}^{2})$ 
that preserves a non-empty compact set. Then $G$ has a global fixed point.
\end{teo}
The nilpotence condition is useful in the quest of global fixed points since 
 it is sort of a rigidity condition.  For instance, properties of existence of fixed points that hold
 for a single element of the group can be frequently generalized to properties of existence of
 global fixed points. The next results illustrate this point of view.
\begin{cor}
\label{cor:plane} \cite{JR:arxivsp}
Let $G$ be a  
nilpotent subgroup of $\mathrm{Diff}_{+}^{1}({\mathbb R}^{2})$ such that
$\mathrm{Fix}(\phi)$ is a non-empty compact set for some $\phi \in G$.
Then $G$ has a global fixed point.
\end{cor}
\begin{pro} \cite[Theorem 5]{Firmo-LeCalvez-Ribon:fixed}
\label{pro:clp} 
Let $K_{0}$ be a non-empty compact subset of ${\mathbb R}^{2}$.
Consider $f \in \mathrm{Diff}_{+}^{1}({\mathbb R}^{2})$ such that
$f(K_{0})=K_{0}$ and  $\mathrm{Fix}(f) \cap K_{0} = \emptyset$.
Then there exists a lift $\tilde{f}$ of $f$ to the universal covering of a connected
component of ${\mathbb R}^{2} \setminus K_{0}$ such that
$\mathrm{Fix}(\tilde{f})$ is a non-empty compact set whose projection ${\mathcal N}$
to ${\mathbb R}^{2}$ satisfies that the Lefschetz number $L({\mathcal N},f)$ is non-zero.
\end{pro}
\begin{pro}
\label{pro:clp2} 
Let $G = \langle H, f \rangle$ be a nilpotent subgroup of $\mathrm{Diff}_{+}^{1}(S)$
where $S$ is an orientable surface and $H$ is a normal subgroup of $G$. 
Suppose that $K_{0}$ is a compact $f$-invariant subset of $S$ and 
consider a connected component $M$ of $S \setminus K_0$.
Assume that $M$ is not a sphere, a torus or an annulus such that $f$ permutes its ends. 
Fix a non-empty compactly covered $f_{|M}$-Nielsen class ${\mathcal N}$. 
Furthermore, suppose that
any $h \in H$ is isotopic to $\mathrm{Id}$
rel $K_0$. Then, there exists $z \in \mathrm{Fix} (G) \cap  {\mathcal N}$. 
Moreover,  $\mathrm{Fix} (G) \cap \mathrm{Fix}_{c}(H) \cap  {\mathcal N} \neq \emptyset$
holds if $\chi (S) <0$.
\end{pro}
\begin{proof}
Let $\varpi: \tilde{M} \to M$ be the universal covering map.
Fix a point $w^{\sharp}$ in  $\varpi^{-1}({\mathcal N})$ and a lift ${f}^{\sharp}$ of 
$f$ such that $f^{\sharp} (w^{\sharp})= w^{\sharp}$. 

We claim that $M$ is a topological disc or $\chi (M) < 0$. 
Otherwise $M$ is a sphere, a torus or an annulus. 
We just have to consider the case where $M$ is an annulus and $f$ does not permute its ends by
hypothesis.
This never happens since it would contradict that ${\mathcal N}$ is  compactly covered. 

Given $h \in H$, consider an isotopy $(h_t)_{t \in [0,1]}$ rel $K_0$ from 
 $h_0 = \mathrm{Id}_{S}$ to $h_1 = h$.
 It induces an isotopy $({h}_t^{\sharp})_{t \in [0,1]}$ from $\mathrm{Id}_{\tilde{M}}$ to 
the identity lift  $h^{\sharp}$ of $h$ where $({h}_t^{\sharp})_{t \in [0,1]}$ is a lift of 
the isotopy $({h}_t)_{t \in [0,1]}$ 
Notice that $h^{\sharp}$ does not depend on the isotopy $({h}_t^{\sharp})_{t \in [0,1]}$
since  $M$ is a topological disc or $\chi (M) < 0$. 
We denote  
 $H^{\sharp} = \{ h^{\sharp} : h \in H \}$ and 
  $G^{\sharp} = \langle H^{\sharp}, f^{\sharp} \rangle$.

  Since $H$ is a normal subgroup of $G$, $H^{\sharp}$ is a normal subgroup of
  $G^{\sharp}$ by Remark \ref{rem:iso_nor}.
 Since the map  $h \to h^{\sharp}$ is well-defined in $H$,
 the group  $H^{\sharp}$
 does not contain non-trivial deck transformations. As a consequence 
 ${\mathcal C}^{(q)}(G)= \{ \mathrm{Id} \}$ and $H^{\sharp} \lhd G^{\sharp}$
 imply ${\mathcal C}^{(q)}(G^{\sharp})= \{ \mathrm{Id} \}$.
 In particular $G^{\sharp}$ is nilpotent.
  
 Since ${\mathcal N}$ is compactly covered, $\mathrm{Fix}(f^{\sharp})$ is a non-empty
 compact set.  There exists $z^{\sharp} \in \mathrm{Fix}(G^{\sharp})$
 by Corollary \ref{cor:plane}.
 The point $z = \varpi(z^{\sharp})$ belongs to $\mathrm{Fix} (G) \cap  {\mathcal N}$. 
 
 Finally, suppose $\chi (S) <0$. Given $h \in H$, we consider the path $\gamma_{h}:[0,1] \to S$
 defined by $\gamma_{h}(t)= h_{t} (z)$ for $t \in [0,1]$. Since $H^{\sharp}$ consists of 
 identity lifts rel $K_0$, we deduce that $[\gamma_{h}]=0$ in 
 $\pi_{1} (M, z)$. Since $M \subset S$, it follows that  $[\gamma_{h}]=0$ in 
 $\pi_{1} (S, z)$. Therefore, $z$ belongs to $\mathrm{Fix}_{c}(h)$ for any $h \in H$ and thus 
 $z \in \mathrm{Fix}_{c}(H)$.
\end{proof}
\begin{cor}
\label{cor:clp2} 
Let $G = \langle H, f \rangle$ be a nilpotent subgroup of $\mathrm{Diff}_{+}^{1}({\mathbb R}^{2})$
where $H$ is a normal subgroup of $G$. 
Suppose that $K_{0}$ is a non-empty compact subset of ${\mathbb R}^{2}$ such that 
$K_0 \cap \mathrm{Fix}(f) = \emptyset$ and any element $h$ of $H$ is isotopic to $\mathrm{Id}$
rel $K_0$. Then there exists $z \in \mathrm{Fix} (G)$ such that its 
$f_{|{\mathbb R}^{2} \setminus K_0}$-Nielsen class 
is compactly covered.
\end{cor}
\begin{proof}
There exists a non-empty compactly covered $f_{|{\mathbb R}^{2} \setminus K_0}$-Nielsen class 
${\mathcal N}$ 
by  Proposition \ref{pro:clp}. We are done, since 
$\mathrm{Fix} (G) \cap {\mathcal N} \neq \emptyset$ by Proposition \ref{pro:clp2}.
\end{proof}
Let us generalize the previous corollary, weaking the hypotheses that we require $H$ to satisfy.
\begin{pro}
\label{pro:clp3}
Let $G = \langle H, f \rangle$ be a nilpotent subgroup of $\mathrm{Diff}_{+}^{1}({\mathbb R}^{2})$
where $H$ is a normal subgroup of $G$. 
Suppose that $K_{0}$ is a non-empty compact subset of ${\mathbb R}^{2}$ such that 
$K_0 \cap \mathrm{Fix}(f) = \emptyset$ and $K_0 \subset \mathrm{Fix}(H)$. Then 
there exists $z \in \mathrm{Fix} (G) \cap P(K_0, f)$. 
\end{pro}
\begin{proof}
There exists a $f$-minimal set $K_{0}'$ contained in $K_0$. Since  
\[  \mathrm{Fix} (G) \cap P(K_{0}', f) \subset  \mathrm{Fix} (G) \cap P(K_0, f), \]
we can consider that $K_0$ is $f$-minimal up to replace it with $K_{0}'$.

Suppose that any $h \in H$ is isotopic to $\mathrm{Id}$ rel $K_0$. 
Then there exists $z \in \mathrm{Fix}(G)$ whose $f_{|{\mathbb R}^{2} \setminus K_0}$-Nielsen
class is compactly covered by Corollary \ref{cor:clp2} and hence 
$z \in \mathrm{Fix} (G) \cap P(K_0, f)$. So we can assume from now on that there exists 
an element of $H$ that is not isotopic to $\mathrm{Id}$ rel $K_0$.

Let $k \geq 0$ be the maximum nonnegative integer number such that 
$H \cap {\mathcal C}^{(k)} (G)$ contains an element $h$ that is not isotopic to $\mathrm{Id}$ rel 
$K_0$. Such a $k$ exists since $G$ is nilpotent.

Let ${\mathcal W} (h, K_0, \emptyset)$ be the set of finite type subsurfaces $W$ of 
${\mathbb R}^{2} \cup \{ \infty \}$ such that 
\begin{itemize}
\item $\partial W$ is a union of pairwise disjoint essential simple closed curves of 
${\mathbb R}^{2} \setminus K_0$;
\item $K_0 \setminus W$ is finite and
every connected component of $W$ contains infinitely many points of $K_0$;
\item There exists $h' \in \mathrm{Homeo}_{+} ({\mathbb R}^{2})$ such that $h$ is isotopic to 
$h'$ rel $K_0$ and $h_{|W}' \equiv \mathrm{Id}$.
\end{itemize}
We say that $W \in {\mathcal W} (h, K_0, \emptyset)$ is maximal if given 
$W' \in  {\mathcal W} (h, K_0, \emptyset)$ such that $W$ is isotopic rel $K_0$ to a subsurface 
of $W'$ then $W$ is isotopic to $W'$ rel $K_0$.
There exists a unique maximal element $W$ of ${\mathcal W} (h, K_0, \emptyset)$
\cite[Lemmas 2.6 and 2.8]{FHPs}.

Given $g \in G$, we have $\eta:=[g,h] \in  H \cap {\mathcal C}^{(k+1)} (G)$.
Since $g h g^{-1} = \eta h$, we obtain 
$g( {\mathcal W} (h, K_0, \emptyset)) = {\mathcal W} (\eta h, K_0, \emptyset)$.
By the choice of $k$ we have that $\eta$ is isotopic to $\mathrm{Id}$ rel $K_0$ and hence
$ {\mathcal W} (\eta h, K_0, \emptyset) =  {\mathcal W} (h, K_0, \emptyset)$.
We deduce that $g(W)$ is the unique maximal element of $ {\mathcal W} (h, K_0, \emptyset)$
and thus $g(W)$ is isotopic to $W$ rel $K_0$ for any $g \in G$.

Consider the unbounded connected component $U$ of ${\mathbb R}^{2} \setminus \partial W$. 
Since $U$ is $G$-invariant modulo isotopy rel $K_0$, the compact set $K_0 \cap U$ is 
$G$-invariant. We claim that $K_0 \cap U= \emptyset$. Otherwise, we have $K_0 \cap U= K_0$ 
since $K_0$ is $f$-minimal. This implies $U \subset W$ and $W= {\mathbb R}^{2}$,  
contradicting that $h$ is not isotopic to $\mathrm{Id}$ rel $K_0$.

Given $g \in G$, there exists $\theta_{g} \in \mathrm{Homeo}_{+} ({\mathbb R}^{2})$ 
and an isotopy $(g_t)_{t \in [0,1]}$ rel $K_0$ such that 
$g_0 = \theta_g$, $g_1 =g$ and $\theta_{g} (U)=U$. Since 
$\chi (U) <0$, the class $[\theta_{g}] \in \mathrm{MCG} (U)$ depends just on $g$. 
Moreover $[G]:= \{ [\theta_{g}]: g \in G \}$ is a nilpotent subgroup of $ \mathrm{MCG} (U)$. 
The group $[G]$ is  virtually abelian and finitely generated \cite{BLM}.
Since $U$ is a finite type surface, it admits a Thurston decomposition \cite[Lemma 2.2]{Handel-top}.
In particular, up to replace $U$ with the unbounded connected component of 
$U \setminus {\mathcal R}$ (where ${\mathcal R}$ is the
set of reducing curves of the Thurston decomposition), we can suppose that all elements of 
$[G]$ are irreducible, i.e. either have finite order or are pseudo-Anosov.

There exists a model $\Theta$ for $[G]$. More precisely, there exists a subgroup 
$\Theta$ of $\mathrm{Homeo}_{+} (U)$ such that the the natural map 
$\theta \to [\theta]$ from $\Theta$ to $\mathrm{MCG}(U)$ is an injective morphism of groups 
whose image is equal to $[G]$  \cite[Lemma 2.1]{JR:arxivsp}.  
We can consider $U$ as a sphere with finitely many punctures (corresponding to the boundary
components of $U$). 
A nilpotent group of orientation-preserving irreducible homeomorphisms of the sphere
that has a fixed point, has a second fixed point \cite[Remark 2.2]{JR:arxivsp}.  
In our case $\infty \in \mathrm{Fix} (\Theta)$
and $\Theta$ fixes no other puncture of $U$ since $K_0$ is $f$-minimal. 
Thus, there exists $z \in \mathrm{Fix}(\Theta) \cap U$. 

Given $g \in G$, we can suppose that $(\theta_{g})_{|U} \in \Theta$. 
Consider the universal covering $\tilde{M}$ of the connected component $M$ of 
${\mathbb R}^{2} \setminus K_0$ containing $U$. 
Fix a lift $\tilde{z}$ of $z$ in $\tilde{M}$. There exists a lift 
$\tilde{\theta}_{g} \in \mathrm{Homeo}_{+} (\tilde{M})$ of $\theta_g$ such that 
$\tilde{\theta}_{g} (\tilde{z})=\tilde{z}$. Consider the lift $\tilde{g}$ of $g$ to $\tilde{M}$
with $\tilde{g} = \tilde{g}_1$ where $(\tilde{g}_t)$ is the lift of $(g_t)$ such that 
$\tilde{g}_0 = \tilde{\theta}_g$. Since $\chi (U) < 0$, $\tilde{g}$ does not depend on $(g_t)$. Thus 
$\tilde{G} := \{ \tilde{g}: g \in G \}$ is a nilpotent subgroup of $\mathrm{Diff}_{+} (\tilde{M})$. 

Any non-trivial element $\theta \in \Theta$ satisfies that the $\theta$-Nielsen class of 
$z$ has non-vanishing fixed point index, is compactly covered and does not peripherally contain
points of $U$. In particular, given $g \in G$ with $[\theta_{g}] \neq 1$, we obtain that 
$\mathrm{Fix}(\tilde{g})$ is a non-empty compact set
and its projection to ${\mathbb R}^{2}$ is a compactly covered 
$g_{|{\mathbb R}^{2} \setminus K_0}$-Nielsen class. Since $\chi (U) <0$ and $K_0$ is 
$f$-minimal, we obtain $[\theta_f] \neq 1$. There exists $\tilde{w} \in \mathrm{Fix}(\tilde{G})$
by Corollary \ref{cor:plane}.  Its projection $w$ in ${\mathbb R}^{2}$ is a global fixed point of $G$
whose $f_{|{\mathbb R}^{2} \setminus K_0}$-Nielsen class is compactly covered and hence 
$w \in \mathrm{Fix}(G) \cap P(K_0, f)$.
\end{proof} 
\section{Irrotational nature of derived groups}
\label{sec:irr}
In order to find global fixed points for a nilpotent subgroup $G$ of $\mathrm{Diff}_{+}^{1}(S)$,
we are going to work with compact $G$-invariant sets $K$ contained in the derived group 
$G'$.  In this section we see that any $\phi \in G'$ is isotopic to identity rel $\mathrm{Fix}(G')$
and hence rel $K$. This implies that the group induced by $G$ in the group of isotopy classes
rel $K$ is an image by a group morphism of $G/G'$ and in particular it is abelian. 
Such a property will be quite useful.

In this section $S$ is an orientable compact surface. 
The main result of the section is next
proposition.
\begin{teo}
\label{teo:deriso}
Let $G$ be a nilpotent subgroup of $\mathrm{Diff}_{+}^{1}(S)$
such that $G'$ is contained in $\mathrm{Diff}_{0}^{1}(S)$.
Consider $\phi \in {\mathcal C}^{(k)}(G)$ with $k \geq 1$. Then $\phi$ is isotopic to
$\mathrm{Id}$ rel $\mathrm{Fix}(\phi) \cap \mathrm{Fix}({\mathcal C}^{(k+1)}(G))$.
In particular $\phi$ is isotopic to
$\mathrm{Id}$ rel $\mathrm{Fix}({\mathcal C}^{(k)}(G))$.
\end{teo}
\begin{cor}
\label{cor:cont}
Let $G$ be a nilpotent subgroup of $\mathrm{Diff}_{+}^{1}(S)$
such that $G'$ is contained in $\mathrm{Diff}_{0}^{1}(S)$.
Consider $\phi$ in ${\mathcal C}^{(k)}(G)$ with $k \geq 1$. Then
$\mathrm{Fix}(\phi) \cap \mathrm{Fix}({\mathcal C}^{(k+1)}(G)) \subset \mathrm{Fix}_{c}(\phi)$.
\end{cor}
\begin{cor}
\label{cor:cder}
Let $G$ be a nilpotent subgroup of $\mathrm{Diff}_{+}^{1}(S)$
such that $G'$ is contained in $\mathrm{Diff}_{0}^{1}(S)$.
Then we obtain $\mathrm{Fix}({\mathcal C}^{(k)}(G)) = \mathrm{Fix}_{c}({\mathcal C}^{(k)}(G))$
for any $k \geq 1$.
\end{cor}

In the proof of Theorem \ref{teo:deriso} we associate some isotopy classes to 
$\phi \in  {\mathcal C}^{(k)}(G)$. We will show that they are trivial by using the nilpotence of $G$
to see that they have finite order and then next lemma.

\begin{lem}
\label{lem:per-inv-mcg}
Let $f$ be an orientation-preserving homeomorphism  
of an orientable finitely punctured compact surface $S$ 
such that $f$ is isotopic to  $\mathrm{Id}$.  
Suppose that there exist a finite subset $F$ of $\mathrm{Fix}(f)$ and  
$k \in {\mathbb N}$ such that  $\chi (S \setminus F) <0$ and $f^{k}$ is isotopic to $\mathrm{Id}$
relative to $F$. Then $f$ is isotopic to $\mathrm{Id}$ rel to $F$. 
\end{lem}
\begin{proof}
In order to show the result we can 
replace $f$ with another element in $\mathrm{Homeo}(S)$ that is isotopic to $f$ rel $F$. 
Thus, we can suppose that $f_{|S \setminus F}$ is an isometry of a hyperbolic metric 
in $S \setminus F$ such that $f^{k} \equiv \mathrm{Id}$ by Kerckhoff's solution to the 
Nielsen realization problem \cite{Kerck}.  

Suppose $\chi (S) <0$. Thus  $f$ is the identity map 
since $f$ is a finite order element of $\mathrm{Homeo}_{+}(S)$ that is isotopic to  $\mathrm{Id}$.

Suppose $\chi (S) \geq 0$. Then either $S$ is a finitely punctured sphere 
${\mathbb S}^{2} \setminus A$ with  
$\sharp A \leq 2$ or $S$ is a torus ${\mathbb T}^{2}$. 
Suppose that we are in the former case. The result is obvious if 
$\sharp (A \cup F) \leq 3$ since $A \cup F \subset \mathrm{Fix} (f)$ and the 
mapping class group of a sphere minus $k$ points ($k \leq 3$) is isomorphic to the group of
permutations $S_k$. On the other hand, if $\sharp (A \cup F) > 3$, $f$ is an 
orientation-preserving finite order homeomorphism of the sphere 
with more than three fixed points and thus  $f \equiv \mathrm{Id}$.

Finally, let us assume $S= {\mathbb T}^{2}$. In such a case, we consider 
a lift $\tilde{f}$ of $f$ to the universal covering $\tilde{f}$ such that $\tilde{f} (\tilde{z})= \tilde{z}$
where $\tilde{z}$ is a lift of a point $z \in F$. 
Since $f^{k} \equiv \mathrm{Id}$, we obtain $\tilde{f}^{k} \equiv \mathrm{Id}$. 
Moreover, since $f$ is isotopic to $\mathrm{Id}$, it follows that any point in 
$\tilde{z} + {\mathbb Z}^{2}$ is contained in $\mathrm{Fix}(\tilde{f})$.
By compactifying ${\mathbb R}^{2}$ by adding a point at $\infty$, we can interpret 
$\tilde{f}$ as a finite order orientation-preserving homeomorphism of a sphere with infinitely many 
fixed points. Therefore $\tilde{f}$ is the identity map and $f$ is also the identity map.
\end{proof}
\subsection{Translation numbers} 
\label{subsec:translation}
In this section we study homeomorphisms $\phi \in \mathrm{Homeo}_{+} (S)$
that are isotopic to a product of finitely many Dehn twists.
Specifically, we examine translation numbers induced in the ideal boundary
of the universal cover of a connected component of $S \setminus K$ where 
$K$ is a compact subset of $\mathrm{Fix}(\phi)$.


Let $\phi \in \mathrm{Homeo}_{+} (S)$ where $S$ is an orientable compact  
surface of negative
Euler characteristic. Let $K$ be a compact subset of $\mathrm{Fix} (\phi)$. 
Moreover, assume that there are  
essential simple closed curves $\gamma_1, \hdots, \gamma_m$
in $S \setminus K$ such that the following properties hold:
\begin{itemize}
\item $\gamma_1, \hdots, \gamma_m$ are pairwise non-homotopic in $S \setminus K$;
\item There are pairwise disjoint closed annuli $A_1, \hdots, A_m$ in $S \setminus K$ whose
core curves are $\gamma_1, \hdots, \gamma_m$ respectively;
\item $\phi$ is isotopic to $f$ rel $K$, where $f_{|S \setminus \cup_{l=1}^{m} A_l}$ is the 
identity map and $f_{|A_l}$ is a Dehn twist for any $1 \leq l \leq m$. 
\end{itemize}
Consider the connected component $B$ of $S \setminus K$ containing $\gamma:=\gamma_{1}$.
We denote by $\gamma^{1}$ and $\gamma^{2}$ the connected components of  $\partial A_{1}$.
Consider a lift $\tilde{A}_1$ of $A_1$ to the universal covering $\tilde{B}$ of $B$
and denote by $\tilde{\gamma}$, $\tilde{\gamma}^{1}$ and $\tilde{\gamma}^{2}$ the lifts 
of $\gamma$, $\gamma^1$ and $\gamma^2$ respectively, contained in $\tilde{A}_1$.
Let $T$ be a deck transformation 
associated to $\gamma$ and such that 
$T(\tilde{A}_1)=\tilde{A}_1$. 

There are two possible cases, namely $B$ is an annulus or $\chi (B) <0$.   
First, suppose that $B$ is an annulus. 
Consider the prime end compactification $\hat{B}$ of $B$ in $S$ (cf. \cite{epstein:prime});
it is a closed annulus.  Let $\pi: \overline{B} \to \hat{B}$ the 
universal covering map. Then $\overline{B}$ is the union of $\tilde{B}$ and two 
disjoint topological lines $\partial_1$ and $\partial_2$. 
We can suppose that $\partial_l$ is in the same connected component of 
$\overline{B} \setminus \tilde{\gamma}$ as $\tilde{\gamma}^{l}$ for $l \in \{1,2\}$.

Suppose  $\chi (B) <0$.   
Then $B$ admits a complete hyperbolic metric such that 
if we consider the universal covering map $\pi: {\mathbb D} \to B$, where ${\mathbb D}$
is the hyperbolic disc, then any homeomorphism $g \in \mathrm{Homeo} (B)$ satisfies 
that any lift $\tilde{g}$ of $g$ to ${\mathbb D}$ has a continuous extension to a homeomorphism
of $\overline{\mathbb D} = {\mathbb D} \cup S_{\infty}$ \cite{Cantwell-Conlon:lifts}.
We stress that we do not need to suppose $\chi (B) \neq - \infty$. 
 Notice that $T$ is not parabolic since 
otherwise $\gamma$ is inessential in $S \setminus K$. 
 As a consequence we can suppose that $\gamma$ is a geodesic.
 
We define the quotient $\hat{B}$ of $\overline{\mathbb D} \setminus \mathrm{Fix}(T)$ by the group
$\langle T \rangle$. The set $\hat{B}$ is a closed annulus and the natural mapping
$\hat{\pi}: \overline{\mathbb D} \setminus \mathrm{Fix}(T) \to \hat{B}$ 
is its universal covering map.
The axis $\tilde{\gamma}$ of $T$ is a lift of $\gamma$. 
Moreover, $\tilde{\gamma}^{1}$ and $\tilde{\gamma}^{2}$ are simple curves such that
$\partial \tilde{\gamma}^{1} = \partial \tilde{\gamma}^{2} = \mathrm{Fix}(T)$.
Consider the connected component 
$E_{l}$ of ${\mathbb D} \setminus \tilde{\gamma}^{l}$ that
does not contain $\tilde{\gamma}$ for $l \in \{1,2\}$.
We define $\partial_{l} = (\partial E_{l} \cap S_{\infty}) \setminus \mathrm{Fix}(T)$
for $l \in \{1,2\}$.
 
Consider a lift   $\tilde{f}$ of   $f$   to $\tilde{B}$ 
that commutes with $T$. In both cases, 
the sets $\tilde{\gamma}^{1}$, $\tilde{\gamma}^{2}$, $\partial_1$ and $\partial_2$ are 
$\tilde{f}$-invariant. Hence, 
we can define translation numbers  $\rho (\tilde{f}, \tilde{\gamma}^{1})$, 
$\rho (\tilde{f}, \tilde{\gamma}^{2})$, 
$\rho (\tilde{f}, \partial_1)$ and  $\rho (\tilde{f}, \partial_2)$.
\begin{rem}
\label{rem:partial_rot}
Let $(I_t)_{t \in [0,1]}$ the isotopy from $f$ to $\phi$ rel $K$. 
The isotopy $(I_t)_{t \in [0,1]}$ and $\tilde{f}$ induce a lift $\tilde{\phi}$ to $\tilde{B}$.
We have 
$\tilde{\phi}_{|\overline{B} \setminus \tilde{B}} \equiv \tilde{f}_{|\overline{B} \setminus \tilde{B}}$
if $B$ is an annulus by the properties of the prime end compactification. 
Analogously,   $\tilde{\phi}_{|S_{\infty}} \equiv \tilde{f}_{|S_{\infty}}$ holds
if $\chi (B) <0$ (cf. \cite{Cantwell-Conlon:lifts}).
In particular, we obtain
$\rho (\tilde{\phi}, \partial_l) = \rho (\tilde{f}, \partial_l)$ for any $l \in \{1,2\}$.
\end{rem}
 \begin{lem}
 \label{lem:trans}
 Consider the above setting. Then 
 $\rho (\tilde{f}, \tilde{\gamma}^{1})=\rho (\tilde{f}, \partial_1)$ and 
  $\rho (\tilde{f}, \tilde{\gamma}^{2})=\rho (\tilde{f}, \partial_2)$.
 \end{lem}
 \begin{proof}
 Assume $\chi (B) <0$ since otherwise the result is obvious.
 Let us show  $\rho (\tilde{f}, \tilde{\gamma}^{1})=\rho (\tilde{f}, \partial_1)$, the proof of 
 the other equality is analogous. 

 Let  $\tilde{f}_{1}$ be a lift of $f$ to ${\mathbb D}$ such that $\tilde{f}_{1}$ is the identity restricted to
$\tilde{\gamma}^{1}$. Since $\tilde{f}$ and $\tilde{f}_{1}$ commute with $T$ and the 
centralizer of $T$ in the group of deck transformations is $\langle T \rangle$, 
we deduce that $\tilde{f}_1 = T^{j} \circ \tilde{f}$ for some $j \in {\mathbb Z}$.
In particular, it suffices to show 
$\rho (\tilde{f}_1, \tilde{\gamma}^{1})=\rho (\tilde{f}_1, \partial_1)$.

Consider a path $\beta: [0,1] \to \overline{\mathbb D}$ such that
$\beta(0) \in \tilde{\gamma}^{1}$, $\beta(0,1) \subset E_{1}$ and
$\beta(1) \in   \partial_{1}$.
Consider the set $F= \{t \in (0,1) : \pi(\beta(t)) \in A_{1} \cup \hdots \cup A_{m} \}$.
If $F$ is empty then $\tilde{f}_{1}$ is the identity by restriction to $\beta[0,1]$ and
in particular the point $\beta(1)$ is a fixed point of $\tilde{f}_{1}$.
Since $\tilde{f}_{1}$ has fixed points in $\tilde{\gamma}^{1}$ and
$\partial_{1}$ then the translation numbers of
$\tilde{f}_{1}$ at $\tilde{\gamma}^{1}$ and
$\partial_{1}$ are both zero. 

So, we can assume $F \neq \emptyset$.  Let $t_{0} = \inf (F)$. The point
$\beta(t_{0})$ belongs to a lift of a component of some $\partial A_{l}$, hence
$\tilde{f}_{1}$ fixes the points in an axis of a covering transformation $T'$ associated to
some $\gamma_{l}$. Since clearly $T \neq T'$ then $\tilde{f}_{1}$ fixes two points in $\partial_{1}$
and $\rho (\tilde{f}_1, \tilde{\gamma}^{1})=\rho (\tilde{f}_1, \partial_1)=0$.
 \end{proof}

\subsection{Proof of Theorem \ref{teo:deriso}}
We denote ${\mathcal C}^{(j)} = {\mathcal C}^{(j)}(G)$ for simplicity.
Let $q$ be the nilpotency class of $G$. 
Since  ${\mathcal C}^{(q)} = \{ \mathrm{Id} \}$,
the result is obviously true for $k \geq q$.
Let us prove that if the result holds true for $k=j+1$ and $j \geq 1$ then it holds true for $k=j$.
An element $\phi \in {\mathcal C}^{(j)}$ is of the form
\[  \phi = [g_{1},h_{1}] \hdots [g_{m},h_{m}] \]
where either $(g_l , h_l)$ or $(h_l, g_l)$ belong to ${\mathcal C}^{(j-1)} \times G$ 
for any $1 \leq l \leq m$. 
Consider the diffeomorphisms $\alpha_k \in {\mathcal C}^{(j+1)}$ provided by Lemma 
\ref{lem:dist} for $k \geq 1$.

Consider the set $K= \mathrm{Fix}(\phi) \cap \mathrm{Fix}({\mathcal C}^{(j+1)})$.
If $K = \emptyset$ there is nothing to prove. Suppose that $K$ is a nonempty set.
Consider the Thurston decomposition of $\phi$ rel to $K$
(\cite[Proposition 3.1]{JR:arxivsp}, $H_{1}=\langle \phi,{\mathcal C}^{(j+1)} \rangle$, 
$H_{2}=G$, $K_{1}=K$, $K_{2}=\emptyset$).
Indeed there exists a finite set ${\mathcal R}$ of reducing curves in $S \setminus K$ 
that are core curves of a finite set of pairwise disjoint annuli ${\mathcal A}$ in $S \setminus K$ 
such that any $\psi \in G$ is isotopic rel $K$ to some $g \in \mathrm{Homeo}_{+}(S)$
satisfying
\begin{itemize}
\item $g$ preserves the decomposition, i.e. 
$g(\cup_{A \in {\mathcal A}} \partial A) = \cup_{A \in {\mathcal A}} \partial A$.
\item Let $M$ be a connected component of $S \setminus {\mathcal A}$.
Suppose $M \cap K$ is finite. Then $\chi (M \setminus K) <0$ and $g_{|M}$ has 
finite order or is a pseudo-Anosov homeomorphism of $M \setminus K$ if $g(M)=M$
\item Suppose $\sharp (M \cap K)= \infty$. Then $g_{|M} \equiv \mathrm{Id}$ if 
$\psi \in \langle \phi,{\mathcal C}^{(j+1)} \rangle$.
\end{itemize}
By induction hypothesis, $\eta$ is isotopic to $\mathrm{Id}$ rel $ \mathrm{Fix}({\mathcal C}^{(j+2)}, \eta)$
and then relatively to the smaller set $\mathrm{Fix}({\mathcal C}^{(j+1)}, \phi)$
for any $\eta \in {\mathcal C}^{(j+1)}$. Let
$f \in \mathrm{Homeo}_{+} (S)$ be a homeomorphism isotopic to $\phi$ rel $K$ in Thurston
normal form (see above). 
Consider a connected component $M_{0}$  of $S \setminus {\mathcal A}$.
and the class $[\phi]_{M_0} := [f_{|M_0}]$ induced in the mapping class group
 $\mathrm{MCG}(M_{0},M_{0} \cap K)$. It does not depend on the choice of $f$
 since $\sharp (M_0 \cap K) = \infty$ or $\chi (M_0 \setminus K) <0$.
 Notice that $[\phi]_{M_0}$ is trivial if $M_{0} \cap K$ is an infinite set.
If $M_{0} \cap K$ is finite then the subgroup of $\mathrm{MCG}(M_{0}, M_{0} \cap K)$ induced by
the stabilizer of $M_{0}$ in $G$ (modulo isotopy rel $K$)
is a nilpotent subgroup of a mapping class group of
a finite type surface.
Therefore it is virtually abelian \cite{BLM}, i.e it has a finite index normal subgroup that is 
abelian.
In particular there exist $k \in {\mathbb N}$ such that the classes induced by $g_j^{k}$ and 
$h_j^{k}$ belong to the stabilizer of $M_0$ and commute for any $1 \leq j \leq m$. 
Since $\alpha_{k}$ is isotopic to $\mathrm{Id}$ rel $K$, Equation (\ref{equ:k2}) implies that
$[\phi]_{M_0}$ is periodic.
Since $\phi$ is isotopic to $\mathrm{Id}$ and $\chi (M_0 \setminus K) <0$, 
$[\phi]_{M_0}$ satisfies the hypotheses of Lemma  \ref{lem:per-inv-mcg}. 
Therefore, we obtain $[\phi]_{M_0} \equiv \mathrm{Id}$. 
We deduce that $f$ is
either the identity map or
a product of non-trivial Dehn twists in closed annuli
$A_{1}$, $\hdots$ , $A_{m}$
whose core curves $\gamma_{1}$, $\hdots$, $\gamma_{m}$ 
are some of the curves in ${\mathcal R}$.
Let us prove that this last situation is impossible.

Consider the connected component $B$ of $S \setminus K$ containing $\gamma=\gamma_{1}$.
The elements of $G$ preserve the Thurston decomposition. Hence
up to replace $\phi$ with an iterate (see Equation (\ref{equ:k2}) in Lemma \ref{lem:dist}) 
we can suppose that $\phi$ is of the form
\[ \phi     =
[g_{1} ,h_{1}] \hdots [g_{m} ,h_{m}] \alpha \]
where  $g_{1}$, $h_{1}$, $\hdots$, $g_{m}$, $h_{m}$, fix $\gamma$ 
modulo isotopy rel $K$ and $\alpha \in {\mathcal C}^{(j+1)}$.
The open set $B$ is either an annulus or it has negative Euler characteristic. 
In the following, we adopt the notations in subsection \ref{subsec:translation}

Assume that $B$ is an annulus.
We have that $g_j, h_j, \alpha, \phi$ induce homeomorphisms 
$\hat{g}_j, \hat{h}_j, \hat{\alpha}, \hat{\phi}$ of the prime end compactification $\hat{B}$ of $B$. 
Since $\alpha$ is isotopic to $\mathrm{Id}$ rel $K$ by induction hypothesis, there exists a lift 
$\tilde{\alpha}$ of 
$\alpha$ to $\overline{B}$ such that $\tilde{\alpha}_{|\partial \overline{B}} \equiv \mathrm{Id}$.
We define
\begin{equation}
\label{equ:lift_p}
 \tilde{\phi} = [\tilde{g}_{1},\tilde{h}_{1}] \hdots [\tilde{g}_{m},\tilde{h}_{m}]  \tilde{\alpha},  
 \end{equation}
where $\tilde{g}_j$ (resp. $\tilde{h}_j$) is a lift of $\hat{g}_j$  (resp. $\hat{h}_j$)  for any 
$1 \leq j \leq m$.  Now, assume $\chi (B) <0$. 
Notice that since $\alpha$ is isotopic to $\mathrm{Id}$ rel $K$, 
it has a lift $\tilde{\alpha}$ to $\overline{\mathbb D}$
such that $\tilde{\alpha}_{|S_{\infty}} \equiv \mathrm{Id}$ (cf. \cite{Cantwell-Conlon:lifts}).
We consider lifts $\tilde{g}_j$ of $g_j$ and $\tilde{h}_j$ of $h_j$ such that they commute with $T$
for $1 \leq j \leq m$. We define $\tilde{\phi}$ as in Equation (\ref{equ:lift_p}).

We claim that the group
$\tilde{H}:=\langle \tilde{\phi}, \tilde{g}_{1}, \tilde{h}_{1}, \hdots ,\tilde{g}_{m},\tilde{h}_{m},T \rangle$ 
is nilpotent. Indeed, ${\mathcal C}^{(q)} (\tilde{H})$ is a subgroup of the centralizer of $T$
in the group of deck transformations and hence it is contained in $\langle T \rangle$.
We deduce  ${\mathcal C}^{(q+1)} (\tilde{H}) = \{ \mathrm{Id} \}$ since $T$ belongs to the center
of $\tilde{H}$. Therefore,  the translation number
defines a morphism from $\tilde{H}$ to ${\mathbb R}$ for $\partial_1$ and $\partial_2$.  
Since the translation numbers of commutators vanish and 
$\tilde{\alpha}_{|\partial \tilde{B}} \equiv  \mathrm{Id}$, it follows that  
the translation numbers of $\tilde{\phi}$ at $\partial_1$ and $\partial_2$
are both zero.   It implies $\rho (\tilde{f}, \partial_1)=\rho (\tilde{f}, \partial_2) =0$
for some lift $\tilde{f}$ of $f$ that commutes with $T$ by Remark \ref{rem:partial_rot}. 
This contradicts Lemma \ref{lem:trans} since
$f_{|A_1}$ is a non-trivial Dehn twist and thus 
$\rho (\tilde{f}, \tilde{\gamma}^{1}) \neq \rho (\tilde{f}, \tilde{\gamma}^{2})$.
%
%
%
\section{Thurston decomposition} 
\label{sec:thurston}
We introduce the configuration that will be used to show 
Theorems  \ref{teo:maina},   \ref{teo:mainb}, \ref{teo:mainc} and \ref{teo:baux}. 
\begin{defi} 
\label{def:frame}
 Suppose that ${\mathcal G}$ is a 
finitely generated nilpotent subgroup 
${\mathcal G}$ of $\mathrm{Diff}_{0}^{1}({\mathcal S})$, where ${\mathcal S}$ is an 
orientable connected finite type surface (not 
necessarily compact) of negative Euler
characteristic satisfying the following properties:
\begin{itemize}
\item ${\mathcal H}$ is a normal subgroup of ${\mathcal G}$ such that
${\mathcal G}/{\mathcal H}$ is abelian.
\item There exist $e \in {\mathbb N}$ and   $f_{1}$, $\hdots$, $f_{q} \in {\mathcal G}$ such that
$\langle {\mathcal H}, f_{i_1}, \hdots, f_{i_e} \rangle$ is a finite index
subgroup of ${\mathcal G}$ for any choice of $1 \leq i_1 < \hdots < i_e \leq q$.
\item Every element of ${\mathcal H}$ is isotopic to the identity map
rel $\mathrm{Fix}_{c}({\mathcal H})$.
\end{itemize}
This is the {\it framework} associated to ${\mathcal G}$, ${\mathcal H}$, $e$ and 
$f_1, \hdots, f_q$. 
\end{defi}
We will make ${\mathcal S}=S$, ${\mathcal G}= G$ and 
${\mathcal H}= G'$ for the proof of Theorems \ref{teo:maina}, \ref{teo:mainb} and \ref{teo:baux}
whereas we use a different setting for the proof of Theorem \ref{teo:mainc}. 

In this section we define a Thurston decomposition of ${\mathcal G}$ 
relative to a compact ${\mathcal G}$-invariant set ${\mathcal K}$.  
It is complicated to provide normal forms since 
${\mathcal S} \setminus {\mathcal K}$ is not in general a finite type surface. 
In order to make up for this problem, we will require extra properties for a fixed a priori family
$\langle {\mathcal H}, f_1 \rangle, \hdots, \langle {\mathcal H}, f_q \rangle$. 
of subgroups of ${\mathcal G}$ that help us to obtain global fixed points. 
Mostly, we generalize the approach of Franks, Handle and Parwani in \cite{FHP-g}
for the abelian case to the nilpotent setting.

We would like to consider the isotopy classes of elements of $G$ relative to the invariant set 
$Y=\cup_{j=1}^{q} \mathrm{Fix}_{c}  \langle {\mathcal H},f_{j} \rangle$
 but this set is not easy to 
handle since $Y \cap \mathrm{Fix} (f_j)$ is not open and closed in $Y$. We will replace 
$Y$ with a $G$-invariant compact subset $K$ of $Y$ such that 
$K \cap \mathrm{Fix} (f_j)$ is open and closed in $K$ for any $1 \leq j \leq q$.
Moreover, in a certain sense that we explain below, $K$ contains representatives of 
the $f_j$-Nielsen classes of elements of $Y$ for $1 \leq j \leq q$.

\begin{defi}
\label{def:homotopy}
We say that paths $\gamma:[0,1] \to {\mathcal S}$ and $\beta:[0,1] \to {\mathcal S}$
are homotopic rel ${\mathcal K}$ if there exists an 
homotopy  $I:[0,1] \times [0,1] \to {\mathcal S}$ such that
the following properties hold:
\begin{itemize}
\item $I(0,s)=\gamma(0)$ and $I(1,s)=\gamma(1)$ for any $s \in [0,1]$;
\item  $I(t,0)= \gamma(t)$ and $I(t,1)= \beta (t)$ for any $s \in [0,1]$;
\item If $I(t_0, s_0) \in {\mathcal K}$ then $I(t_0,s)=I(t_0, s_0)$ for any $s \in [0,1]$.
\end{itemize}
\end{defi}
\begin{defi}
\label{def:Nielsen_class}
Let $f \in \mathrm{Homeo} ({\mathcal S})$ and ${\mathcal K}$ a $f$-invariant closed set.
We say that  $z,w \in \mathrm{Fix}(f)$ are in the same $f$-Nielsen class rel ${\mathcal K}$ if 
there exists a path $\gamma:[0,1] \to {\mathcal S}$ such that 
$\gamma (0)=z$, $\gamma (1)=w$ and $\gamma$ is homotopic to $f \circ \gamma$ rel 
${\mathcal K}$.
\end{defi}
\begin{rem}
\label{rem:Nielsen_equ}
A $f_{|{\mathcal S} \setminus {\mathcal K}}$-Nielsen class is contained in a $f$-Nielsen class 
rel ${\mathcal K}$. If the latter class does not contain points of ${\mathcal K}$ then 
both Nielsen classes coincide. In particular, 
if ${\mathcal K} \cap \mathrm{Fix}(f)= \emptyset$ then $z, w \in  \mathrm{Fix}(f)$ 
are in the same $f$-Nielsen class rel ${\mathcal K}$ if and only if 
$z$ and $w$ are in the same $f_{|{\mathcal S} \setminus {\mathcal K}}$-Nielsen class.

Let $z \in {\mathrm Fix} (f) \setminus {\mathcal K}$. 
Let ${\mathcal N}$ and ${\mathcal N}'$ the $f$-Nielsen class rel ${\mathcal K}$ and the 
$f_{|{\mathcal S} \setminus {\mathcal K}}$-Nielsen class of $z$ respectively.
In general, the equality ${\mathcal N} \setminus {\mathcal K} = {\mathcal N}'$ does not hold.
This can be checked out by considering $f = \mathrm{Id}$ where 
${\mathcal N}={\mathcal S}$ and ${\mathcal N}'$ is the connected component of 
${\mathcal S} \setminus {\mathcal K}$ containing $z$. Anyway, for many of the situations
in this paper both concepts coincide but we will make precise which one we are
considering at any moment. 
\end{rem}
\begin{rem}
Since ${\mathcal H}$ is a normal subgroup of ${\mathcal G}$ and 
${\mathcal G}/ {\mathcal H}$ is abelian, any subgroup of ${\mathcal G}$ containing 
${\mathcal H}$ is normal in ${\mathcal G}$.
In particular $\langle {\mathcal H}, f_j \rangle$ is a normal subgroup of ${\mathcal G}$
and $\mathrm{Fix}_{c} \langle {\mathcal H}, f_j \rangle$ is a ${\mathcal G}$-invariant set
for any $1 \leq j \leq q$ by Remark \ref{rem:inv_nor0}.
\end{rem}
\begin{defi}
\label{def:good_excellent}
Denote $Y=  \cup_{j=1}^{q} \mathrm{Fix}_{c}  \langle {\mathcal H},f_{j} \rangle$.
We say that a compact subset ${\mathcal K}$ of  ${\mathcal S}$
is an {\it excellent} set for the group ${\mathcal G}$ and the subgroups 
 $\langle {\mathcal H},f_{1} \rangle$, $\hdots$, $\langle {\mathcal H},f_{q} \rangle$ if
\begin{itemize}
\item ${\mathcal K}$ is a ${\mathcal G}$-invariant subset of $Y$
and ${\mathcal K} \cap  \mathrm{Fix}_{c} (f_{j}) \neq \emptyset \ \forall 1 \leq j \leq q$.
\item $\sum_{j=1}^{q} {\bf 1}_{\mathrm{Fix}(f_{j})}$ is locally constant in ${\mathcal K}$
where ${\bf 1}_{\mathrm{Fix}(f_{j})}$ is the characteristic function of $\mathrm{Fix}(f_{j})$.
\item If $z \in Y \cap \mathrm{Fix}  \langle f_{j_{1}}, \hdots, f_{j_{l}} \rangle$
there exists $w \in K \cap \mathrm{Fix}  \langle f_{j_{1}}, \hdots, f_{j_{l}} \rangle$ such that
$z$ and $w$ are in the same $g$-Nielsen class rel ${\mathcal K}$ for any
$g \in \{ f_{j_{1}}, \hdots, f_{j_{l}} \}$.
\end{itemize}
We say that ${\mathcal K}$ is a {\it good} (compact) set for
the group ${\mathcal G}$ and the subgroups  
$\langle {\mathcal H},f_{1} \rangle$, $\hdots$, $\langle {\mathcal H},f_{q} \rangle$ if
\begin{itemize}
\item ${\mathcal K} \subset \mathrm{Fix}_{c} ({\mathcal H})$ is ${\mathcal G}$-invariant  
and ${\mathcal K} \cap  \mathrm{Fix}_{c} (f_{j}) \neq \emptyset \ \forall 1 \leq j \leq q$.
\item $\sum_{j=1}^{q} {\bf 1}_{\mathrm{Fix}(f_{j})}$ is locally constant in ${\mathcal K}$.
\item If $z \in Y \cap \mathrm{Fix}_{c}  \langle f_{j_{1}}, \hdots, f_{j_{l}} \rangle$
there exists $w \in {\mathcal K} \cap \mathrm{Fix}_{c} \langle f_{j_{1}}, \hdots, f_{j_{l}} \rangle$ 
such that $z$ and $w$ are in the same $g$-Nielsen class rel ${\mathcal K}$ for any
$g \in \{ f_{j_{1}}, \hdots, f_{j_{l}} \}$.
\end{itemize}
\end{defi}
\begin{rem}
It is clear that an excellent set is a good set.
\end{rem}
\begin{rem}
\label{rem:ex_ex}
The existence of an excellent set ${\mathcal K}$ is a consequence of Lemma 4.3 of \cite{FHP-g}. 
The original statement is for commutative
groups but the proof also works if all elements of ${\mathcal H}$
are isotopic to the identity rel to $Y$. Essentially, we are working with the commutative group 
${\mathcal G}/{\mathcal H}$. 
\end{rem}
\begin{rem}
\label{rem:from_ex_good}
 Consider an open and closed 
${\mathcal G}$-invariant subset ${\mathcal K}'$ of a good set ${\mathcal K}$ 
for the subgroups 
$\langle {\mathcal H},f_{1} \rangle$, $\hdots$, $\langle {\mathcal H},f_{q} \rangle$
such that 
\[ {\mathcal K}' \cap \mathrm{Fix}_{c} \langle {\mathcal H}, f_j \rangle =
{\mathcal K} \cap \mathrm{Fix}_{c} \langle {\mathcal H}, f_j \rangle  \]
for some $1 \leq q' \leq q$ and any $1 \leq j \leq q'$. 
Then ${\mathcal K}'$ is a good compact set for the group 
${\mathcal G}$ and the subgroups 
$\langle {\mathcal H},f_{1} \rangle$, $\hdots$, $\langle {\mathcal H},f_{q'} \rangle$.
We will use this remark later on and this is the reason justifying the introduction of good sets
since we can not guarantee that ${\mathcal K}'$ is excellent if ${\mathcal K}$ is.
\end{rem}
 The next step is considering a Thurston decomposition for the normal subgroups
$\langle {\mathcal H},f_{1} \rangle$, $\hdots$, $\langle {\mathcal H},f_{q} \rangle$
of ${\mathcal G}$ relative to a good  compact set ${\mathcal K}$.
The situation is analogous to the abelian one.
\begin{defi}
\label{def:decomp}
Let ${\mathcal K}$ be a {\it good} compact set for
the group ${\mathcal G}$ and the subgroups  
$\langle {\mathcal H},f_{1} \rangle$, $\hdots$, $\langle {\mathcal H},f_{q} \rangle$.
We say that a finite set ${\mathcal R}$ of curves
is a {\it reducing set} rel ${\mathcal K}$ if ${\mathcal R}$ consists of
disjoint pairwise non-homotopic simple closed
curves ${\mathcal R}$ in ${\mathcal S} \setminus {\mathcal K}$ such that
\begin{itemize}
\item ${\mathcal R}$ is ${\mathcal G}$-invariant modulo isotopy rel ${\mathcal K}$.
\item No connected component of $S \setminus {\mathcal R}$ is a disk
$D$ with $\sharp (D \cap {\mathcal K}) \leq 1$.
\item Given a connected component  $E$ of $S \setminus {\mathcal R}$
with $\sharp (E \cap {\mathcal K})= \infty$  
and  $E \cap {\mathcal K} \cap \mathrm{Fix}(f_{j}) \neq \emptyset$ for some $1 \leq j \leq q$, 
there exists a homeomorphism $\phi_{j}:S \to S$
such that $(\phi_{j})_{|E} \equiv Id$ and
$f_{j}$ is isotopic to $\phi_{j}$ rel ${\mathcal K}$.
\end{itemize} 
\end{defi}
%
%
%
%

\begin{rem}
\label{rem:red}
Consider the framework introduced in this section. 
The existence of a set of reducing curves 
for ${\mathcal G}$, the subgroups 
$\langle {\mathcal H},f_{1} \rangle$, $\hdots$, $\langle {\mathcal H},f_{q} \rangle$ and 
an excellent compact set ${\mathcal K}$ is a consequence of  \cite[Theorem 4.4.]{FHP-g}  
if ${\mathcal H}$ is trivial (and hence ${\mathcal G}$ is abelian).  A key part of
the proof is  by contradiction, and they use the next lemma for the case where
${\mathcal H} = \{ \mathrm{Id} \}$   
\cite[Lemma 3.10]{FHP-g}
to obtain points in $Y \cap \mathrm{Fix} \langle f_1, \hdots, f_l \rangle$ that are not 
$f_l$-Nielsen equivalent rel ${\mathcal K}$ to any point of ${\mathcal K}$, contradicting
the choice of ${\mathcal K}$. In our case, the points of ${\mathcal K}$ have to be 
contained in $\mathrm{Fix}_{c} ({\mathcal H})$ and hence we need to replace 
 \cite[Lemma 3.10]{FHP-g} with next lemma to obtain points in 
 $Y \cap \mathrm{Fix} \langle {\mathcal H}, f_1, \hdots, f_l \rangle$ that are not 
 $f_l$-Nielsen equivalent rel ${\mathcal K}$ to any point of ${\mathcal K}$, 
 providing the contradiction. 
 The remainder of the proof Theorem 4.4 in \cite{FHP-g} 
 works well if 
 ${\mathcal G}$ is abelian up to isotopy rel ${\mathcal K}$. 
 This is our case since 
  the elements of ${\mathcal H}$ are isotopic to $\mathrm{Id}$ rel 
$Y$ by hypothesis and ${\mathcal G}/{\mathcal H}$ is abelian.
\end{rem}
We denote by $\pi: \tilde{\mathcal S} \to {\mathcal S}$ the universal covering map. 
\begin{lem}
\label{lem:guarantees_thurston}
Let ${\mathcal G} = \langle {\mathcal H},f_0, f_1, \hdots, f_l \rangle$ 
be a nilpotent subgroup of $\mathrm{Diff}_{0}^{1} ({\mathcal S})$ 
where ${\mathcal S}$ is an orientable
finitely punctured compact surface with $\chi ({\mathcal S}) <0$, $f_0 = \mathrm{Id}$, $l \geq 1$ 
and ${\mathcal H}$ is a normal subgroup of ${\mathcal G}$ such that ${\mathcal G}/{\mathcal H}$
is abelian. 
Let ${\mathcal K}$ be a ${\mathcal G}$-invariant compact set such that 
every element $h$ of ${\mathcal H}$ is isotopic to $\mathrm{Id}$ rel ${\mathcal K}$. 
Assume that there exist a closed disc $D$ in ${\mathcal S}$
such that $\partial D \cap {\mathcal K} = \emptyset$,  
$\emptyset \neq {\mathcal K} \cap D  \subset \cap_{j=0}^{l-1} \mathrm{Fix}(f_j)$, 
${\mathcal K} \cap D \cap \mathrm{Fix}(f_l)=\emptyset$ and 
$f_{j} (\partial D)$ is homotopic to $\partial D$ rel ${\mathcal K}$ for any $1 \leq j \leq l$. 
Consider a lift $\tilde{f}_j$ of $f_j$ to $\tilde{\mathcal S}$ such that 
$\tilde{f}_j (\pi^{-1}({\mathcal K}) \cap \tilde{D}_0) = \pi^{-1}({\mathcal K}) \cap \tilde{D}_0$ 
for some connected component  $\tilde{D}_0$ of $\pi^{-1}(D)$ and $0 \leq j \leq l$. 
Then there exists a global fixed point $\tilde{z}$ of the group 
$\langle \tilde{\mathcal H}, \tilde{f}_1, \hdots, \tilde{f}_l \rangle$,
where $\tilde{\mathcal H}$ consists of the identity lifts of elements of ${\mathcal H}$,
such that $z:= \pi (\tilde{z})$ is not $f_l$-Nielsen equivalent rel ${\mathcal K}$ 
to any element of ${\mathcal K}$. 
In particular if $\tilde{f}_j$ is the identity lift of $f_j$ for some $1 \leq j \leq l$, $z$ belongs to 
$\mathrm{Fix}_{c} \langle {\mathcal H}, f_j \rangle$.
\end{lem}
\begin{proof}
Since ${\mathcal H} \lhd  {\mathcal G}$, the group 
$\tilde{\mathcal H}$ is normal in 
$\tilde{\mathcal G}:=\langle \tilde{\mathcal H}, \tilde{f}_1, \hdots,  \tilde{f}_l \rangle$.
Therefore, ${\mathcal C}^{(q)} ({\mathcal G}) = \{ \mathrm{Id} \}$ implies 
${\mathcal C}^{(q)} (\tilde{\mathcal G}) = \{ \mathrm{Id}  \}$ and so $\tilde{\mathcal G}$ is nilpotent.
Moreover, the elements of  $\tilde{\mathcal H}$ are isotopic to $\mathrm{Id}$ rel 
$\pi^{-1} ({\mathcal K})$.

We denote $\tilde{\mathcal K}_0= \pi^{-1}({\mathcal K}) \cap \tilde{D}_0$.
Since 
\[ \tilde{\mathcal K}_0 \subset \mathrm{Fix} \langle \tilde{\mathcal H}, \tilde{f}_1, \hdots,  \tilde{f}_{l-1} \rangle \setminus \mathrm{Fix} (\tilde{f}_l) \]
we can apply Proposition \ref{pro:clp3} to 
$H =  \langle \tilde{\mathcal H}, \tilde{f}_1, \hdots,  \tilde{f}_{l-1} \rangle$, $f=\tilde{f}_l$
to obtain $\tilde{z} \in \mathrm{Fix} (\tilde{\mathcal G}) \cap P(\tilde{K}_0, \tilde{f}_l)$.

We claim that $z$ is not $f_l$-Nielsen equivalent rel ${\mathcal K}$ to $w \in {\mathcal K}$. 
 Otherwise, there exists 
 $\tilde{w} \in \pi^{-1}(w)$ that is $\tilde{f}_l$-Nielsen equivalent to $\tilde{z}$ rel 
 $\pi^{-1} ({\mathcal K})$. 
 Since $\tilde{\mathcal K}_0 \cap \mathrm{Fix}(\tilde{f}_l)= \emptyset$, it follows that
 $\tilde{w} \in P(\tilde{K}_0, \tilde{f}_l)$.
 Notice that 
 \[ \mathrm{Fix} (\tilde{f}_l) \cap \pi^{-1} ({\mathcal K} )= 
 \mathrm{Fix} (\tilde{f}_l)  \cap (\pi^{-1} ({\mathcal K} )\setminus \tilde{\mathcal K}_0) \]
 and thus all points in $ \mathrm{Fix} (\tilde{f}_l) \cap  \pi^{-1} ({\mathcal K} )$  
 are peripheral to $\infty$ rel $\tilde{\mathcal K}_0$ for $\tilde{f}_l$
contradicting  $\tilde{w} \in P(\tilde{K}_0, \tilde{f}_l)$.
 

 It is clear that $z \in \mathrm{Fix}_{c} ({\mathcal H})$. Moreover, if $\tilde{f}_j$ is the identity
 lift then $z \in \mathrm{Fix}_{c} \langle {\mathcal H}, f_j \rangle$. 
\end{proof}
We can refine the previous decomposition in the connected components $M$
of ${\mathcal S} \setminus {\mathcal R}$ where $M \cap {\mathcal K}$ is finite.
\begin{defi}
Let ${\mathcal R}$ be a reducing set of curves for a group ${\mathcal G}$, 
subgroups $\langle {\mathcal H},f_{1} \rangle$, $\hdots$, $\langle {\mathcal H},f_{q} \rangle$
and a good compact set ${\mathcal K}$. 
Consider a connected component $M$ of
${\mathcal S} \setminus {\mathcal R}$.
We define ${\mathcal G}_{M}$ as the subgroup of ${\mathcal G}$ of diffeomorphisms $f$
such that $f(M)$ is isotopic to $M$ rel ${\mathcal K}$.
In particular ${\mathcal G}_{M}$ contains ${\mathcal H}$.
Given $f \in {\mathcal G}_{M}$ there exists a diffeomorphism $\theta_{f}$
such that $\theta_{f}(M)=M$ and $f$ is isotopic to $\theta_{f}$ rel ${\mathcal K}$.
Consider the group 
$\hat{\mathcal G}_{M} = \langle (\theta_{f})_{|M}: f \in {\mathcal G}_{M} \rangle$.
Given $f \in {\mathcal G}_M$, the
element induced by $\theta_{f}$ in the mapping class group of the marked surface
$ (M, M \cap {\mathcal K})$ depends only on $f$ since either 
$\sharp (M \cap {\mathcal K})= \infty$ or $\chi (M \setminus {\mathcal K}) < 0$. 
Thus, we can define the group $\overline{\mathcal G}_M$ induced by 
$\hat{\mathcal G}_{M}$ in the mapping class group of $(M, M \cap {\mathcal K})$.
\end{defi}
\begin{defi}
Let ${\mathcal K}$ be a  good compact set for
the group ${\mathcal G}$ and the subgroups  
$\langle {\mathcal H},f_{1} \rangle$, $\hdots$, $\langle {\mathcal H},f_{q} \rangle$.
We say that a finite set ${\mathcal R}$ of simple closed curves in 
${\mathcal S} \setminus {\mathcal K}$ is a  {\it Thurston-reducing set} rel ${\mathcal K}$ if
${\mathcal R}$ is a reducing set rel ${\mathcal K}$ and given any connected 
component $M$ of ${\mathcal S} \setminus {\mathcal R}$ such that 
$\sharp (M \cap {\mathcal K}) < \infty$, the group $\overline{\mathcal G}_{M}$ 
consists of irreducible (periodic or pseudo-Anosov) elements.
\end{defi}
\begin{pro}
\label{pro:Tdecomp}
Let ${\mathcal K}$ be an excellent compact set for
the group ${\mathcal G}$ and the subgroups  
$\langle {\mathcal H},f_{1} \rangle$, $\hdots$, $\langle {\mathcal H},f_{q} \rangle$.
There exists a Thurston-reducing set 
rel ${\mathcal K}$. 
\end{pro} 
\begin{proof}
There exists a reducing set ${\mathcal R}$ of curves rel ${\mathcal K}$ by Remark \ref{rem:red}.
Consider a connected component $M$ of ${\mathcal S} \setminus {\mathcal R}$ such that 
$M \cap {\mathcal K}$ is finite.
Note that  the group $\overline{\mathcal G}_M$ 
is abelian since every element of ${\mathcal H}$ is isotopic to 
$\mathrm{Id}$ rel ${\mathcal K}$ and ${\mathcal G}/{\mathcal H}$ is abelian. 
It is known that there exists a Thurston-reducing set for an abelian subgroup of the mapping
class group of a finite type surface \cite[Lemma 2.2]{Handel-top}.
Thus, it suffices to add to ${\mathcal R}$ a set of Thurston-reducing curves of 
$\overline{\mathcal G}_M$ for any connected component $M$ of 
${\mathcal S} \setminus {\mathcal R}$ such that $M \cap {\mathcal K}$ is finite.
\end{proof}
At some points it will be convenient to consider fewer subgroups in our framework because 
that allows us to recover arguments of the case where ${\mathcal K}$ is finite.
We still obtain a Thurston decomposition.
\begin{rem}
\label{rem:tgood}
Let ${\mathcal K}$ be a  good compact set for
the group ${\mathcal G}$ and the subgroups  
$\langle {\mathcal H},f_{1} \rangle$, $\hdots$, $\langle {\mathcal H},f_{q} \rangle$.
Assume that there exists a 
reducing set ${\mathcal R}$ of curves rel ${\mathcal K}$. 
Consider a subset $M'$ of ${\mathcal S} \setminus {\mathcal R}$
that is a union of connected components of ${\mathcal S} \setminus {\mathcal R}$.
Assume that $M'$ is ${\mathcal G}$-invariant modulo isotopy rel ${\mathcal K}$. 
Let $e \leq q' \leq q$ such that 
\[ {\mathcal K}'' :=
M' \cap {\mathcal K} \cap (\cup_{j=1}^{q'}  \mathrm{Fix}_{c}  \langle {\mathcal H},f_{j} \rangle) \]
is a finite set. The set ${\mathcal K}' := ({\mathcal K} \setminus M') \cup {\mathcal K}''$
is a good compact set for the group ${\mathcal G}$ and 
the subgroups  
$\langle {\mathcal H},f_{1} \rangle$, $\hdots$, $\langle {\mathcal H},f_{q'} \rangle$
(Remark \ref{rem:from_ex_good}).

Assume that $\chi (M \setminus {\mathcal K}'' ) < 0$ for 
every connected component $M$ of $M'$.
This guarantees that ${\mathcal R}$ still satisfies 
the properties in Definition \ref{def:decomp}. More precisely, no connected component
of ${\mathcal S} \setminus {\mathcal R}$ is an annulus that contains no points of 
${\mathcal K}'$ or a disc containing at most one point of ${\mathcal K}'$.
In this case, we obtain a Thurston-reducing set of curves rel ${\mathcal K}'$ 
for the group ${\mathcal G}$ and the subgroups  
$\langle {\mathcal H},f_{1} \rangle$, $\hdots$, $\langle {\mathcal H},f_{q'} \rangle$
by adding the Thurston-reducing curves of the group induced by ${\mathcal G}$ in the mapping
class group of the finite type surface $M' \setminus {\mathcal K}''$.
\end{rem}
\begin{rem}
\label{rem:decomp}
Let ${\mathcal R}$ be a Thurston-reducing set of curves rel ${\mathcal K}$ for the group 
${\mathcal G}$ and the subgroups 
$\langle {\mathcal H},f_{1} \rangle$, $\hdots$, $\langle {\mathcal H},f_{q} \rangle$
where ${\mathcal K}$ is a good compact set. 
As a consequence of Definition \ref{def:decomp}
there exists a finite set ${\mathcal A}$
of pairwise disjoint non-homotopic closed annuli in ${\mathcal S} \setminus {\mathcal K}$ such that 
\begin{itemize}
\item 
every annulus of ${\mathcal A}$ contains exactly a reducing curve in ${\mathcal R}$
as a core curve of the annulus;
\item any $g \in {\mathcal G}$ is isotopic rel ${\mathcal K}$ 
to some $\theta_{g} \in \mathrm{Homeo}_{+} ({\mathcal S})$
such that $\theta_{g} (\partial {\mathcal A}) = \partial {\mathcal A}$;
\item Given a connected component $M$ of ${\mathcal S} \setminus {\mathcal A}$
such that $M \cap {\mathcal K}$ is an infinite set, we have 
$(\theta_{f_j})_{|M} \equiv \mathrm{Id}$ whenever 
$M \cap {\mathcal K} \cap \mathrm{Fix}(f_j) \neq \emptyset$.
 Given a connected component $M$ of ${\mathcal S} \setminus {\mathcal A}$
such that $M \cap {\mathcal K}$ is finite and $g \in {\mathcal G}$ with $\theta_{g} (M)=M$,
we have that $(\theta_{g})_{|M}$ is periodic or pseudo-Anosov rel $M \cap {\mathcal K}$.
\end{itemize}
\end{rem}
\section{Annular dynamics}
\label{sec:annular}
Roughly speaking, in this paper there are two ways of finding global fixed points 
of groups of diffeomorphisms depending
on whether the Euler characteristic of a relevant  subsurface is negative or not. 
In this section, we consider the case where such a subsurface is an annulus.

Let us introduce the setting of this section.
Let $A$ be an open annulus and $\pi: \tilde{A} \to A$ its universal covering map.
We denote by $0$ and $\infty$ the ends of $A$.
Let $K_{0}$ be a compact subset of $A$.
Consider $g \in \mathrm{Diff}_{+}^{1}(A)$
and a lift $\tilde{g} \in \mathrm{Diff}_{+}^{1}(\tilde{A})$ such that
$g(K_{0})=K_{0}$, $\mathrm{Fix}(\tilde{g}) \cap \pi^{-1}(K_{0}) = \emptyset$ and $g$ fixes the ends of $A$.
\begin{defi}
\label{def:p0i}
Let $z \in \mathrm{Fix}(g)$. We say that $z$ does not belong to
$P_{0}(K_{0},g)$ if there exists a path
$\gamma:[0,1) \to A \setminus K_{0}$ such that
$\gamma(0)=z$, $\lim_{t \to 1} \gamma(t)=0$ and
$\gamma$ is properly homotopic to $g \circ \gamma$ rel $0$.
The definition of  $P_{\infty}(K_{0},g)$ is analogous.
We define $P_{0,\infty}(K_{0},g) = P_{0}(K_{0},g) \cap P_{\infty}(K_{0},g)$.
\end{defi}
We want to find compactly covered $g$-Nielsen classes rel $K_0$ in 
$\pi (\mathrm{Fix}(\tilde{g}))$.  Next, and as a temporary substitute we consider classes
in $P_{0,\infty}(K_{0},g)$.
\begin{lem}
\label{lem:p0i}
Suppose that $K_{0}$ is   the support of a $g$-invariant ergodic measure
$\mu$ and there exists a lift $\tilde{g}$ of $g$ to $\tilde{A}$ 
with $\rho_{\mu}(\tilde{g})=0$. 
Furthermore, assume that $\mathrm{Fix}(\tilde{g}) \cap \pi^{-1}(K_{0}) = \emptyset$. 
Then the intersection of
$P_{0,\infty}(K_{0},g)$ and the projection of $\mathrm{Fix}(\tilde{g})$
is a non-empty set.
\end{lem}
\begin{proof}
We denote by $\pi: \tilde{A} \to A$ the universal covering map.
Atkinson's lemma  \cite{Atkinson}
implies that the set of recurrent points in $\pi^{-1}({K}_{0})$ is of total measure.
Let $\tilde{z} \in \pi^{-1}({K}_{0})$ be a recurrent point.
Given an arbitrary small neighborhood $V$ of $z$ in $A$
there exists $h \in \mathrm{Homeo}_{+}(A)$ whose support is contained in $V$
and such that $\tilde{g} \circ \tilde{h}$ has a finite orbit ${\mathcal O}$ contained in
$\pi^{-1}({K}_{0})$, where $\tilde{h}$ is the lift of $h$ such that 
$\tilde{h}_{|\tilde{A} \setminus  \pi^{-1}(V)} \equiv \mathrm{Id}$.
We define $h = Id$ if the $\tilde{g}$-orbit of $\tilde{z}$ is finite.
The sets $\mathrm{Fix}(\tilde{g})$ and $\mathrm{Fix}(\tilde{g} \circ \tilde{h})$ coincide
if $V$ is a small neighborhood of $z$.
The set $P({\mathcal O}, \tilde{g} \circ \tilde{h})$ contains a point $\tilde{w}$
(Proposition 5.3 of \cite{FHPs}, see also \cite{Gamba}).
We claim that if $h$ is close enough to the identity in the $C^{0}$ topology then
the point $\pi (\tilde{w})$ also belongs  to $P_{0,\infty}(K_{0},g)$.

It is clear that there exists a compact subannulus $A_{K}$ of $A$
containing $K_{0}$ such that
$(\mathrm{Fix}(g) \setminus A_{K}) \cap \pi(P({\mathcal O}, \tilde{g} \circ \tilde{h}))$ is empty
if $V$ is a small neighborhood of $z$.
Consider a point $w_{0} \in \mathrm{Fix}(g) \setminus P_{0}(K_{0},g)$.
Let $\gamma$ be the path in Definition \ref{def:p0i}. Let $J$ be the homotopy
between $\gamma$ and $g \circ \gamma$.
The distance between the image of $J$ and $K_{0}$ is positive, say greater than
some $\epsilon>0$.
Hence if $V$ is a small neighborhood of $z$ the point ${w}_{0}$
does not belong to $P_{0}({\mathcal O},g \circ h)$.
Therefore no lift $\tilde{w}_{0}$ of $w_{0}$ belongs to
$P({\mathcal O}, \tilde{g} \circ \tilde{h})$.
Moreover $\tilde{w}$ does
not belong to $P({\mathcal O}, \tilde{g} \circ \tilde{h})$
for any lift $\tilde{w}$ of a point $w \in \mathrm{Fix} (g)$ in the
neighborhood of $w_{0}$.
An analogous property holds true for points in
$\mathrm{Fix}(g) \setminus P_{\infty}(K_{0},g)$.

Since $(\pi(\mathrm{Fix}(\tilde{g})) \cap A_{K}) \setminus P_{0,\infty}(K_{0},g)$ is compact then
\[ (\pi(\mathrm{Fix}(\tilde{g})) \setminus P_{0,\infty}(K_{0},g)) \cap 
\pi (P({\mathcal O}, \tilde{g} \circ \tilde{h})) = \emptyset \]
if $V$ is a small neighborhood of $z$.
Thus we obtain that $\pi(P({\mathcal O}, \tilde{g} \circ \tilde{h}))$ is a non-empty
set contained in
$P_{0,\infty}(K_{0},g) \cap \pi(\mathrm{Fix}(\tilde{g}))$.
\end{proof}
We want to find and localize global fixed points of a nilpotent subgroup 
$G = \langle H, f \rangle$ of $\mathrm{Diff}_{+}^{1} (S_0)$, where $H$ is a normal subgroup of 
$G$ consisting of isotopic to the identity elements. 
We will use Proposition \ref{pro:clr} to obtain
compactly covered $f$-Nielsen classes and Proposition \ref{pro:clp2} to show that
such Nielsen classes contain elements of $\mathrm{Fix}(G)$.
The localization of such classes is the subject of Proposition \ref{pro:ann}.

Let us introduce the setting of Propositions \ref{pro:clr}  and \ref{pro:ann}. 
\begin{defi}
\label{def:rho}
Let $A$ be an open annulus contained in a surface $S_{0}$.
We denote by $\pi: \tilde{S}_{0} \to S_{0}$ the universal covering map.
Suppose that the core curve of the annulus is not null-homotopic in $S_{0}$;
in this case the universal covering $\tilde{A}$ of $A$ is contained in
$\tilde{S}_{0}$.
Let $K_{0}$ be a non-empty compact subset of $A$.
Consider $f \in \mathrm{Diff}_{+}^{1}(S_{0})$ such that
$f$ is isotopic to some $\theta \in \mathrm{Homeo}_{+} (S_0)$ rel $K_0$ such that 
$f(K_{0})=K_{0}$, $\theta (A)=A$ and $\theta$ preserves the ends of $A$. 
Consider a lift $\tilde{f}$ of $f$ to $\tilde{S}_{0}$
such that $\tilde{f}(\tilde{A})$ is isotopic to $\tilde{A}$
rel $\pi^{-1}(K_{0})$.
Let $\mu$ be a $f$-invariant measure with support contained in $K_{0}$.
Up to fix a generator of $H_{1} (A, {\mathbb Z})$, we can define
the rotation number $\rho_{\mu}(\tilde{f})$ of
$\tilde{f}$ with respect to $\mu$ in $\tilde{A}$.

The rotation number is well-defined since $f$ preserves $A$ modulo isotopy
relative to the support of $\mu$.
\end{defi}
\begin{rem}
\label{rem:rho}
Suppose $f \in \mathrm{Diff}_{0}^{1}(S_{0})$, $\chi (S_0) <0$ and $\tilde{f}$ is the identity lift of 
$f$ to $\tilde{S}_0$.  Then $\rho_{\mu}(\tilde{f})$ does not depend on the choice of 
$\tilde{A}$ since it commutes with deck transformations.
\end{rem} 
\begin{pro}
\label{pro:clr}
Suppose that $K_{0}$ is the support of a $f$-invariant ergodic measure
$\mu$ such that $\rho_{\mu}(\tilde{f}) = 0$ and 
$\mathrm{Fix}(\tilde{f}) \cap (\pi^{-1}(K_{0}) \cap \tilde{A}) = \emptyset$.
Then there exists a lift $f^{\natural}$ such that
$\mathrm{Fix}(f^{\natural})$ is a non-empty compact set whose projection to $S_0$
is contained in the $f$-Nielsen class defined by $\tilde{f}$.
\end{pro} 
\begin{proof}
Let $\pi: \tilde{S}_{0} \to {S}_{0}$ be the universal covering map.
Note that since the core curve of the annulus $A$ is not null-homotopic, $\tilde{S}_0$ is 
not a sphere and hence  $\tilde{S}_0$ is homeomorphic to ${\mathbb R}^{2}$.
Consider the primitive deck transformation
$T$ such that $T(\tilde{A})=\tilde{A}$. We define $\hat{S}_{0}$ as the quotient
of $\tilde{S}_{0}$ by the action of the group $\langle T \rangle$.
Let $\hat{\pi}: \hat{S}_{0} \to S_{0}$  and
$\varpi: \tilde{S}_{0} \to \hat{S}_{0}$ be the natural maps.
The annulus $\hat{A}:=\varpi(\tilde{A})$
satisfies that
$\hat{\pi}_{|\hat{A}}: \hat{A} \to A$ is a homeomorphism.
We denote $\hat{K}_{0} = \hat{\pi}^{-1}(K_{0}) \cap \hat{A}$.
Consider the unique lift $\hat{f}$ of $f$ to $\hat{S}_{0}$ such that
$\tilde{f}$ is a lift of $\hat{f}$. We obtain
$\hat{f}(\hat{K}_{0})  = \hat{K}_{0}$. Let $\hat{\mu}$ be the
measure $\hat{\pi}_{|\hat{A}}^{*} \mu$.

Since $\mathrm{Fix}(\tilde{f}) \cap \varpi^{-1}(\hat{K}_{0}) \cap \tilde{A} = \emptyset$
by hypothesis, there exists 
$\hat{z} \in P_{0,\infty}(\hat{K}_{0},\hat{f}) \cap \varpi(\mathrm{Fix}(\tilde{f}))$  by 
Lemma \ref{lem:p0i}.
Let $\hat{U}$ be the connected component of $\hat{S}_{0} \setminus \hat{K}_{0}$
containing $\hat{z}$.
We claim that $\hat{U}$ is not an annulus.
Indeed, since $\hat{K}_{0} = \mathrm{sup} (\hat{\mu})$ and 
$\hat{\mu}$ is ergodic, the set  $\hat{S}_0 \setminus \hat{U}$ has just one connected component. 
If $\hat{U}$ is an inessential annulus then  $\hat{S}_0 \setminus \hat{U}$
has two connected components and we get a contradiction. 
If $\hat{U}$ is an essential annulus such that 
 $\hat{S}_0 \setminus \hat{U}$ has just one connected component then
$\hat{z} \not \in P_{0}(\hat{K}_{0},\hat{f})$ or $\hat{z} \not \in P_{\infty}(\hat{K}_{0},\hat{f})$ and 
again we obtain a contradiction.

We denote by $\pi_{\hat{U}}: \tilde{U} \to \hat{U}$ the universal covering map.
Fix a lift $\tilde{z}$ of $\hat{z}$ in $\tilde{U}$.
We define $f^{\flat}$ as the lift of $\hat{f}$ to $\tilde{U}$
such that $f^{\flat}(\tilde{z})=\tilde{z}$.
If $\mathrm{Fix}(f^{\flat})$ is compact we define $f^{\natural} = f^{\flat}$, 
$D= \tilde{U}$ and $K_{0}' = \hat{K}_0$.
The map $f^{\natural}:D \to D$ is a lift of $f$ by the map
$\hat{\pi} \circ \pi_{\hat{U}}$.

Suppose that $\mathrm{Fix}(f^{\flat})$ is not compact.
There exists
a non-trivial $[\nu] \in \pi_{1}(\hat{U},\hat{z})$ such that $\hat{f} [\nu]=[\nu]$.
Up to replace $\nu$ with an homotopic closed path defining the same class in
 $\pi_{1}(\hat{U},\hat{z})$,
there exists a finite type connected subsurface
$S$ contained in $\hat{U}$, containing $\nu [0,1]$ and
a homeomorphism $\hat{\theta}$ of $\hat{S}_0$ isotopic to $\hat{f}$ rel $\hat{K}_{0}$
such that $\hat{\theta}_{|S}$ has finite order (Lemma 2.12 of \cite{FHPs}).
Moreover no boundary curve $\alpha$ in $\partial S$ bounds a disk contained
in $\hat{S}_{0} \setminus \hat{K}_{0}$.

%
%

Consider the connected components $S^{0}$ and $S^{\infty}$ of
$\hat{S}_{0} \setminus S$ containing the ends $0$ and $\infty$
respectively. If $S^{0} = S^{\infty}$ then we denote by $\gamma$ the unique 
curve in $\partial S \cap \partial S^{0} \cap \partial S^{\infty}$. The curve $\gamma$
is non-essential.

Suppose now $S^{0} \neq S^{\infty}$. 
The curve $\partial S \cap \partial S^{0}$ is essential. Moreover
the compact set $\hat{K}_{0} \cap S^{0}$
is invariant and either has measure
$1$ or $0$. Since $\hat{K}_{0} \setminus S^{0}$ is a compact invariant set and
$\mathrm{supp}(\hat{\mu})=\hat{K}_{0}$ then either
$\hat{K}_{0} \cap S^{0} = \emptyset$ or $\hat{K}_{0} \cap S^{0} = \hat{K}_{0}$.
The latter situation is impossible since it would imply $\hat{z} \not \in P_{\infty}(\hat{K}_{0},\hat{f})$.
Analogously we obtain that $\partial S \cap \partial S^{\infty}$ is an essential curve and 
$\hat{K}_{0} \cap S^{\infty} = \emptyset$.
Since $S$ is connected all the boundary curves other than $\partial S \cap \partial S^{0}$ 
and $\partial S \cap \partial S^{\infty}$ are non-essential.
There are at least one such non-essential boundary curve since otherwise we have 
$\hat{K}_0 = \emptyset$.
Consider the homeomorphism $\theta^{\flat}$ induced by $\hat{\theta}$ in the surface $S^{\flat}$
obtained by contracting the curves in $\partial S$. It has two fixed points corresponding to 
$\partial S \cap \partial S^{0}$ and $\partial S \cap \partial S^{\infty}$.  
Moreover since $\hat{\theta}$ is isotopic to $\hat{f}$ rel $\hat{K}_0$ and the rotation number 
$\rho_{\hat{\mu}} (\hat{f})$ vanishes, it follows that all boundary curves of $S$ correspond to 
fixed points of $\theta^{\flat}$ in $S^{\flat}$. Moreover since $\theta^{\flat}$ is a finite order 
homeomorphism of a sphere $S^{\flat}$ with at least three fixed points, we obtain 
$\hat{\theta}_{|S} \equiv Id$.
The ergodicity of $\hat{\mu}$ implies that
$\partial S$ contains exactly one non-essential curve $\gamma$.

Let $D_0$ be the topological disc in $\hat{S}_{0}$ enclosed by $\gamma$.
Since $\rho_{\mu}(\tilde{f}) = 0$
the diffeomorphism $\tilde{f}$ fixes (modulo isotopy rel $\varpi^{-1}(\hat{K}_{0})$)
a lift $\tilde{D}_0$ of $D_0$. Moreover $K_{0}' := \tilde{D}_0 \cap \varpi^{-1}(\hat{K}_{0})$
is a non-empty compact $\tilde{f}$-invariant set such that
$\pi_{|K_{0}'}: K_{0}' \to K_{0}$
is an homeomorphism. Hence we can lift $\mu$ to $K_{0}'$.
The hypothesis $\mathrm{Fix}(\tilde{f}) \cap (\pi^{-1}(K_{0}) \cap \tilde{A}) = \emptyset$
implies $\mathrm{Fix}(\tilde{f}) \cap  K_{0}'  = \emptyset$.
Now we just apply Proposition \ref{pro:clp} to $\tilde{f}$
and $K_{0}'$. Hence ${f}^{\natural}: D \to D$ is a lift of $f$ by the map
$\pi \circ \pi_{V}$ where $D = \tilde{V}$ and
$\pi_{V}: \tilde{V} \to V$ is the universal covering map
of a connected component $V$ of $\tilde{S}_{0} \setminus  K_{0}'$.
Note that $V$ is not an annulus since otherwise it does not contain non-empty
compactly covered $\tilde{f}_{|\tilde{S}_0 \setminus K_{0}'}$-Nielsen classes. 
Indeed, in such a case  the lift of $\tilde{f}$ commutes with 
every deck transformation of $\tilde{V}$.
\end{proof}
\begin{rem}
\label{rem:wlif}
We proved that the homeomorphism ${f}^{\natural}:D \to D$ is
a lift of $f$ by a map of the form
$\pi' \circ \pi_{W}$ where $\pi'$ is a covering of $S_{0}$, $f'$ is a lift of
$f$ to $(\pi')^{-1}(S_{0})$ and
$\pi_{W}:D \to W$ is the universal covering map of a connected component $W$ of
$(\pi')^{-1}(S_{0}) \setminus K_{0}'$, where $K_{0}'$ is a $f'$-invariant
compact set contained in $(\pi')^{-1}(K_{0})$. Moreover, $W$ is not an annulus. 
\end{rem}
\begin{rem}
The map $\pi' \circ \pi_{W}$ is not a covering map and it is just a local homeomorphism.
Anyway, since $\pi'$ and $\pi_{W}$ are covering maps we can lift $f$ to $(\pi')^{-1}(S_{0})$
by $\pi'$ to obtain $f'$ and then to $D$ by $\pi_{W}$ to obtain ${f}^{\natural}$. 
\end{rem}
 %
%
%
%
%
The next result finds and localizes fixed points.
\begin{pro}
\label{pro:ann}
Consider the hypotheses of Proposition \ref{pro:clr}.
Suppose that there exists a compact $f$-invariant
subset $K$ of $S_{0} \setminus (\overline{A} \setminus A)$ such that 
$\emptyset \neq K_{0} \subset K$,
$f(A)$ is isotopic to $A$ rel $K$ and $f$ preserves the ends of $A$ modulo isotopy rel $K$. 
Consider a nilpotent group $G= \langle H,f \rangle  \subset \mathrm{Diff}_{+}^{1}(S_{0})$.
Suppose that any point in $\mathrm{Fix}(H) \cap \pi (\mathrm{Fix}(\tilde{f}))$
is $f$-Nielsen equivalent rel $K$ to a point in $K$. Suppose further that
$H$ is a normal subgroup of $G$ of elements isotopic to the identity rel
$K$.  Then
$\mathrm{Fix}(H) \cap \pi (\mathrm{Fix}(\tilde{f})) \cap K \cap A 
\neq \emptyset$.
Moreover, if $\chi (S_0) <0$, we can  replace $\mathrm{Fix}(H)$ with $\mathrm{Fix}_{c}(H)$
in the statement and then 
$\mathrm{Fix}_{c}(H) \cap \pi (\mathrm{Fix}(\tilde{f})) \cap K \cap A  \neq \emptyset$ holds.
\end{pro}
\begin{proof}
Let  ${f}^{\natural}:D \to D$ be the lift of $f$
provided by Proposition \ref{pro:clr}.
Consider the description of ${f}^{\natural}$ in Remark \ref{rem:wlif}.
There are two cases, namely $(\pi')^{-1}(S_{0})$ is an annulus or a topological disc. 
We adopt the notations in Proposition \ref{pro:clr} and Remark \ref{rem:wlif}.

First, assume that  $\hat{S}_{0}:=(\pi')^{-1}(S_{0})$   is an annulus. 
Consider the set 
$\hat{H}$ consisting of the identity lifts rel $K$ of elements of $H$. 
It is a subgroup of $\mathrm{Diff}_{0} (\hat{S}_{0})$ consisting of elements isotopic to 
$\mathrm{Id}$ rel $\hat{\pi}^{-1}(K)$ and in particular rel $\hat{K}_{0}$.
Moreover, $\hat{H}$ is a normal subgroup of $\hat{G}:=\langle \hat{H}, \hat{f} \rangle$ and thus
$\hat{G}$ is nilpotent. By construction, 
there exists a non-empty compactly covered 
$\hat{f}_{|\hat{S}_{0} \setminus \hat{K}_{0}}$-Nielsen class 
${\mathcal N} = \pi_{\hat{U}} (\mathrm{Fix}(f^{\natural}))$. 
There exists $\hat{w} \in \mathrm{Fix} (\hat{G}) \cap {\mathcal N}$ by 
Proposition \ref{pro:clp2}; let $w = \hat{\pi} (w)$ be its projection in $S_0$. 
Since $w \in \mathrm{Fix}(H) \cap \pi (\mathrm{Fix}(\tilde{f}))$, there exists 
$z \in K$ such that $z$ is $f$-Nielsen equivalent to $w$ rel $K$ by hypothesis. 
In particular, we have $z \in  \pi (\mathrm{Fix}(\tilde{f}))$.
There exist $\gamma:[0,1] \to S_0$ such that $\gamma (0) = w$, $\gamma (1)=z$
and a homotopy $I$ rel $K$ between $\gamma$ and $f \circ \gamma$
(cf. Definition \ref{def:homotopy}).
We define $\hat{z} = \hat{\gamma} (1)$ where $\hat{\gamma}$ is the lift of $\gamma$ to 
$\hat{S}_{0}$ such that $\hat{\gamma} (0) = \hat{w}$. 
Moreover, we have that the image of $I$ does not intersect $\hat{K}_{0}$  since otherwise 
$\mathrm{Fix}(\tilde{f}) \cap (\pi^{-1}(K_{0}) \cap \tilde{A}) \neq \emptyset$, contradicting the
hypothesis. 
Thus $\hat{z}$ belongs to the same  $\hat{f}_{|\hat{S}_{0} \setminus \hat{K}_{0}}$-Nielsen class  
as $\hat{w}$.  Since such a class is compactly covered, we obtain 
$\hat{z} \in \hat{A}$. In particular $z$ belongs to 
$\mathrm{Fix}(H) \cap \pi (\mathrm{Fix}(\tilde{f})) \cap K \cap A$.  
Notice, that for the case $\chi (S_0) <0$, if we replace $\mathrm{Fix}(H)$ with 
$\mathrm{Fix}_{c}(H)$ in the statement of the proposition, we obtain 
$w \in \mathrm{Fix}_{c}(H)$ by Proposition \ref{pro:clp2}
and hence we can proceed analogously to obtain
$z \in \mathrm{Fix}_{c}(H) \cap \pi (\mathrm{Fix}(\tilde{f})) \cap K \cap A$. 

Finally, suppose that  $\tilde{S}_{0}:=(\pi')^{-1}(S_{0})$   is a topological disc. 
Consider the group $\tilde{H}$ of identity lifts rel $K$ of elements of $H$.
The group $\tilde{H}$ is a normal subgroup of the nilpotent group 
$\tilde{G}:= \langle \tilde{H}, \tilde{f} \rangle$.
Since 
$\mathrm{Fix}(\tilde{f}) \cap (\pi^{-1}(K_{0}) \cap \tilde{A}) = \emptyset$, it follows 
that $\mathrm{Fix} (\tilde{f}) \cap K_{0}' = \emptyset$.
We apply Corollary \ref{cor:clp2} to obtain $\tilde{w} \in \mathrm{Fix} (\tilde{G})$ such that its 
 $\tilde{f}_{|\tilde{S}_{0} \setminus K_{0}'}$-Nielsen class ${\mathcal N}$ 
 is compactly covered.
 Let $w = \pi (\tilde{w})$ and consider a point $z \in K$ such that it is 
 $f$-Nielsen equivalent to $w$ rel $K$. 
 We can proceed as above to show that there exists a lift $\tilde{z}$ of $z$ to $\tilde{S}_0$
 that belongs to ${\mathcal N}$. Since ${\mathcal N}$ is compactly covered 
 rel $K_{0}' =  \tilde{D}_0 \cap \pi^{-1}(K_0)$ and 
 $\tilde{z} \in \pi^{-1}(K)$, we deduce that 
 $\tilde{z}$ belongs to $ \tilde{D}_0 \cap \pi^{-1}(K)$ and then
 $z \in \mathrm{Fix}(H) \cap \pi (\mathrm{Fix}(\tilde{f})) \cap K \cap A$.
 Moreover, if $\chi (S_0) <0$, we obtain $z  \in \mathrm{Fix}_{c}(H)$ by construction. 
\end{proof}  
\section{Localization of fixed points in the ``cyclic" case}
\label{sec:cyclic}
 The goal of the next two section is finding global fixed points for suitable subgroups of 
nilpotent subgroup of $\mathrm{Diff}_{0}^{1} ({\mathcal S})$ where ${\mathcal S}$
is a finitely punctured compact surface of negative Euler characteristic. 
The next theorem is useful to localize such global fixed points. 
\begin{teo}
\label{teo:mainc}
Let $S$ be an orientable connected compact surface (maybe with boundary) of 
negative Euler characteristic.
Let $G$ be a finitely generated nilpotent subgroup of 
$\{ \phi \in \mathrm{Homeo}_{0} (S): 
\phi_{|S \setminus \partial S} \in  \mathrm{Diff}_{0}^{1}(S \setminus \partial S) \}$
of nilpotency class  $\iota$.
Assume that $\phi$ is isotopic to $\mathrm{Id}$ rel
$\mathrm{Fix}({\mathcal C}^{(k)}(G)) \cup \partial S$
for all $k \geq 1$ and $\phi \in {\mathcal C}^{(k)}(G)$.
Let $g \in G$ and $0 \leq l < \iota$.
Let $\langle G',g \rangle^{\sharp}$ be a lift of $\langle G',g \rangle$ 
by a covering $S^{\sharp} \to S \setminus (K \cup \partial S)$ where 
$K \subset S \setminus \partial S$ is a $g$-invariant compact set contained in 
$\mathrm{Fix}(G')$, $S^{\sharp}$ is a finitely punctured
surface of negative Euler characteristic and the natural map
$\pi_{1}(S^{\sharp}) \to \pi_{1}(S)$ is injective.
Suppose that $g^{\sharp}$ is isotopic to the identity and that
$\mathrm{Fix}_{c}((gh)^{\sharp})$ is a non-empty compact set 
for any $h \in {\mathcal C}^{(\iota-l)}(G)$.
Then $\mathrm{Fix}_{c} \langle {\mathcal C}^{(\iota-l)}(G),g \rangle^{\sharp} \neq \emptyset$.
\end{teo}
We introduce the following auxiliary property:
\begin{itemize}
\item ${\bf B_{k}}$: Theorem \ref{teo:mainc} holds true for any finitely generated group $G$,
any nilpotency class  $\iota$ and any $l \leq \min(k, \iota -1)$.
\end{itemize}
We show Property $\bf{B_k}$ by induction on $k$.  Property $\bf{B_0}$ holds by hypothesis  
since ${\mathcal C}^{(\iota)} (G) = \{ \mathrm{Id} \}$.
\subsection{Global fixed points and the framework}
Later on, we will describe the particular framework (see Definition \ref{def:frame})
that we will use to show Theorem \ref{teo:mainc}. 
In this section, we will introduce general properties that will be used also in the proof of 
Theorem \ref{teo:maina}.
\begin{rem}
\label{rem:frame}
Consider the framework described in Definition \ref{def:frame}. 
There exists  an excellent compact set  ${\mathcal K}$  for  
${\mathcal G}$ and $\langle {\mathcal H}, f_{1} \rangle$, $\hdots$, 
$\langle {\mathcal H}, f_{q} \rangle$ by Remark \ref{rem:ex_ex}. We fix
 a Thurston decomposition for ${\mathcal G}$ and
$\langle {\mathcal H}, f_{1} \rangle$, $\hdots$, $\langle {\mathcal H}, f_{q} \rangle$ 
rel ${\mathcal K}$ by Proposition 
\ref{pro:Tdecomp} and Remark \ref{rem:decomp}.
\end{rem}
\begin{defi}
We denote by ${\mathcal R}$ and ${\mathcal R}_{e}$ the sets of
reducing curves and essential (in ${\mathcal S}$) reducing curves
respectively for the Thurston decomposition. 
\end{defi}
We are going to present several results that allow us to find global fixed points of 
subgroups in the framework.  Then, next lemma ensures 
that there are global contractible fixed points of ${\mathcal G}$. 
More precisely, we will apply the next result to subgroups of the form 
$\langle {\mathcal H}, f_{i_1}, \hdots, f_{i_e} \rangle$
where $1 \leq i_1 < i_2 < \hdots < i_e \leq q$. 
\begin{lem}
\label{lem:fin_ord}
Let ${\mathcal G}$ be a subgroup of $\mathrm{Diff}_{0}^{1}({\mathcal S})$ where 
${\mathcal S}$ is a  connected surface of negative Euler characteristic. Consider a finite index 
subgroup ${\mathcal J}$ of ${\mathcal G}$ such that 
$\mathrm{Fix}_c ({\mathcal J}) \neq \emptyset$. 
Then $\mathrm{Fix}_c ({\mathcal G}) \neq \emptyset$. 
\end{lem}
\begin{proof}
Denote by $\tilde{\mathcal J}$ and $\tilde{\mathcal G}$ the groups of identity lifts of 
elements of ${\mathcal J}$ and ${\mathcal G}$ respectively to the universal covering 
$\tilde{\mathcal S}$ of ${\mathcal S}$.
Fix $z \in \mathrm{Fix}_c ({\mathcal J})$ and one of its lifts $\tilde{z} \in \tilde{\mathcal S}$. 
Since $\tilde{z} \in \mathrm{Fix}(\tilde{\mathcal J})$ and $\tilde{\mathcal J}$ is a finite subgroup 
of $\tilde{\mathcal G}$, we deduce that the $\tilde{\mathcal G}$-orbit of $\tilde{z}$ is finite. 
Thus $\tilde{\mathcal G}$ preserves a non-empty compact set and as a consequence 
$\tilde{\mathcal G}$ has a global fixed point by Theorem \ref{teo:plane}. 
Therefore, we get $\mathrm{Fix}_c ({\mathcal G}) \neq \emptyset$.   
\end{proof}
Later on, we will need to localize the global fixed points of the subgroups of the framework
in some specific regions of the Thurston decomposition that are introduced in next 
definition.
\begin{defi}
\label{def:xm}
Let ${\mathcal C}$ be the set of connected components of
${\mathcal S} \setminus {\mathcal R}_{e}$ that either have
negative Euler characteristic or are annuli bounding connected
components of negative Euler characteristic. Let ${\mathcal C}_{-}$
be the subset of ${\mathcal C}$ of connected components of negative Euler characteristic.
Given a connected component $M$ of $S \setminus {\mathcal R}_{e}$
we denote by $X_{M}$ the set obtained by removing from $M$ all the topological
discs enclosed by curves in $({\mathcal R} \setminus {\mathcal R}_{e}) \cap M$
and the points in ${\mathcal A}$.
\end{defi}
\subsection{Chasing global fixed points}
\label{subsec:chase}
Now, we are going to study several configurations, for the global fixed points of the auxiliary 
subgroups in the framework, that induce global fixed points of ${\mathcal G}$.
\subsubsection{Infinite case}
\begin{pro}
\label{pro:inf}
Let $M$ be a connected component of ${\mathcal S} \setminus {\mathcal R}_{e}$
such that 
$\sharp (X_{M} \cap {\mathcal K} \cap \mathrm{Fix}_{c} \langle {\mathcal H},f_{b} \rangle)= \infty$
for any $1 \leq b \leq e$.
Then $\mathrm{Fix}_{c}({\mathcal G})$ is a non-empty set.
\end{pro}
\begin{proof}
Every element of ${\mathcal J}:= \langle {\mathcal H}, f_1, \hdots, f_e \rangle$ 
is isotopic rel ${\mathcal K} \cap X_{M}$
to a diffeomorphism that is the identity map in $M$ by the properties of the 
Thurston decomposition.
Since 
$X_{M} \cap {\mathcal K} \cap \mathrm{Fix}_{c} \langle {\mathcal H},f_{b} \rangle \neq \emptyset$,
we deduce that 
$X_{M} \cap {\mathcal K} \subset \mathrm{Fix}_{c} \langle {\mathcal H},f_{b} \rangle$
for any $1 \leq b \leq e$. 
Thus $X_{M} \cap {\mathcal K}$ is contained in $\mathrm{Fix}_{c} ( {\mathcal J})$.
We obtain $\mathrm{Fix}_c ({\mathcal G}) \neq \emptyset$ by Lemma \ref{lem:fin_ord}.
\end{proof}

\subsubsection{Finite case in negative Euler characteristic}
\begin{pro}
\label{pro:puzfne}
Let $M$ be a connected component of ${\mathcal S} \setminus {\mathcal R}_{e}$
such that $\chi (M) <0$. Suppose that
\[ \sharp (X_{M} \cap {\mathcal K} \cap \mathrm{Fix}_{c} \langle {\mathcal H}, f_{b}\rangle)< \infty
 \ \mathrm{and} \ 
M \cap {\mathcal K} \cap \mathrm{Fix}_{c} \langle {\mathcal H}, f_{b} \rangle \neq \emptyset \]
for any $1 \leq b \leq e$.
Then $\mathrm{Fix}_{c}({\mathcal G})$ is a non-empty set.
\end{pro}
\begin{proof}
It suffices to show  $\mathrm{Fix}_{c}({\mathcal J}) \neq \emptyset$, where 
${\mathcal J}= \langle {\mathcal H},f_{1},\hdots, f_{e} \rangle$, 
by Lemma \ref{lem:fin_ord}.
The group $[{\mathcal J}]$ induced in $\mathrm{MCG}(X_{M},{\mathcal K} \cap X_{M})$
by ${\mathcal J}$ is an abelian group ($[{\mathcal H}]=1$)
whose elements are all irreducible.
Consider the group ${\mathcal J}_{0}$ of ${\mathcal J}$ consisting of   
the elements whose image
in $\mathrm{MCG}(X_{M},{\mathcal K} \cap X_{M})$ is trivial.
Since the boundary of $\partial M$ consists of essential curves, 
the universal covering $\tilde{M}$ of $M$ can be interpreted as a subset
of $\tilde{\mathcal S}$.
We denote by $\pi: \tilde{\mathcal S} \to {\mathcal S}$ the universal covering.

We claim that $\mathrm{Fix}_{c}({\mathcal J}_{0})$ is non-empty.
Suppose that there exists  $z \in X_{M} \cap {\mathcal K}$.
Any $\phi \in {\mathcal J}_{0}$ is isotopic rel $\{z\}$
to a map $\theta$ that is the identity in $M$. Since
$\chi (M)<0$, we obtain $z \in \mathrm{Fix}_{c}({\mathcal J}_{0})$.
Obviously $\mathrm{Fix}(\tilde{\phi})$ contains
$\pi^{-1}(z)$ where $\tilde{\phi}$ is the identity lift.
Analogously, if there exists a reducing curve $\gamma$ in
$M \cap ({\mathcal R} \setminus  {\mathcal R}_{e})$ 
we obtain that $\tilde{\phi}(\tilde{\gamma})$ is isotopic to
$\tilde{\gamma}$ rel $\pi^{-1}({\mathcal K})$ for any
lift $\tilde{\gamma}$ of $\gamma$.
Let $D_{\gamma}$ be the disc in ${\mathcal S}$ bounded by $\gamma$.
Let $\tilde{D}_{\gamma}$ be a lift of $D_{\gamma}$. We have
that $\pi^{-1}({\mathcal K}) \cap \tilde{D}_{\gamma}$ is a non-empty
compact ${\mathcal J}_{0}$-invariant set.
We obtain $\mathrm{Fix}_{c}({\mathcal J}_{0}) \neq \emptyset$ by
Theorem \ref{teo:plane}.

Suppose that $[{\mathcal J}]$ is finite. Then
${\mathcal J}_{0}$ is a finite index normal subgroup of ${\mathcal J}$
and the result is a consequence of Lemma \ref{lem:fin_ord}.
%
%

Suppose $ \sharp [{\mathcal J}]=\infty$.
Since ${\mathcal J}$ is nilpotent, its subset $\mathrm{Tor} ({\mathcal J})$ of 
finite order elements is a subgroup
\cite[Theorem 16.2.7]{Karga}. Hence  $ \sharp [{\mathcal J}]=\infty$
implies the existence of 
%
%
$1 \leq b \leq e$ such that $[f_{b}]$ is of infinite order
and then pseudo-Anosov.
%
%
Moreover $[{\mathcal J}]$ is virtually cyclic (see Lemma 2.1 of \cite{JR:arxivsp})
so $\langle {\mathcal J}_{0},f_{b} \rangle$ is a finite index normal subgroup of
${\mathcal J}$.
It suffices to prove  
$\mathrm{Fix}_{c} \langle {\mathcal J}_{0}, f_{b} \rangle \neq \emptyset$
by Lemma \ref{lem:fin_ord}.
There are two cases, namely
$X_{M} \cap {\mathcal K} \cap \mathrm{Fix}_{c}(f_{b}) \neq \emptyset$
and $(M \setminus X_{M}) \cap {\mathcal K} \cap \mathrm{Fix}_{c}(f_{b}) \neq \emptyset$.
The former case is easy since
$X_{M} \cap {\mathcal K} \cap \mathrm{Fix}_{c}(f_{b})$
is contained in $\mathrm{Fix}_{c} \langle {\mathcal J}_{0}, f_{b} \rangle$.
In the latter case, consider a reducing curve $\gamma$ in
$M \cap ({\mathcal R} \setminus  {\mathcal R}_{e})$ 
such that
$D_{\gamma} \cap (M \setminus X_{M}) \cap {\mathcal K} \cap \mathrm{Fix}_{c}(f_{b}) \neq \emptyset$.
The curve 
$\tilde{f}_{b}(\tilde{\gamma})$ is isotopic to $\tilde{\gamma}$ rel $\pi^{-1}({\mathcal K})$.
The compact set $\pi^{-1}({\mathcal K}) \cap \tilde{D}_{\gamma}$ is
invariant by the identity lift of $\langle {\mathcal J}_{0}, f_{b} \rangle$.
Thus $\mathrm{Fix}_{c} \langle {\mathcal J}_{0},f_{b} \rangle$ is non-empty 
by Theorem \ref{teo:plane}. 
\end{proof}
\subsubsection{Finite irrotational case and finite isotopically trivial case}
\begin{pro}
\label{pro:puzfic}
Let $M$ be a connected component of ${\mathcal S} \setminus {\mathcal R}_{e}$
such that $\chi (M) =0$. Suppose that
\[ \sharp (X_{M} \cap {\mathcal K} \cap \mathrm{Fix}_{c} \langle{\mathcal H}, f_{b} \rangle)< \infty
\ \mathrm{and} \ 
M \cap {\mathcal K} \cap \mathrm{Fix}_{c}\langle{\mathcal H}, f_{b} \rangle \neq \emptyset \]
for any $1 \leq b \leq e$.
Suppose further that there exists a ${\mathcal G}$-invariant measure
$\mu$ such that $\mathrm{supp}(\mu) \subset {\mathcal K} \cap M$
and $\rho_{\mu}(\tilde{\phi}) =0$ (see Definition \ref{def:rho} and Remark \ref{rem:rho})
for any $\phi \in {\mathcal G}$.
Then $\mathrm{Fix}_{c}({\mathcal G})$ is a non-empty set.
\end{pro} 
\begin{pro}
\label{pro:puzfitc}
Let $M$ be a connected component of ${\mathcal S} \setminus {\mathcal R}_{e}$
such that $\chi (M) =0$. Suppose that
\[ \sharp (X_{M} \cap {\mathcal K} \cap \mathrm{Fix}_{c} \langle{\mathcal H}, f_{b} \rangle)< \infty
\ \mathrm{and} \ 
M \cap {\mathcal K} \cap \mathrm{Fix}_{c}\langle{\mathcal H}, f_{b} \rangle \neq \emptyset \]
for any $1 \leq b \leq e$.
Assume that the element $[f_{b}]$ of
$\mathrm{MCG}(X_{M},{\mathcal K} \cap X_{M})$ induced by
$f_{b}$ is trivial for any $1 \leq b \leq e$.
Then $\mathrm{Fix}_{c}({\mathcal G})$ is a non-empty set.
\end{pro}
\begin{proof}[Proof of Propositions \ref{pro:puzfic} and \ref{pro:puzfitc}]
It suffices to show $\mathrm{Fix}_{c} ({\mathcal J}) \neq \emptyset$, where 
${\mathcal J}=  \langle {\mathcal H},f_{1},\hdots, f_{e} \rangle$, 
by Lemma \ref{lem:fin_ord}.
The proof is the same as in Proposition \ref{pro:puzfne}
except that we have to change slightly the second paragraph.
Indeed we need to prove that in the new setting it is still true
that $X_{M} \cap {\mathcal K} \subset \mathrm{Fix}_{c}({\mathcal J}_{0})$
and that given a curve $\gamma$ in $M \cap ({\mathcal R} \setminus  {\mathcal R}_{e})$ 
then $\tilde{\phi}(\tilde{\gamma})$ is isotopic to $\tilde{\gamma}$
rel $\pi^{-1}({\mathcal K})$ for any lift $\tilde{\gamma}$ of $\gamma$ and
any $\phi \in {\mathcal J}_{0}$.
If any of these properties does not hold true for some $\phi \in {\mathcal J}_{0}$ then
there exists $k \in {\mathbb Z} \setminus \{0\}$
such that $\rho(\tilde{\phi},z) =k$ for any $z \in {\mathcal K} \cap M$.
This contradicts the property $\rho_{\mu}(\tilde{\phi})=0$. Hence the properties
hold true in the setting of Proposition \ref{pro:puzfic}.

Let us consider the setting in Proposition \ref{pro:puzfitc}; we have ${\mathcal J}={\mathcal J}_{0}$.
We claim that the properties hold true for any
$\phi \in {\mathcal H} \cup \{f_{1}, \hdots, f_{e}\}$.
Otherwise we deduce
$\mathrm{Fix}_{c}(\phi) \cap M \cap {\mathcal K}=\emptyset$;
it contradicts that the set 
$M \cap {\mathcal K} \cap \mathrm{Fix}_{c} \langle {\mathcal H}, f_{b} \rangle$
is non-empty for any $1 \leq b \leq e$.
Hence the properties hold true for any $\phi \in {\mathcal J}_{0}$.
\end{proof}
\subsection{Framework for the proof of Theorem \ref{teo:mainc}}
The results providing global fixed points in subsection \ref{subsec:chase} are a fundamental
ingredient of the proofs of Theorems \ref{teo:maina} and  \ref{teo:mainc}.
Let us focus on the latter result and how to particularize the framework in 
Definition \ref{def:frame} to such a situation. Let us remind the reader that our goal in this
section is showing that ${\bf B_k}$ implies ${\bf B_{k+1}}$ for $k \geq 0$.

Assume that the hypotheses of Theorem  \ref{teo:mainc} and ${\bf B_{k}}$ hold and 
consider $l \leq \min(k+1, \iota -1)$.
Consider the groups 
\[ H={\mathcal C}^{(\iota+1-l)}(G), \ 
{\mathcal H}={\mathcal C}^{(\iota+1-l)}(G)^{\sharp} \ \mbox{and} \ 
{\mathcal G}=\langle {\mathcal C}^{(\iota -l)}(G),g \rangle^{\sharp}. \]
Notice that every subgroup of a finitely generated nilpotent group is finitely generated
(cf. \cite[Theorem 17.2.2]{Karga}).
The group 
$\langle {\mathcal C}^{(\iota -l)}(G),g \rangle / H$ is a finitely generated abelian group
by Remark \ref{rem:des_abe}
.  
There exists a subgroup $J$ of ${\mathcal C}^{(\iota-l)}(G)$ such that $H \subset J$, the
group $J/H$ is torsion free and ${\mathcal C}^{(\iota-l)}(G)/J$ is a finite group by the fundamental
theorem of finitely generated abelian groups.
We denote $e= \mathrm{rank}(J/H) +1$ and $q = -2   \chi (S^{\sharp}) e$.
Consider a sequence
\[ h_{1}, \hdots, h_{q} \in J \]
where the group generated by
$\langle H, h_{i_2}^{-1} h_{i_1},  h_{i_3}^{-1} h_{i_1}, h_{i_{e}}^{-1} h_{i_1} \rangle$
is a finite index subgroup of  ${\mathcal C}^{(\iota-l)}(G)$ for any choice of 
pairwise different indexes $i_1, \hdots, i_{e}$ in $\{ 1, \hdots, q\}$.
Hence, 
$\langle H, g h_{i_{1}}, \hdots, g h_{i_{e}} \rangle$ is a finite index subgroup
of $\langle {\mathcal C}^{(\iota-l)}(G), g \rangle$ for any choice of 
$1 \leq i_{1} < \hdots < i_{e} \leq q$.
We denote $g_j = g h_j$, $f_j = g_{j}^{\sharp}$ 
and  $H_{j}=\langle H, g_{j} \rangle$ for $1 \leq j \leq q$.
Any group of the form $H_{j}$ is a normal subgroup of
$\langle {\mathcal C}^{(\iota-l)}(G),g \rangle$ (Remark  \ref{rem:des_abe}) that
satisfies the hypotheses of ${\bf B_{k}}$.
We deduce that $\mathrm{Fix}_{c} (H_{j}^{\sharp})$ is a non-empty compact set for
any $1 \leq j \leq q$. 

 Note that every element of 
${\mathcal H}$ is isotopic to $\mathrm{Id}$ rel $\mathrm{Fix} ({\mathcal H})$
by the inclusions $K \subset \mathrm{Fix}(G') \subset \mathrm{Fix}(H)$ and by hypothesis.
Thus, the conditions established in the framework in 
Definition \ref{def:frame} are satisfied. 
Consider an excellent compact set ${\mathcal K}$ with respect to the group
${\mathcal G}$, the
subgroups $H_{1}^{\sharp}$, $\hdots$, $H_{q}^{\sharp}$ and the set
$Y=\cup_{j=1}^{q} \mathrm{Fix}_{c} (H_{j}^{\sharp})$
(Remark \ref{rem:ex_ex}).
We also consider a Thurston decomposition
for ${\mathcal G}$ and $H_{1}^{\sharp}$, $\hdots$, $H_{q}^{\sharp}$ 
relative to ${\mathcal K}$ (Proposition \ref{pro:Tdecomp}).
We define ${\mathcal S}=S^{\sharp}$.   
\subsection{Cyclic nature of ${\mathcal G}$} 
We keep moving forward with the proof of Theorem \ref{teo:mainc}.
Let us introduce the idea behind the results in this section. 
Our point of view relies on profiting on the ``isotopically trivial"  nature of the 
subgroups  ${\mathcal C}^{(j)} (G)$ for $j \geq 1$ in the descending
central series.
As a consequence, it is reasonable to expect that many of the properties of 
the group $\langle {\mathcal C}^{(\iota -l)}(G),f \rangle$ are going to be derived from
the properties of $f$.

There are at most $-\chi ({\mathcal S})$ components of negative Euler characteristic
in ${\mathcal S} \setminus {\mathcal R}_{e}$.
Obviously if there is no annulus in ${\mathcal S} \setminus {\mathcal R}_{e}$ then
the points in $\mathrm{Fix}_{c}  \langle {\mathcal H}, f_j \rangle$ are localized
in those components for $1 \leq j \leq q$.

The situation is even simpler if
one of the connected components $A$ of ${\mathcal S} \setminus {\mathcal R}_{e}$
is an annulus. We have
$A \cap {\mathcal K} \cap \mathrm{Fix}_{c}  \langle{\mathcal H}, f_{j_0} \rangle \neq \emptyset$ 
by construction for some $1 \leq j_{0} \leq q$.   
Notice that $A \cap {\mathcal K} \cap \mathrm{Fix}_{c}  \langle{\mathcal H}, f_{j_0} \rangle$
is ${\mathcal G}$-invariant by the invariance of the Thurston decomposition and 
Remark  \ref{rem:inv_nor0}.

%
%

Let $\pi:\tilde{\mathcal S} \to {\mathcal S}$ be the universal covering.
Let $\tilde{A}$ be one of the lifts of $A$ to $\tilde{\mathcal S}$.
Given $\phi \in  \langle G',f \rangle$, we denote by $\phi^{\sharp}$ its lift to  
$\langle G',g \rangle^{\sharp}$
and by $\tilde{\phi}^{\sharp}: \tilde{\mathcal S} \to \tilde{\mathcal S}$ 
the identity lift of $\phi^{\sharp}$ to $\tilde{\mathcal S}$.
The rotation number $\rho_{\mu}(\tilde{\phi}^{\sharp})$ of $\tilde{\phi}^{\sharp}$ in 
$\tilde{A}$ is well-defined (Definition \ref{def:rho}) and does not depend on the choice of 
$\tilde{A}$ (Remark \ref{rem:rho}).
\begin{lem}
\label{lem:irr}
Let ${\mathcal K}_{0}$ be a non-empty compact subset of
$A \cap {\mathcal K} \cap \mathrm{Fix}_{c}(f_{j_{0}})$.
Suppose that there exist $1 \leq j \leq q$ and a measure $\mu$
such that $\mathrm{supp}(\mu)$ is contained in ${\mathcal K}_{0}$ and
$\mu$ is $f_{j}$-invariant. Then $\rho_{\mu}(\tilde{f}_{j})=0$.
\end{lem}
\begin{proof}

We can assume that $\mu$ is $f_{j}$-ergodic by the ergodic decomposition theorem.
Since $\pi_{1}(S^{\sharp}) \to \pi_{1}(S)$ is injective by hypothesis, 
the primitive deck transformation $T^{\flat}:\tilde{\mathcal S} \to \tilde{\mathcal S}$
preserving $\tilde{A}$ induces by construction
a deck transformation $T:\tilde{S} \to \tilde{S}$.
We can consider that $\tilde{A}$ is contained in $\tilde{S}$.
Moreover,  $T$ is associated to a non-null homotopic closed curve in $S$.
Since $S$ is compact,
$T$ is a hyperbolic isometry of the Poincar\'{e} disc.
As a consequence $\lim_{|n| \to \infty} d(z,T^{n}(z))/n$ is equal to $\kappa$
where $\kappa$ is the length of the geodesic in the free homotopy class determined
by $T$ and $d$ is the hyperbolic metric in $\tilde{S}$.
We denote $\beta =g_{j_{0}}^{-1} g_{j}$. Since $f_{j} = f_{j_{0}} \beta^{\sharp}$
and ${\mathcal K}_0 \subset  \mathrm{Fix}_{c}(f_{j_{0}})$, we deduce
$(\tilde{f}_{j})_{|\pi^{-1}({\mathcal K}_{0})} \equiv
\tilde{\beta}^{\sharp}_{|\pi^{-1}({\mathcal K}_{0})}$
and then $\rho_{\mu}(\tilde{f}_{j})= \rho_{\mu} ( \tilde{\beta}^{\sharp})$.
Since $\mu$ is $f_j$-ergodic, we have
$\lim_{|n| \to \infty} d(z_0,\tilde{\beta}^{n}(z_0))/n = \rho_{\mu} (\tilde{\beta}) \kappa$
for almost every point $z_0$ in $\tilde{A} \cap \pi^{-1}({\mathcal K}_{0})$ 
where we consider $\tilde{A} \subset \tilde{S}$ and $\tilde{\beta}$ is the identity
lift of $\beta$   to $\tilde{S}$. 
Notice that $\beta$ belongs to ${\mathcal C}^{(\iota-l)}(G)$
and $z_0$ belongs to $\mathrm{Fix}({\mathcal C}^{(\iota+1-l)}(\tilde{G}))$
where $\tilde{G}$ is the group consisting of the identity lifts of elements of $G$ to $\tilde{S}$.
Equation (\ref{equ:k2}) implies
$\lim_{|n| \to \infty} d(z_0,(\tilde{\beta}^{n^{2}}(z_0))/n^{2} = 0$.
We deduce that 
$\rho_{\mu}(\tilde{f}_{j}) = \rho_{\mu}(\tilde{g}_j)= \rho_{\mu}(\tilde{\beta})=0$.
\end{proof} 
\begin{cor}
\label{cor:locb}
Let $A$ be an annulus connected component of ${\mathcal S} \setminus {\mathcal R}_e$.
Then the set 
$A \cap {\mathcal K} \cap  \mathrm{Fix}_{c}  \langle{\mathcal H}, f_{j} \rangle$
is non-empty for any $1 \leq j \leq q$.
\end{cor}
\begin{proof}
There exists $1 \leq j_0 \leq q$ such that 
$A \cap {\mathcal K} \cap \mathrm{Fix}_{c}  \langle{\mathcal H}, f_{j_0} \rangle \neq \emptyset$ 
by construction.  
Since 
$A \cap {\mathcal K} \cap \mathrm{Fix}_{c}  \langle{\mathcal H}, f_{j_0} \rangle \neq \emptyset$ 
is ${\mathcal G}$-invariant, we consider a  $f_{j}$-invariant measure $\mu$ with 
${\mathcal K}_{0}:=\mathrm{supp}(\mu)$ contained in
$A \cap {\mathcal K} \cap \mathrm{Fix}_{c}  \langle{\mathcal H}, f_{j_0} \rangle$.
We obtain $\rho_{\mu}(\tilde{f}_{j})=0$ by Lemma \ref{lem:irr}. We can assume 
$\mathrm{Fix}(\tilde{f}_{j}) \cap (\pi^{-1} ({\mathcal K}_0) \cap \tilde{A}) =
\emptyset$ since otherwise there is nothing to show.
Let us apply Proposition \ref{pro:ann} with $f=f_{j}$, $\tilde{f} = \tilde{f}_{j}$. The set
$\mathrm{Fix}_{c} ( {\mathcal H} ) \cap \pi(\mathrm{Fix}(\tilde{f}_{j}))$
coincides with $\mathrm{Fix}_{c}  \langle{\mathcal H}, f_{j} \rangle$.
Moreover, ${\mathcal H}$ is a normal subgroup of $\langle {\mathcal H}, f_{j} \rangle$
because $H$ is a normal subgroup of $G$. 
We deduce
$A \cap {\mathcal K} \cap  \mathrm{Fix}_{c}  \langle{\mathcal H}, f_{j} \rangle \neq \emptyset$
by Proposition \ref{pro:ann}.
\end{proof} 
\begin{cor}
\label{cor:irr}
Let $A$ be an annulus connected component of ${\mathcal S} \setminus {\mathcal R}_e$.
Suppose that there exists a ${\mathcal G}$-invariant measure $\mu$
whose support is contained in
$A \cap {\mathcal K} \cap \mathrm{Fix}_{c}(f_{j_{0}})$.
Then $(\rho_{\mu})_{|\tilde{\mathcal G}} \equiv 0$.
\end{cor}
\begin{proof}
Let us remind the reader that ${\mathcal G}=\langle {\mathcal C}^{\iota -l}(G),f \rangle^{\sharp}$. 
Since the group $\tilde{\mathcal G}$ of identity lifts of elements of ${\mathcal G}$ to 
$\tilde{\mathcal S}$ is nilpotent, the map
$\rho_{\mu}: \tilde{\mathcal G} \to {\mathbb R}$ is a morphism of groups.
We know that $\rho_{\mu}$ vanishes  in 
$\{  \tilde{f}_{1}, \hdots, \tilde{f}_{q} \}$ and $\tilde{\mathcal H}$ 
by Lemma \ref{lem:irr} and because the elements of ${\mathcal H}$ are isotopic to 
$\mathrm{Id}$ rel ${\mathcal K}$ (see the proof of Corollary \ref{cor:locb}).
Since  $\langle \tilde{\mathcal H}, \tilde{f}_{1}, \hdots, \tilde{f}_{q} \rangle$
is  a finite index subgroup of $\tilde{\mathcal G}$ by construction, 
$\rho_{\mu}$ vanishes over $\tilde{\mathcal G}$.
\end{proof}
\subsection{Proof of ${\bf B_k}$ implies ${\bf B_{k+1}}$}
\label{subsec:bkbkp1}
Let us consider the possible cases for the configuration of global fixed points of subgroups of
the framework. First,  suppose 
$\sharp 
(X_{M} \cap {\mathcal K} \cap  \mathrm{Fix}_{c} \langle {\mathcal H}, f_{j}\rangle)=\infty$
for some $M \in {\mathcal C}$ (see Definition \ref{def:xm}) and
at least $e$ values of $j \in \{1, \hdots,q\}$.
We can suppose that the property holds true for $j=1,\hdots,e$ without lack of
generality. We obtain $\mathrm{Fix}_{c}({\mathcal G}) \neq \emptyset$
by Proposition \ref{pro:inf}.

Now suppose that  one of the connected components of 
${\mathcal S} \setminus {\mathcal R}_{e}$ is an annulus $A$. We have that 
$A \cap {\mathcal K} \cap \mathrm{Fix}_{c} \langle {\mathcal H}, f_{j} \rangle$
is a non-empty set for any $1 \leq j \leq q$ by Corollary \ref{cor:locb}.
If $\sharp (X_{A} \cap {\mathcal K} \cap \mathrm{Fix}_{c} \langle {\mathcal H},f_{j} \rangle)=\infty$
for at least $e$ values of $j$ then there is nothing to prove by 
the argument in the above paragraph. 
Up to reordering the indexes we can suppose
$\sharp (X_{A} \cap {\mathcal K} \cap \mathrm{Fix}_{c} \langle {\mathcal H},f_{j} \rangle)<\infty$
for $1 \leq j \leq q  -(e -1)$.
Notice that $q -e +1 >e$ since $q=-2 e \chi({\mathcal S})$. We define
\[ K_{1} = 
\cup_{j=1}^{e} (X_{A} \cap {\mathcal K} \cap  \mathrm{Fix}_{c} \langle{\mathcal H},f_{j}\rangle), \
K = ({\mathcal K} \setminus X_{A}) \cup K_{1}. \]
The set $K$ is a good compact set for the group ${\mathcal G}$ and the subgroups
$\langle {\mathcal H}, f_{1} \rangle$, $\hdots$, $\langle {\mathcal H}, f_{e} \rangle$
by Remark \ref{rem:tgood}.
The group ${\mathcal G}$ induces a group $[{\mathcal G}]$
in $\mathrm{MCG}(X_{A},X_{A} \cap K)$
that does not necessarily consist of irreducible elements.
Anyway we can add the curves of the Thurston decomposition 
(see Remark \ref{rem:decomp}) of $[{\mathcal G}]$ to ${\mathcal R}$.
We obtain a connected component $A'$ of
${\mathcal S} \setminus {\mathcal R}_{e}$
such that $A'$ is an annulus contained in $A$,
$A' \cap K \cap \cup_{j=1}^{e} \mathrm{Fix}_{c} \langle {\mathcal H}, f_{j} \rangle$
is non-empty, $X_{A'} \cap K$ is finite and the group induced by ${\mathcal G}$ in
$\mathrm{MCG}(X_{A'},X_{A'} \cap K)$ is irreducible.
We replace ${\mathcal K}$ with $K$ and $A$ with $A'$.
The old elements in ${\mathcal C}$ contain the new ones.

We have that 
$A \cap {\mathcal K} \cap \mathrm{Fix}_{c} \langle {\mathcal H}, f_{j} \rangle$
is non-empty for any $1 \leq j \leq e$ by Corollary \ref{cor:locb}.
Fix $1 \leq j_{0} \leq e$.
Since $A \cap {\mathcal K} \cap \mathrm{Fix}_{c} \langle {\mathcal H}, f_{j_{0}} \rangle$
is a compact non-empty ${\mathcal G}$-invariant set and ${\mathcal G}$ is nilpotent there exists
a ${\mathcal G}$-invariant measure $\mu$ such that
$\mathrm{supp}(\mu) \subset 
A \cap {\mathcal K} \cap \mathrm{Fix}_{c} \langle {\mathcal H}, f_{j_{0}} \rangle$.
We obtain $\rho_{\mu}(\tilde{\phi})=0$ for any $\phi \in {\mathcal G}$ by
Corollary \ref{cor:irr}.
Proposition \ref{pro:puzfic}
implies $\mathrm{Fix}_{c}({\mathcal G}) \neq \emptyset$.
Hence, we can suppose that there are no annuli connected components in  
${\mathcal S} \setminus {\mathcal R}_{e}$. We deduce  
$\sharp {\mathcal C} \leq - \chi ({\mathcal S})$.

 \strut

By the first paragraph of the proof and up to  reordering the indexes we can suppose
$\sharp (X_{M} \cap {\mathcal K} \cap  \mathrm{Fix}_{c} \langle {\mathcal H}, f_{j} \rangle) 
< \infty$
for all $M \in {\mathcal C}$ and  $1 \leq j \leq q'$ where 
$q'=q+e \chi({\mathcal S}) = -e \chi({\mathcal S})$. We define
\[ K_{1} = \cup_{M \in {\mathcal C}} \cup_{j=1}^{q'} (X_{M} \cap {\mathcal K} \cap
\mathrm{Fix}_{c} \langle {\mathcal H}, f_{j} \rangle), \
K = ({\mathcal K} \setminus  \cup_{M \in {\mathcal C}} X_{M}) \cup K_{1}. \]
The set $K$ is a good compact set for the group ${\mathcal G}$ and
the subgroups
$\langle {\mathcal H}, f_{1} \rangle$, $\hdots$, $\langle {\mathcal H}, f_{q'} \rangle$.
By adding simple closed curves contained in 
$(\cup_{M \in {\mathcal C}} X_M) \setminus K$ to ${\mathcal R}$, we can obtain 
a Thurston-reducing set ${\mathcal R}$ of curves for 
${\mathcal G}$ and 
$\langle {\mathcal H}, f_{1} \rangle$, $\hdots$, $\langle {\mathcal H}, f_{q'} \rangle$
rel $K$ (Remark \ref{rem:tgood}). We replace ${\mathcal K}$ with $K$.


We can suppose that there is no annulus in
${\mathcal S} \setminus {\mathcal R}_{e}$ 
for the new Thurston decomposition
since otherwise we obtain 
$\mathrm{Fix}_{c}({\mathcal G}) \neq \emptyset$.
Since $\sharp {\mathcal C}= \sharp {\mathcal C}_{-} \leq - \chi ({\mathcal S})$ and
$q' = -e \chi ({\mathcal S})$, we get 
$M \cap {\mathcal K} \cap \mathrm{Fix}_{c}\langle{\mathcal H}, f_{j} \rangle \neq \emptyset$
for $e$ values of $j$ and  some $M \in {\mathcal C}_{-}$.
Then $\mathrm{Fix}_{c}({\mathcal G}) \neq \emptyset$ by Proposition \ref{pro:puzfne}, completing
the proof.
\section{Localization of fixed points in the general case}
\label{sec:loc} 
Our next goal is proving Theorem \ref{teo:maina}. 
It will be obtained as a corollary of the following theorem:
\begin{teo}
\label{teo:baux}
Let $G$ be a finitely generated nilpotent subgroup of $\mathrm{Homeo}_{0} (S)$ where
$S$ is a compact connected orientable 
surface (maybe with boundary) of negative Euler characteristic. 
Suppose $\phi_{|S \setminus \partial S} \in \mathrm{Diff}_{0}^{1} (S \setminus \partial S)$
for any $\phi \in G$.
Furthermore, assume that $\phi$ is isotopic to $\mathrm{Id}$ rel 
$\mathrm{Fix} ({\mathcal C}^{(k)} (G)) \cup \partial S$ for all $k \geq 1$ and 
$\phi \in {\mathcal C}^{(k)} (G)$. Then $\mathrm{Fix}_{c} (G) \neq \emptyset$.
\end{teo} 
In this section we introduce another fundamental result, providing localization 
of global fixed points. 
The localization process executed in the proof of Theorem \ref{teo:mainc} 
has some simplifying characteristics: 
either there exists one 
connected component of ${\mathcal S} \setminus {\mathcal R}_e$ that is a topological annulus 
(and we find global fixed points in such a component) or the number of connected 
components of ${\mathcal S} \setminus {\mathcal R}_e$ is bounded by above by
$- \chi ({\mathcal S})$.
In order to prove Theorem \ref{teo:maina} we can not suppose anymore that there are finitely
many connected components of ${\mathcal S} \setminus {\mathcal R}_e$ and hence we
need to localize the global fixed points in some specific ones.
\subsection{Framework for the proof of Theorem \ref{teo:baux}}  
\label{sec:frameb}
 First, let us specify the framework 
introduced in Definition  \ref{def:frame} fits with our current target.

Let $\pi: \tilde{S} \to S$ be the universal covering.
Given $f \in G$, we consider the identity lift $\tilde{f}$ of $f$ to $\tilde{S}$.
We denote by $\tilde{G}$ the group consisting of the identity lifts of elements of $G$.
 Suppose first $\partial S \neq \emptyset$. Fix a boundary component  $\gamma$ of $S$. 
 Since $G_{|\gamma}$ is nilpotent, we obtain that 
 $\rho_{\gamma}: \tilde{G} \to {\mathbb R}$ is a morphism of groups.  
 Moreover,   there exists a 
 $G$-invariant measure $\mu$ supported in $\gamma$.
 
Assume that  $\rho_{\gamma} (\tilde{f})=0$ for any $\tilde{f} \in \tilde{G}$. 
 We obtain 
 \[ \mathrm{Fix}_{c} (f) \cap \gamma = \mathrm{Fix}(f) \cap \gamma \neq \emptyset \]
 for any $f \in G$.  Since 
 $\mathrm{Fix} (f)  \cap \gamma \neq \emptyset$ and $f$ acts by translations in the connected
 components of $\gamma \setminus \mathrm{Fix}(g)$, we deduce 
 $\mathrm{sup} (\mu) \subset \mathrm{Fix}_{c} (f) \cap \gamma$ for any $f \in G$. 
 In particular we obtain $\mathrm{Fix}_{c}(G) \cap \gamma \neq \emptyset$. 
 So we can suppose that there exists $f_{0,\gamma} \in G$ such that 
 $\rho_{\gamma} (\tilde{f}_{0,\gamma}) \neq 0$ for any boundary component $\gamma$ of $S$.
 Moreover, since $\rho_{\gamma}$ is a morphism, it is easy to see that there exists 
 $f_0 \in G$ such that
 $\rho_{\gamma'} (\tilde{f}_0) \neq 0$ for any boundary component $\gamma'$ of $S$. 
 For the case $\partial S = \emptyset$ we define $f_0 = \mathrm{Id}$.

Since the group $G/G'$ is commutative, there exists a normal subgroup $J$ of $G$
such that $G' \subset J$, $\sharp (G/J) < \infty$ and $J/G'$ is torsion free
by the fundamental theorem of finitely generated abelian groups.
We can suppose $f_0 \in J$ up to replace it with a non-trivial iterate.
We denote $e = \mathrm{rank}(G/G')$. 
Consider a sequence $f_{1}, \hdots, f_{q} \in J$ where $q= -8 \chi(S) e$
and $\langle [f_{i_{1}}],\hdots,[f_{i_{e}}] \rangle \subset G/G'$ has rank $e$ 
(and thus $\langle G', f_{i_{1}},\hdots,f_{i_{e}} \rangle$ is a finite index normal subgroup of $G$) 
for any choice of $1 \leq i_{1} < \hdots < i_{e} \leq q$.
This property is preserved if we replace $f_{1}, \hdots, f_{q}$ with 
$f_{1} f_{0}^{k}, \hdots, f_{q} f_{0}^{k}$ for some $k \in {\mathbb N}$ big enough
(cf. \cite[Lemma 7.1]{FHP-g}). In particular
the conditions in the framework in Definition \ref{def:frame} are satisfied
for ${\mathcal S}=S$, ${\mathcal G}=G_{|S \setminus \partial S}$, 
${\mathcal H}=G_{|S \setminus \partial S}'$ and on top of that we can suppose 
$\rho_{\gamma} (\tilde{f}_j) \neq 0$ for all boundary component $\gamma$ of $S$ and 
$1 \leq j \leq q$.

We denote $H_{j}= {\langle G', f_{j} \rangle}_{|S \setminus \partial S}$.
Since $L(\mathrm{Fix}_{c} (f_j g), f_j g) = \chi (S)$  and 
$\mathrm{Fix}_{c} (f_j g) \cap \partial S = \emptyset$, it follows that $\mathrm{Fix}_{c} (f_j g)$
is a non-empty compact subset of $S \setminus \partial S$
for any $g \in G'$. Hence we obtain $\mathrm{Fix}_{c} (H_{j}) \neq \emptyset$ by 
Theorem \ref{teo:mainc} (where $K=\emptyset$, $S^{\sharp} = S \setminus \partial S$ and the 
covering $S^{\sharp} \to S \setminus \partial S$ is the identity map).
There exists  an excellent compact set  ${\mathcal K}$ with respect to the group 
${\mathcal G}$, the
subgroups $H_{1}$, $\hdots$, $H_{q}$ and the set
$Y=\cup_{j=1}^{q} \mathrm{Fix}_{c} (H_{j})$ by Remark \ref{rem:ex_ex}. We also consider
a Thurston-reducing set ${\mathcal R}$ of curves rel ${\mathcal K}$ for the group 
${\mathcal G}$ and the subgroups 
$H_1$, $\hdots$, $H_q$ (Proposition \ref{pro:Tdecomp}). We obtain a Thurston decomposition
of $G$ for the normal subgroups $H_{1}$, $\hdots$, $H_{q}$ relative to ${\mathcal K}$ by 
Remark \ref{rem:decomp}.

Consider the set ${\mathfrak C}$ of connected components of $S \setminus {\mathcal R}_{e}$.
There are at most $-\chi (S)$ such components of negative Euler characteristic, i.e.
$\sharp {\mathcal C}_{-} \leq - \chi (S)$ (cf. Definition \ref{def:xm}).
Moreover, the maximal number $I$ of disjoint essential
simple closed curves that are pairwise non-homotopical 
satisfies $2 I + b = - 3 \chi (S)$ where $b$ is the number of boundary components of $S$. 
There are no three  homotopical annuli, 
bounding connected components of
$S \setminus {\mathcal R}_e$ of negative Euler characteristic
(or two homotopic annuli whose core curve is homotopic to a boundary component). 
We deduce that there are at most $2I+b$ annuli in ${\mathfrak C}$ 
bounding components of negative Euler characteristic, i.e. 
$\sharp ({\mathcal C} \setminus {\mathcal C}_{-}) \leq -3 \chi (S)$.

Next, we introduce the definition of the region $C$ in which we localize global fixed points. 
\begin{defi}
\label{def:c}
Consider the framework introduced in this section.
We define  ${\mathcal R}_{e}$ (resp. ${\mathcal A}_e$) as the set of essential 
curves in ${\mathcal R}$ (resp. annuli in ${\mathcal A}$).
We denote by ${\mathfrak a}$ the set of connected components of
$S \setminus {\mathcal R}_{e}$ that are annuli.
Consider the closure $N'$  of a connected component of 
$S \setminus \cup_{A \in {\mathfrak a}} \overline{A}$;
it satisfies $\chi (N') <0$. 
Let $A_{1}'$, $\hdots$, $A_{r}'$ be the open annuli in
$S \setminus {\mathcal R}_{e}$ bounding $N'$.
We denote by ${\mathcal A}_{N,l}$ the set of annuli $A$ of ${\mathcal A}$ such that 
$A \cap \overline{A_{l}'} \cap \partial N' \neq \emptyset$. We  define
${\mathcal A}_{N} = \cup_{l=1}^{r} {\mathcal A}_{N,l}$ and
$N = \overline{N' \setminus {\mathcal A}_N}$. 
We denote   $C= N' \cup A_{1}' \cup \hdots \cup A_{r}'$ and
$A_l = \overline{A_{l}' \setminus {\mathcal A}_e} \cup {\mathcal A}_{N,l}$.
\end{defi}
\subsection{Property ${\bf C}$}
Our goal is finding points of ${\mathcal K} \cap \mathrm{Fix}_{c}(f_j)$ in $C$
(cf. Definition \ref{def:c}). 

First, we introduce a result that allows to localize global fixed points in the annuli of 
the Thurston decomposition.
Fix $1 \leq j \leq q$ and denote $f=f_j$ for simplicity.
\begin{lem}
\label{lem:vrot}
Suppose that $K_{l} \subset {\mathcal K} \cap A_{l}$  is the support  of a $f$-ergodic 
measure $\mu_{l}$ for some $1 \leq l \leq r$ such that $\rho_{\mu_{l}}(\tilde{f})=0$.
Then $A_{l} \cap {\mathcal K} \cap \mathrm{Fix}_{c}(f) \neq \emptyset$
\end{lem}
Notice that $\rho_{\mu_{l}}(\tilde{f})$ is well-defined by Definition \ref{def:rho} and Remark 
\ref{rem:rho}.
\begin{proof}
If $K_{l} \cap \mathrm{Fix}_{c}(f) \neq \emptyset$
there is nothing to prove.  Otherwise, 
we define $A=A_{l}$, $K_{0}=K_{l}$, $H=G'$.  
Note  that $K_{0} \cap \mathrm{Fix}_{c}(f) \neq \emptyset$ implies 
$\mathrm{Fix}(\tilde{f}) \cap (\pi^{-1}(K_{0}) \cap \tilde{A}) = \emptyset$.
Since ${\mathcal K}$ is a good set, the hypotheses of Proposition \ref{pro:ann}
hold. Therefore, we deduce $A_{l} \cap {\mathcal K} \cap \mathrm{Fix}_{c}(f) \neq \emptyset$.
\end{proof}
Next, we introduce the setup to localize global fixed points in the case of negative 
Euler characteristic.
We denote $X_{l}=X_{A_{l}'}$ (cf. Definition \ref{def:xm}). 
There exists an isotopy $(f_t)_{t \in [0,1]}$ in $\mathrm{Homeo}_{+} (S)$
rel ${\mathcal K} \setminus N$ such that $f=f_1$ and $\theta:=f_0$ satisfies
$\theta_{|N} \equiv Id$, $\theta({\mathcal A}_{N}) ={\mathcal A}_{N}$ and
$\theta(X_{l})= X_{l}$ for any $1 \leq l \leq r$.
Moreover we can suppose that $\theta_{|X_{l}}$ is irreducible for
any $l$ such that $\sharp (X_{l} \cap {\mathcal K}) < \infty$.

Consider a connected component $\tilde{N}$ of $\pi^{-1}(N)$.
 We can define $\tilde{\theta}$ as the lift of $\theta$ to $\tilde{S}$ such that 
$\tilde{\theta}_{|\tilde{N}} \equiv Id$. Since $\chi (N) <0$, it is the identity lift of $\theta$.
\begin{defi}
\label{def:fin_inf}
We denote by ${\mathcal F}_{f}$ the subset of $\{ 1, \hdots, r \}$
consisting of indexes $l$ such that the following properties hold:
\begin{itemize}
\item$\sharp (X_{l} \cap {\mathcal K}) < \infty$;   
\item the class induced by $\theta$ in $\mathrm{MCG} (X_l)$ is non-trivial or the restriction of
$\theta$ to any connected component of ${\mathcal A}_{N}$ contained in $A_{l}$ is
a non-trivial Dehn twist.
\end{itemize}
We denote by ${\mathcal E}_{f}$ the subset of 
$\{1, \hdots, r\} \setminus {\mathcal F}_{f}$ consisting of indexes $l$
such that  $A_{l} \cap {\mathcal K} \cap \mathrm{Fix}_{c}(f) = \emptyset$.
\end{defi}
\begin{defi}
\label{def:choice_k}
We define $K_l = {\mathcal K} \cap A_l$ if $l \in {\mathcal F}_{f}$.
Given $l \in {\mathcal E}_f$, we choose $K_l$ as the support of a $f$-invariant
ergodic measure $\mu_l$ with $K_l \subset {\mathcal K} \cap A_l$.  
\end{defi}
\begin{rem}
\label{rem:turnr}
Given $l \in {\mathcal E}_{f}$ , we have  $\rho_{\mu_{l}}(\tilde{f}) \neq 0$  
by Lemma \ref{lem:vrot}. Moreover, it is obvious that 
$K_{l} \cap \mathrm{Fix}_{c}(f) = \emptyset$.
\end{rem}
We introduce the main property of this section.
\begin{defi}[Property  ${\bf C}$]
Consider the setting in Definitions \ref{def:c} and \ref{def:fin_inf}.
The property holds if 
\[ N \cap {\mathcal K} \cap \mathrm{Fix}_{c}(f_j) \neq \emptyset \]
whenever ${\mathcal F}_{f_j} \cup {\mathcal E}_{f_j} = \{1, \hdots, r \}$. 
\end{defi}
We will  prove Theorem \ref{teo:baux} 
in section \ref{sec:proofa}.
The goal of this section is showing that 
Theorem \ref{teo:mainc}  implies   ${\bf C}$. 
Let us remind the reader that we proved  
$\mathrm{Fix}_{c} \langle {\mathcal H}, f_{m} \rangle \neq \emptyset$ for any $1 \leq m \leq q$
in section \ref{sec:frameb}. 
\begin{rem}
\label{rem:case_r0}
Since $\mathrm{Fix}_{c} \langle {\mathcal H}, f_{m} \rangle \neq \emptyset$ and 
${\mathcal K}$ is a good set, we deduce 
${\mathcal K} \cap \mathrm{Fix}_{c}(f_{m}) \neq \emptyset$ for any $1 \leq m \leq q$. 
In particular Property ${\bf C}$ trivially holds if $r=0$, i.e. if all connected components of
${\mathcal S} \setminus {\mathcal R}_{e}$ have negative Euler characteristic.
\end{rem}

Let $S^{\flat}$ be the universal covering of
the connected component of
$S \setminus \cup_{l=1}^{r} K_{l}$ containing $N$.
We can consider $\tilde{N}$ as a subset of $S^{\flat}$.
We denote by $\theta^{\flat}$ the lift of $\theta$ to $S^{\flat}$
such that $(\theta^{\flat})_{|\tilde{N}} \equiv \mathrm{Id}$.

Fix $l \in {\mathcal E}_{f}$. Consider ${z}_{l} \in K_{l}$
such that it is recurrent
and has $\rho_{\mu_{l}}(\tilde{f})$ as its   rotation number.
By defining a homeomorphism $h: S \to S$ whose support is
contained $\cup_{l \in {\mathcal E}_{f}} V_{l}$, where $V_{l}$ is a
small neighborhood of $z_{l}$, we can suppose that
$\theta':={\theta} \circ {h}$ has a finite orbit ${\mathcal O}_{l}$
containing ${z}_{l}$ and contained in $K_{l}$. Moreover
if $V_{l}$ is a small neighborhood of $z_{l}$ the condition $\rho_{\mu_{l}}(\tilde{f}) \neq 0$
implies that the orbit  of any lift $\tilde{z}_{l}$ of $z_{l}$ by $\tilde{\theta}'$
is not finite, where $\tilde{\theta}'$ is the identity lift of $\theta'$ to $\tilde{S}$.

The next result is the main ingredient to localize Nielsen classes of negative Lefschetz number.
\begin{lem}
\label{lem:nn}
Suppose  ${\mathcal F}_{f} \cup {\mathcal E}_{f} = \{1, \hdots, r \}$. 
Let ${\mathcal N}$ (resp. ${\mathcal N}'$) be the 
$\theta_{|S \setminus \cup_{l=1}^{r} K_{l}}$-Nielsen class
(resp. $\theta_{|S \setminus \cup_{l=1}^{r}  {\mathcal O}_{l}}'$-Nielsen class) of $N$. 
Then $\theta^{\flat}$ commutes exactly with the covering maps
of $S^{\flat}$ that preserve $\tilde{N}$.
The set ${\mathcal N}'$ is a union of $\theta_{|S \setminus \cup_{l=1}^{r} K_{l}}$-Nielsen classes, 
\[ N  \subset {\mathcal N} \subset {\mathcal N}' \subset
(N  \cup {\mathcal A}_{N} \cup \cup_{l \in {\mathcal E}_{f}}  A_{l}) \cap \mathrm{Fix}_{c}(\theta) \]
and  ${\mathcal N} \cap A_l = {\mathcal N}' \cap A_l$ for any $l \in {\mathcal F}_f$. 
The $\theta_{|S \setminus \cup_{l=1}^{r} K_{l}}$-Nielsen class 
of any point in $({\mathcal N}' \setminus {\mathcal N}) \cap A_{\ell}$
does not peripherally contain any of the ends of $A_{\ell}$
for any $\ell \in {\mathcal E}_{f}$. Moreover, the sum of the Lefschetz numbers
of the $\theta_{|S \setminus \cup_{l=1}^{r} K_{l}}$-Nielsen classes
in ${\mathcal N}'$ is less or equal than $\chi (N)$.
\end{lem}
\begin{proof}
The property $N   \subset {\mathcal N} \cap  {\mathcal N}'$ is obvious.

Let us prove a book-keeping property. Fix $l \in {\mathcal E}_{f}$. 
We claim that there is no path $\gamma:[0,1] \to A_{l} \setminus K_{l}$
such that $\gamma (0) \in \partial A_{l} \cap \partial N$,
$\gamma (1) \in \partial A_{l}$ and 
\[ \gamma(0) \not \simeq \gamma \simeq \theta \circ \gamma:[0,1],0,1 \to
A_{l} \setminus K_{l}, \partial A_{l}, \partial A_{l} \]
where $\gamma (0)$ is understood as a constant path and
we are considering homotopies of maps 
$\alpha: ([0,1], \{0,1\}) \to (A_{l} \setminus K_l, \partial A_l)$ of pairs. 
Otherwise we obtain  $\rho_{\mu_{l}}(\tilde{f}) = 0$, contradicting Remark \ref{rem:turnr}.
The homeomorphism $\theta'$ satisfies an analogous property replacing
$K_{l}$ with ${\mathcal O}_{l}$.

 Fix $l \in {\mathcal F}_{f}$. We have an analogous book-keeping property.
Denote ${\mathcal X}_l = X_{l} \cup  {\mathcal A}_{N,l}$.
There is no path $\gamma:[0,1] \to  {\mathcal X}_l $
for some $1 \leq l \leq r$ such that
$\gamma (0) \in  \partial A_{l} \cap \partial N$,
$\gamma (1) \in \partial {\mathcal X}_{l} \cup \mathrm{Fix}(\theta)$
and
\[ \gamma(0) \not \simeq \gamma \simeq \theta \circ \gamma:[0,1],0,1 \to
{\mathcal X}_{l}, \partial A_{l},  \partial {\mathcal X}_{l} \cup \mathrm{Fix}(\theta), \]
see Lemmas 1.2, 2.2 of \cite{Jiang-Guo} for the case
$\theta_{|X_{l}} \not \equiv Id$. The case where
$\theta_{|X_{l}} \equiv Id$ and $\theta$ restricted to any connected component of 
${\mathcal A}_{N,l}$ is a non-trivial
Dehn twist is simpler (note that $A_{l} \cap {\mathcal K} \neq \emptyset$).

The book-keeping properties imply that no preimage
of $\partial N \cup \partial C$ (see Definition \ref{def:c}) in $S^{\flat}$
outside of $\partial \tilde{N}$
is  $\theta^{\flat}$-invariant. Given $z \in N$, we claim that
the fixed points of $\theta^{\flat}$ whose projection in $S \setminus \cup_{l=1}^{r} K_{l}$  is $z$
are contained in $\tilde{N}$ since otherwise we get a violation of the book-keeping properties. 
Therefore $\theta^{\flat}$ commutes exactly with the covering maps
of $S^{\flat}$ that preserve $\tilde{N}$.
Moreover, we deduce that ${\mathcal N}$ and ${\mathcal N}'$ are contained in
$C$ and satisfy ${\mathcal N} \cap  A_l = {\mathcal N}' \cap A_l$ for any $l \in {\mathcal F}_f$.

We claim that ${\mathcal N}'$ 
is contained in a compact subset of $S \setminus \cup_{l=1}^{r} K_{l}$
for any sufficiently small perturbation $\theta'$ of $\theta$ as described above.
It suffices to prove the claim for ${\mathcal N}' \cap A_l$ and $A_l \setminus K_l$
and any $1 \leq l \leq r$.
For the case $l \in {\mathcal F}_{f}$,  
the claim is a consequence of the book-keeping property.
Given $l \in {\mathcal E}_f$, it is a consequence of 
$\mathrm{Fix}_{c}(\theta) \cap  K_l =\emptyset$ (Remark \ref{rem:turnr}) and 
${\mathcal N}' \subset \mathrm{Fix}_{c}(\theta)$. As a consequence,  we obtain that 
${\mathcal N}'$ is a union of $\theta_{|S \setminus \cup_{l=1}^{r} K_{l}}$-Nielsen classes 
and in particular
${\mathcal N} \subset {\mathcal N}'$. Fix $1 \leq \ell \leq r$.
Analogously, we get 
that for any point in ${\mathcal N}' \cap A_{\ell}$ 
whose $\theta_{|S \setminus \cup_{l=1}^{r} K_{l}}$-Nielsen class 
peripherally contains an end of $\partial A_{\ell}$ its
$\theta_{|S \setminus \cup_{l=1}^{r} {\mathcal O}_{l}}'$--Nielsen class 
also contains peripherally the same end.

Given $l \in {\mathcal E}_f$, it is clear that 
$A_l \cap {\mathcal N}' \subset \mathrm{Fix}_{c}(\theta)$. 
Given $l \in {\mathcal F}_{f}$, the book-keeping property implies 
${\mathcal N}' \cap A_l = {\mathcal A}_{N,l} \cap \mathrm{Fix}_{c}(\theta)$.

Consider a point $z$ in ${\mathcal N}'  \cap A_{\ell}$ with $\ell \in {\mathcal E}_{f}$.
If the $\theta_{|S \setminus \cup_{l=1}^{r} K_{l}}$-Nielsen class ${\mathcal N}_{z}$
of $z$ 
peripherally contains
and end in $\partial A_{\ell} \setminus \partial N$ then we obtain a violation of the
book-keeping property for $\theta'$.
If ${\mathcal N}_{z}$ peripherally contains and end in  $\partial A_{\ell} \cap \partial N$
then there exists a path $\gamma:[0,1] \to A_{\ell} \setminus K_{\ell}$ such that
$\gamma(0) \in \partial A_{\ell} \cap \partial N$,
$\gamma (1)=z$ and
$\theta  \circ \gamma \simeq \alpha^{m} \gamma$ rel $0,1$ in $A_{\ell} \setminus K_{\ell}$ 
for some $m \in {\mathbb Z}$ where $\alpha$ is a loop with point base at $\gamma(0)$,
contained in $\partial A_{\ell} \cap \partial N$ and turning once
in counter clock wise sense.
Since $z$ belongs to $\mathrm{Fix}_{c}(\theta)$ we deduce $m=0$.
Thus $z$ belongs to ${\mathcal N}$.

Denote 
$C' =N' \cup {\mathcal A}_{N} \cup \cup_{l \in {\mathcal F}_l}X_l \cup  \cup_{l \in {\mathcal E}_l} A_l$.
Consider the Thurston decomposition of $\theta'$ in $C'$ 
rel 
${\mathcal O}:=
\cup_{ l\in {\mathcal F}_l} (X_l \cap K_l) \cup \cup_{l \in {\mathcal E}_l} {\mathcal O}_l$.
We denote by ${\mathcal R}'$ the set of reducing curves. 
Up to isotopy rel ${\mathcal O}$,  we claim that a connected 
component $\gamma$ of $\partial (\cup_{l=1}^{r} A_{l}) \cap \partial N$
 is contained in ${\mathcal R}'$.
Otherwise, $\gamma$ is contained in $\overline{{A}_{\ell}'}$ for some
$\ell \in {\mathcal E}_{f}$.
Moreover, up to replace $\theta'$ with another representative 
in its isotopy class rel   ${\mathcal O}$, we can suppose 
that $\theta_{|M}'  \equiv \mathrm{Id}$ where $M$ is the connected 
component of $C' \setminus {\mathcal R}'$ that contains $N$.
We deduce that  the $\tilde{\theta}'$ rotation number of
every point in ${\mathcal O}_{\ell}$ is $0$. 
This  contradicts the construction of $\theta'$.

There exists a standard map $\hat{\theta} \in \mathrm{Homeo}(C')$ 
(see Jiang and Guo \cite{Jiang-Guo}) such that
$\theta_{|C'}'$ is isotopic to $\hat{\theta}$ rel ${\mathcal O}$ and
$N$ is a $\hat{\theta}_{|C' \setminus {\mathcal O}}$-Nielsen class. 
The Lefschetz index $L(N, \hat{\theta})$ is less or equal than $\chi (N)$
(see Lemma 3.6 of \cite{Jiang-Guo}).
We deduce that $L({\mathcal N}', \theta') = L(N, \hat{\theta}) \leq \chi (N)$
and then that the sum of Lefschetz indexes of the 
$\theta_{|S \setminus \cup_{l=1}^{r} K_{l}}$-Nielsen classes
contained in ${\mathcal N}'$ is equal to $L({\mathcal N}', \theta')$.
\end{proof}
Next, we study the $\theta$-Nielsen classes in ${\mathcal N}' \setminus {\mathcal N}$ and 
their induced $f$-Nielsen classes.
\begin{lem}
\label{lem:nn2}
Consider the setting in Lemma \ref{lem:nn}.
Let ${\mathcal N}_{z}$ the $\theta_{|S \setminus \cup_{l=1}^{r} K_{l}}$-Nielsen class 
of a point $z$ in $({\mathcal N}' \setminus {\mathcal N}) \cap A_{\ell}$ for some 
$\ell \in {\mathcal E}_{f}$. Suppose $L({\mathcal N}_{z},\theta) \neq 0$.
Then ${\mathcal N}_{z}$ is a compactly covered 
$\theta_{|S \setminus \cup_{l=1}^{r} K_{l}}$-Nielsen class 
(and ${\mathcal N}_{z}$ is also a compactly covered $\theta_{|S \setminus K_{\ell}}$-Nielsen class).
Moreover, the $f_{|S \setminus \cup_{l=1}^{r} K_{l}}$-Nielsen class
determined by 
${\mathcal N}_z$ and the isotopy $(f_t)_{t \in [0,1]}$ is non-empty and compactly covered.
\end{lem}
\begin{proof}
Let $S^{\dag}$ be the universal covering of the connected component of
$S \setminus \cup_{l=1}^{r} K_{l}$ containing $z$.
Let $z^{\dag}$ be a lift of $z$ in $S^{\dag}$.
Consider the lift $\theta^{\dag}$ of $\theta$ to $S^{\dag}$
such that
$\theta^{\dag}(z^{\dag})=z^{\dag}$.
The $\theta_{|S \setminus \cup_{l=1}^{r} K_{l}}$-Nielsen class  
${\mathcal N}_{z}$ does not peripherally contain any of the ends of $A_{\ell}$ 
by Lemma \ref{lem:nn}. Hence in order to prove that $\mathrm{Fix}(\theta^{\dag})$ is compact
we can replace $S^{\dag}$ with the universal covering of the connected component of
$A_{\ell} \setminus K_{\ell}$ containing $z$.
If $\mathrm{Fix}(\theta^{\dag})$ is not compact then there exists
a path $\gamma:[0,1] \to A_{\ell} \setminus K_{\ell}$  in $\pi_{1}(A_{\ell} \setminus K_{\ell},z)$
such that
$[\theta \circ \gamma] = [\gamma] \neq 0$.
The next three paragraphs are intended to prove that such $[\gamma]$ does not exist.

Let us show first that $\gamma$ is non-homotopically trivial in $A_{\ell}$. 
Otherwise consider the universal covering $\pi: \tilde{A}_{\ell} \to A_{\ell}$ and 
a lift $\tilde{z}$ of $z$ in   $\tilde{A}_{\ell}$ of $A_{\ell}$.
Let $\tilde{\gamma}$ be the lift of $\gamma$ to $\tilde{A}_{\ell}$ such that 
$\tilde{\gamma}(0)=\tilde{z}$; it is a closed path.
Let ${\mathcal B}$ be the set of bounded connected components of 
$\tilde{A}_{\ell} \setminus \tilde{\gamma}[0,1]$ containing points of $\pi^{-1}(K_{\ell})$.
The set $\pi^{-1}(K_{\ell}) \cap \cup_{B \in {\mathcal B}} B$ is compact.
Thus there exists a finite subset $F$
of $\pi^{-1}(K_{\ell})$ such that any $B \in {\mathcal B}$ 
contains exactly a point of $F$. Moreover we can suppose that the rotation number of every 
point in $F$ is equal to $\rho_{\mu_{\ell}}(\tilde{\theta})$.
The path $\tilde{\gamma}$ is homotopically trivial in $\tilde{A}_{\ell} \setminus \pi^{-1}(K_{\ell})$
if and only if it is homotopically trivial in $\tilde{A}_{\ell} \setminus F$. 
The necessary condition is obvious. For the sufficient condition notice that every bounded 
connected component of $\tilde{A}_{\ell} \setminus \pi^{-1}(K_{\ell})$ is homeomorphic to a disk.
For any $B \in {\mathcal B}$ we can consider a closed disk $D_{B}$ containing 
$K_{\ell} \cap B$. The path $\gamma$ is homotopically trivial in 
$\tilde{A}_{\ell} \setminus \cup_{B \in {\mathcal B}} D_{B}$
since this is a strong deformation retract of $\tilde{A}_{\ell} \setminus F$. 
Thus $\gamma$ is homotopically trivial in $\tilde{A}_{\ell} \setminus \pi^{-1}(K_{\ell})$
since $\pi^{-1}(K_{\ell}) \cap \cup_{B \in {\mathcal B}} B$ is contained in 
$\cup_{B \in {\mathcal B}} D_{B}$.
The property $\rho_{\mu_{\ell}}(\tilde{\theta}) \neq 0$ implies that 
$\tilde{\gamma}$ is null-homotopic in $\tilde{A}_{\ell} \setminus \tilde{\theta}^{-k_{0}}(F)$
for some $k_{0} \in {\mathbb N}$. Hence $\tilde{\theta}^{k_{0}}(\tilde{\gamma})$ is 
null-homotopic in $\tilde{A}_{\ell} \setminus F$. Since $\tilde{\theta}^{k_{0}}(\tilde{\gamma})$ 
is homotopic to $\tilde{\gamma}$ in $\tilde{A}_{\ell} \setminus \pi^{-1}(K_{\ell})$ and hence in 
 $\tilde{A}_{\ell} \setminus F$, we deduce that 
$\tilde{\gamma}$ is null-homotopic in  $\tilde{A}_{\ell} \setminus F$.
Therefore, $\tilde{\gamma}$ is null-homotopic in 
$\tilde{A}_{\ell} \setminus \pi^{-1}(K_{\ell})$. We obtain that $\gamma$ is null-homotopic
in $A_{\ell} \setminus K_{\ell}$. This contradicts the choice of $\gamma$.

Up to replace $\gamma$ with another representative in its homotopy class
there exist an isotopy $(\theta_{t})_{t \in [0,1]}$ rel $K_{\ell}$
and a connected closed $\theta_{1}$-invariant subsurface $W$ containing $\gamma$, contained in
$A_{\ell} \setminus K_{\ell}$, with no curve of $\partial W$ bounding a disc in
$A_{\ell} \setminus K_{\ell}$ and such that $\theta_{0} = \theta$ and 
$\theta_{1 |W}$ is periodic (Lemma 2.12 of \cite{FHPs}).
The set $\partial W$ contains  
non-essential curves; otherwise
$W$ is an annulus and the ergodicity of $\mu_{\ell}$ implies that
${\mathcal N}_{z}$ peripherally contains an end of $A_{\ell}$.
The $(\theta_{s})_{|S \setminus K_{\ell}}$-Nielsen class ${\mathcal N}_{z,s}$ 
induced by ${\mathcal N}_{z}$ and the isotopy
$(\theta_{t})_{t \in [0,1]}$ still does not peripherally contain any ends for any $s \in [0,1]$. 
Moreover ${\mathcal N}_{z,s}$ never contains points in some fixed small neighborhood of $K_{\ell}$
for any $s \in [0,1]$
since ${\mathcal N}_{z} \subset \mathrm{Fix}_{c}(\theta)$, $K_{\ell}$ is compact and 
$K_{\ell} \cap  \mathrm{Fix}_{c}(\theta) = \emptyset$.
Since $L({\mathcal N}_{z},\theta) \neq 0$, it follows that 
$L({\mathcal N}_{z,s}, \theta_s)  \neq 0$ and in particular 
${\mathcal N}_{z,s} \neq \emptyset$ for any $s \in [0,1]$. 
Fix $z_1 \in {\mathcal N}_{z,1}$.

The set $\partial W$ contains at least an essential curve since otherwise $\gamma$ is 
null-homotopic in $A_{\ell}$. Hence $\partial W$
contains two essential curves.  Consider the annulus $V$ obtained by adding
to $W$ the interior of the closed disks bounded by the non-essential curves of $\partial W$.
The construction implies that $K_{\ell} \cap V$ is a non-empty $\theta_1$-invariant set. 
Since $\mu_{\ell}$ is $\theta_1$-ergodic and $V$ is $\theta_1$-invariant, we deduce 
$K_{\ell} \subset V$. The $(\theta_{1})_{|S \setminus K_{\ell}}$-Nielsen class 
${\mathcal N}_{z, 1}$ does not peripherally 
contain the ends of $A_{\ell}$ and thus ${\mathcal N}_{z, 1}$ is also contained in the interior $V$. 
Moreover ${\mathcal N}_{z, 1}$ is contained in the interior of $W$, otherwise there exists 
a connected component $M$ of $V \setminus W$such that 
$\overline{M} \cap {\mathcal N}_{z, 1} \neq \emptyset$ and since 
$ {\mathcal N}_{z, 1} \subset  \mathrm{Fix}_c (\theta_1)$, 
we deduce that the  $\tilde{\theta}$-translation number
of the points of $M \cap K_{\ell}$ is equal $0$. This contradicts that the rotation number of points
of $K_{\ell}$ is non-vanishing a.e.  
Consider the surface $T$ obtained from $W$ by contracting the curves in $\partial W$.
We obtain that $(\theta_1)_{|T}$ is a finite order homeomorphism of a sphere 
fixing at least three points, namely
both essential curves of $\partial W$ and $z_{1}$. Hence we obtain $(\theta_1)_{|W} \equiv Id$.
This leads to a contradiction since it implies that ${\mathcal N}_{z, 1}$ peripherally contains 
both ends of $A_{\ell}$.
 
We proved that  $\mathrm{Fix}(\theta^{\dag})$ is compact
and that ${\mathcal N}_{z}$ is a compactly covered 
$\theta_{|S \setminus \cup_{l=1}^{r} K_{l}}$ and 
$\theta_{|S \setminus K_{l}}$-Nielsen class. 
Thus the connected component of 
$S \setminus \cup_{l=1}^{r} K_{l}$ containing $z$ is not an annulus.
Let $f_{t}: S \to S$ be an isotopy rel $\cup_{l=1}^{r} K_{l}$ with
$f_{0}=\theta$ and $f_{1}=f$. Let $(f_{t}^{\dag})_{t \in [0,1]}$ be the lift of
$(f_{t})_{t \in [0,1]}$ such that $f_{0}^{\dag}=\theta^{\dag}$.
We define $f^{\dag} =f_{1}^{\dag}$.
The $\theta_{|S \setminus \cup_{l=1}^{r} K_{l}}$-Nielsen class ${\mathcal N}_{z}$
is contained in $\mathrm{Fix}_{c}(\theta)$ and then it is
compact in $S \setminus \cup_{l=1}^{r} K_{l}$. Moreover
the union of the $(f_{t})_{|S \setminus \cup_{l=1}^{r} K_{l}}$-Nielsen classes 
for $t \in [0,1]$  
determined by ${\mathcal N}_{z}$ and 
$(f_{t})_{t \in [0,1]}$ is contained in $\cup_{t \in [0,1]} \mathrm{Fix}_{c}(f_{t})$ and thus it is
a compact subset of $S \setminus \cup_{l=1}^{r} K_{l}$.
Since $f_{t}^{\dag}$ does not commute with non-trivial covering
maps of $S^{\dag}$, the set  $\cup_{t \in [0,1]} \mathrm{Fix}(f_{t}^{\dag})$
is compact. As a consequence we deduce
\[ L(\mathrm{Fix}(f^{\dag}), f^{\dag}) = L(\mathrm{Fix}(\theta^{\dag}), \theta^{\dag})
=L({\mathcal N}_{z},\theta) \neq 0 \]
and $\mathrm{Fix}(f^{\dag})$ is a non-empty compact set.
\end{proof}
Now,  we reap the rewards of the work in Lemmas \ref{lem:nn} and \ref{lem:nn2}.
\begin{lem}
\label{lem:lann}
Consider the setting in Lemma \ref{lem:nn2}.  
Then Property ${\bf C}$ holds.
\end{lem}
\begin{proof}
Let ${\mathcal N}$ be the $f_{|S \setminus \cup_{l=1}^{r} K_{l}}$-Nielsen class 
induced by ${\mathcal N}_z$ and the isotopy $(f_t)_{t \in [0,1]}$.
We have that  ${\mathcal N}$ is non-empty and compactly covered 
by Lemma \ref{lem:nn2}. We apply Proposition \ref{pro:clp2} with $H=G'$
and $K_0 = \cup_{l=1}^{r} K_{l}$ to obtain 
${\mathcal N} \cap \mathrm{Fix}_{c} (G')  \neq \emptyset$.
Since ${\mathcal N}_{z} \subset \mathrm{Fix}_{c}(\theta)$, we deduce 
${\mathcal N} \cap \mathrm{Fix}_{c} \langle G' , f \rangle  \neq \emptyset$.

Fix $z_{0} \in {\mathcal N} \cap \mathrm{Fix}_{c} \langle G' , f \rangle$.
Since ${\mathcal K}$ is good, there exists a point $z_{1} \in {\mathcal K} \cap \mathrm{Fix}_{c}(f)$
such that $z_{0}$ and $z_{1}$ are in the same $f$-Nielsen class $\hat{\mathcal N}$
rel ${\mathcal K}$.
 Since $\hat{\mathcal N} \subset \mathrm{Fix}_{c}(f)$ and 
$K_{\ell} \cap \mathrm{Fix}_{c}(f)= \emptyset$, we deduce that 
$\hat{\mathcal N}$ is contained in the $f_{|S \setminus K_{\ell}}$-Nielsen class of $z_0$.

Let us show that $z_1 \in N$. Otherwise $z_1$ belongs to 
the $\theta_{|S \setminus K_{\ell}}$-Nielsen class ${\mathcal N}_{z}$
since the homeomorphisms $f$ and $\theta$ are isotopic rel ${\mathcal K} \setminus N$
by construction. Therefore, $z_1$ belongs to 
${\mathcal K} \cap {\mathcal N}_{z} \cap \mathrm{Fix}_{c}(\theta)$.
Since ${\mathcal N}_{z}$ is a compactly covered 
$\theta_{|S \setminus K_{l}}$-Nielsen class by Lemma  \ref{lem:nn2}, we deduce $z_1 \in A_{\ell}$.
This shows that $A_{\ell}   \cap {\mathcal K} \cap \mathrm{Fix}_{c}(f) \neq \emptyset$, 
contradicting the definition of ${\mathcal E}_{f}$.
\end{proof}
\subsection{Proof of ${\bf C}$}
\label{sec:proof_c_k}
We denote $f=f_{j}$. We can suppose $r \geq 1$ by Remark \ref{rem:case_r0}.
We choose $K_{l}$ for $1 \leq l \leq r$ as in Definition 
\ref{def:choice_k}. Let us apply Lemma \ref{lem:nn}.
We can suppose $L({\mathcal N}_{z}, \theta)=0$
for any $\theta_{|S \setminus \cup_{l=1}^{r} K_{l}}$-Nielsen class 
contained in ${\mathcal N}' \setminus {\mathcal N}$
by Lemma \ref{lem:lann}. Thus, we get
$L({\mathcal N}, \theta)=L({\mathcal N}',\theta') \leq \chi(N) <0$.

Let $S^{\flat}$ be the universal covering of the connected component of 
$S \setminus \cup_{l=1}^{r} K_{l}$ containing $N$.
Consider the group $J$ of covering transformations of $S^{\flat}$
preserving $\tilde{N}$. Consider the covering spaces
$S^{\sharp}$, $\overline{S}^{\sharp}$
associated to the spaces $S \setminus (\cup_{l=1}^{r} K_{l} \cup \partial S)$,
$S \setminus \cup_{l=1}^{r} K_{l}$ respectively 
and the group $J$.
Let $\pi^{\sharp}:\overline{S}^{\sharp} \to S \setminus \cup_{l=1}^{r} K_{l}$
and $\pi^{\flat}:S^{\flat} \to \overline{S}^{\sharp}$ be the covering maps.
In particular we have
$\pi_{1}(S^{\sharp}) \simeq J$,
there is a natural embedding 
$N  \hookrightarrow \overline{S}^{\sharp}$ and $S^{\sharp}$ is homeomorphic to 
$N \setminus \partial N$.
We define $\theta^{\sharp}$ as the lift of $\theta$ to $\overline{S}^{\sharp}$ such that
$(\theta^{\sharp})_{|N} \equiv Id$.
Hence $\theta^{\sharp}$ is isotopic to the identity in $S^{\sharp}$ since $\chi (N) <0$.

Consider $g \in \langle G',f \rangle$. There exists an isotopy $g_{t}:S \to S$
rel ${\mathcal K} \setminus N$ such that $(g_{0})_{|N} \equiv Id$ and $g_{1}=g$.
Indeed if $g$  is of the form $f^{s} h$ where $h \in G'$ we can define $g_{0} = \theta^{s}$.
We claim that $\theta^{s}$ is not isotopic to $\mathrm{Id}$ rel 
${\mathcal K} \setminus {\mathcal N}$ for any $s \in {\mathbb Z}^{*}$.
Otherwise, such a property contradicts 
$\rho_{\mu_{l}}(\tilde{f}) \neq 0$ for $l \in {\mathcal E}_f$
if ${\mathcal E}_{f} \neq \emptyset$.
We also obtain a contradiction for the case ${\mathcal F}_{f} \neq \emptyset$. 
More precisely,  consider the surface $S'$ obtained as 
the interior of $N \cup {\mathcal A}_{N} \cup \cup_{l \in {\mathcal F}} X_l$.
The homeomorphism 
$\theta_{|S'}$ would induce a finite order element $[\theta_{|S'}]$ 
of the mapping class group of $S'$ such that 
$\theta_{|N \setminus \partial S} \equiv \mathrm{Id}$. 
Since $\chi (N) <0$, we deduce 
$[\theta_{|S'}] =1$, contradicting  the second bullet point in Definition \ref{def:fin_inf}.
As a consequence of the claim, $s$ is uniquely defined.

We define $g_{0}^{\sharp}$ as the lift of $g_{0}$
to $\overline{S}^{\sharp}$ such that $(g_{0}^{\sharp})_{|N} \equiv Id$.
We define $g^{\sharp} = g_{1}^{\sharp}$ where
$(g_{t}^{\sharp})$ is the lift of the isotopy $(g_{t})$ to $\overline{S}^{\sharp}$.
Since $\chi (N)<0$, it follows that $g^{\sharp}$ does not depend on the isotopy.
Clearly $g^{\sharp}$ is isotopic to the identity in $\overline{S}^{\sharp}$.
We define $\mathrm{Fix}_{c}(g^{\sharp})$ as the Nielsen identity
class of $g^{\sharp}$.
If $g^{\flat}$ is the lift of $g$ to $S^{\flat}$ determined by
$(g_{t})$ such that $(g_{0}^{\flat})_{|\tilde{N}} \equiv Id$ then
$\mathrm{Fix}_{c}(g^{\sharp})$ is the set $\pi^{\flat}(\mathrm{Fix}(g^{\flat}))$.
The group $\langle G',f \rangle^{\sharp}$ consisting of the lifts of elements of
$\langle G',f \rangle$ is isomorphic to the restriction of $\langle G',f \rangle$
to the connected component of $S \setminus \cup_{l=1}^{r} K_{l}$
containing $N$. In particular $\langle G',f \rangle^{\sharp}$ is nilpotent.

We denote ${\mathcal N}^{\sharp}=\mathrm{Fix}_{c}(\theta^{\sharp})$.
The projection of ${\mathcal N}^{\sharp}$ in $S$ is ${\mathcal N}$.
Since ${\mathcal N} \subset  \mathrm{Fix}_{c}(\theta)$ and 
$A_l \cap {\mathcal K} \cap \mathrm{Fix}_{c}(\theta) = \emptyset$ for any $l \in {\mathcal E}_{f}$,
we deduce 
$\overline{\mathcal N} \cap {\mathcal K}  \cap \cup_{l \in {\mathcal E}_{f}} A_l = \emptyset$.
We also get  $\overline{\mathcal N} \cap \cup_{l \in {\mathcal F}_{f}} K_l = \emptyset$
by the book-keeping property (see third paragraph of Lemma \ref{lem:nn}).
Thus ${\mathcal N}$ is a compact subset of 
$S \setminus ({\mathcal K} \cap \cup_{l=1}^{r} A_{l})$. Moreover, it is a
$\theta_{|S \setminus \cup_{l=1}^{r} K_{l}}$-Nielsen class that
does not peripherally contain rel $\cup_{l=1}^{r} K_{l}$
any of the exterior ends $\partial A_{l} \setminus \partial N$ of the annuli
$A_{1}$, $\hdots$, $A_{r}$ (Lemma \ref{lem:nn}). 
Hence we get   
${\mathcal N} \cap {\mathcal K} \subset N$
and ${\mathcal N} \cap \partial S = \partial N$.
Since  $\theta^{\flat}$ commutes exactly with the covering transformations
in $J$ by Lemma \ref{lem:nn}, it follows that ${\mathcal N}^{\sharp}$ is compact in
${S}^{\sharp} \cup (\partial N \cap  (\pi^{\sharp})^{-1}  (\partial S))$.  
Moreover, we obtain
$L({\mathcal N}^{\sharp}, \theta^{\sharp}) = L({\mathcal N}, \theta) < 0$.
Consider the isotopy $f_{t}:S \to S$ rel $\cup_{l=1}^{r} K_{l}$ such that
$f_{0}=\theta$, $f_{1}=f$.
Denote by ${\mathcal N}_{t}$ the $f_t$--Nielsen class of 
$\pi^{\sharp}(\mathrm{Fix}_{c}(f_{t}^{\sharp}))$ rel $\cup_{l=1}^{r} K_l$. 
Notice that $\pi^{\sharp}(\mathrm{Fix}_{c}(f_{t}^{\sharp})) \subset {\mathcal N}_t$ for any 
$t \in [0,1]$. Analogously as above, we obtain that 
$\cup_{t \in [0,1]} \mathrm{Fix}_{c}(f_{t}^{\sharp})$ is a compact subset of 
${S}^{\sharp} \cup (\partial N \cap  (\pi^{\sharp})^{-1}  (\partial S))$ and 
$\cup_{t \in [0,1]} {\mathcal N}_{t}$ is a compact subset of  
$[(S  \setminus \partial S) \cup (\partial S \cap \partial N)] \setminus \cup_{l=1}^{r} K_{l} $.
%
%
We denote $f^{\sharp}=f_{1}^{\sharp}$;
we obtain
\[ L(\mathrm{Fix}_{c}(f^{\sharp}),f^{\sharp}) =
L({\mathcal N}^{\sharp}, \theta^{\sharp}) = L({\mathcal N}, \theta) < 0. \]
Consider 
$g = f h$ where
$h \in G'$. Since $h$ is isotopic to $\mathrm{Id}$ rel $\cup_{l=1}^{r} K_{l}$,
we consider the identity lift ${h}^{\sharp}$ of
$h$ defined in $\overline{S}^{\sharp}$. The previous construction implies
$g^{\sharp} = f^{\sharp} h^{\sharp}$
and
$L(\mathrm{Fix}_{c}(g^{\sharp}),g^{\sharp}) = L(\mathrm{Fix}_{c}(f^{\sharp}),f^{\sharp})$.
In particular $\mathrm{Fix}_{c}(g^{\sharp})$ is a non-empty compact 
subset of 
${S}^{\sharp} \cup (\partial N \cap  (\pi^{\sharp})^{-1}  (\partial S))$.
Notice that since 
$\mathrm{Fix}_{c}(f) \cap \partial S = \mathrm{Fix}_{c}(g) \cap \partial S = \emptyset$, 
we obtain that $\mathrm{Fix}_{c}(g^{\sharp})$ is a subset of $S^{\sharp}$. 
 
We obtain that $\mathrm{Fix}_{c} \langle G',f \rangle^{\sharp}$ is a non-empty compact 
subset of $S^{\sharp}$ by Theorem \ref{teo:mainc}. 
Thus, we get 
${\mathcal N}_{1} \cap \mathrm{Fix}_{c} \langle G',f \rangle \neq \emptyset$.
Fix $z_{0} \in {\mathcal N}_{1} \cap \mathrm{Fix}_{c} \langle G' , f \rangle$.
Since ${\mathcal K}$ is good, there exists a point 
$z_{1} \in {\mathcal K} \cap {\mathcal N}_{1} \cap  \mathrm{Fix}_{c}(f)$
such that $z_{0}$ and $z_{1}$ are in the same $f$-Nielsen class rel ${\mathcal K}$.

We claim  that $z_1 \in N$. Otherwise, 
$z_1$ belongs to ${\mathcal N}$ since 
$f$ and $\theta$ are isotopic rel ${\mathcal K} \setminus N$.
Since ${\mathcal N} \cap {\mathcal K} \subset N$, we obtain a contradiction. 
\section{Proof of Theorem \ref{teo:maina}}
\label{sec:proofa}
\subsection{Proof of Theorem \ref{teo:baux}}
Consider the framework introduced in section \ref{sec:frameb}. 
Let us consider the connected components of $S \setminus {\mathcal R}_e$
(cf. Definition  \ref{def:xm}).
We have $\sharp {\mathcal C}_{-} \leq - \chi (S)$ and 
$\sharp {\mathcal C} \leq - 4 \chi (S)$  by the discussion in section  \ref{sec:frameb}. 
Property ${\bf C}$ guarantees that $\mathrm{Fix}_{c}(f_{j}) \cap {\mathcal K}$
intersects one of these $- 4 \chi (S)$ components for any
$1 \leq j \leq q$ where $q= -8 \chi(S) e$.

Suppose there exist $M \in {\mathcal C}$ and
$e:= \mathrm{rank}(G/G')$ indexes $1 \leq j \leq q$ with
$\sharp (X_M \cap {\mathcal K} \cap \mathrm{Fix}_{c} \langle {\mathcal H}, f_j \rangle) =
\infty$. Hence, we get $\mathrm{Fix}_{c} (G) \neq \emptyset$ by Proposition \ref{pro:inf}.
Thus, if there are at least $- 4 \chi (S) e$ indexes $1 \leq j \leq q$ such that 
there exists $M_j \in {\mathcal C}$ with 
$\sharp (X_{M_j} \cap {\mathcal K} \cap \mathrm{Fix}_{c} \langle {\mathcal H}, f_j \rangle) =
\infty$ then we are done by the pigeonhole principle.
Therefore, and up to replace $q$ with $-4 \chi (S) e$ and reordering the indexes, 
we can suppose that 
$\sharp (X_M \cap {\mathcal K} \cap \mathrm{Fix}_{c} \langle {\mathcal H}, f_j \rangle) < \infty$ 
for all $M \in {\mathcal C}$ and $1 \leq j \leq q$. 
Denote by ${\mathcal C}'$ the subset of ${\mathcal C} \setminus {\mathcal C}_{-}$ whose
elements $M$ satisfy either $\chi (X_{M})< 0$  or 
 $M \cap {\mathcal K} \cap \cup_{j=1}^{q} \mathrm{Fix}_{c} (f_j) \neq  \emptyset$.
We define
\[ K_{1} = \cup_{M \in {\mathcal C}'} \cup_{j=1}^{q} 
(X_{M} \cap {\mathcal K} \cap
\mathrm{Fix}_{c} \langle {\mathcal H}, f_{j} \rangle), \
K = ({\mathcal K} \setminus  \cup_{M \in {\mathcal C}'} X_{M}) \cup K_{1}. \]
The set $K$ is a good compact set for $G$ and 
$\langle {\mathcal H}, f_1 \rangle, \hdots,  \langle {\mathcal H}, f_q \rangle$.
Our choice of $K$ ensures that either $\chi (M') <0$ or 
$\sharp (M' \cap K) \geq 1 + \chi (M')$ for any connected component $M'$ of 
${\mathcal S} \setminus {\mathcal R}$.
We replace ${\mathcal K}$ with $K$.
We can obtain a Thurston-reducing set of curves rel 
${\mathcal K}$ by adding suitable simple closed curves of 
$(\cup_{M \in {\mathcal C}'} X_M) \setminus {\mathcal K}$ to 
${\mathcal R}$ (Remark \ref{rem:tgood}).
Moreover, since the components in ${\mathcal C}$ for the new set of reducing curves are
contained in the old components, we still have 
$\sharp (X_M \cap {\mathcal K} \cap \mathrm{Fix}_{c} \langle {\mathcal H}, f_j \rangle) < \infty$ 
for all $M \in {\mathcal C}$ and $1 \leq j \leq q$.


Assume that there are $e$ indexes $1 \leq j \leq q$ and $M \in {\mathcal C}_{-}$
such that 
$M  \cap {\mathcal K} \cap \mathrm{Fix}_{c} \langle {\mathcal H}, f_j \rangle \neq \emptyset$.
Thus, we get $\mathrm{Fix}_{c} (G) \neq \emptyset$ by Proposition \ref{pro:puzfne}.
So, by applying the pigeonhole principle, we can suppose that 
$M  \cap {\mathcal K} \cap \mathrm{Fix}_{c} \langle {\mathcal H}, f_j \rangle = \emptyset$
for any $1 \leq j \leq q$ and $M \in {\mathcal C}_{-}$ up to reorganizing the indexes and 
replacing $q$ with $-3 \chi (S) e$.  The set 
${\mathcal K} \setminus \cup_{M \in {\mathcal C}_{-}} M$ is a good compact set for 
$G$ and $\langle {\mathcal H}, f_1 \rangle, \hdots,  \langle {\mathcal H}, f_q \rangle$.
Analogously as above, up to replace ${\mathcal K}$ with 
${\mathcal K} \setminus \cup_{M \in {\mathcal C}_{-}} M$  and 
updating the set ${\mathcal R}$ by adding new curves, we
obtain a Thurston-reducing set for $G$ and 
$\langle {\mathcal H}, f_1 \rangle, \hdots,  \langle {\mathcal H}, f_q \rangle$
rel ${\mathcal K}$.  Moreover since the updated components of 
${\mathcal C}$ (resp. ${\mathcal C}_{-}$) are contained in the old ones, we have
${\mathcal K} \cap \cup_{M \in {\mathcal C}_{-}} M = \emptyset$ and 
$\sharp (X_M \cap {\mathcal K} \cap \mathrm{Fix}_{c} \langle {\mathcal H}, f_j \rangle) <
\infty$ for  all $1 \leq j \leq q$ and $M \in {\mathcal C}$.

 We claim that there exist $M \in {\mathcal C}' \setminus {\mathcal C}_{-}$ and
$e$ values of $1 \leq j \leq q$ such that
the class $[f_{j}]_{M}$ of $f_j$ in $\mathrm{MCG}(X_{M}, X_{M} \cap {\mathcal K})$
is trivial and
$M \cap {\mathcal K} \cap \mathrm{Fix}_{c}\langle{\mathcal H}, f_{j}\rangle  \neq \emptyset$.
Otherwise, $\sharp ({\mathcal C}' \setminus {\mathcal C}_{-}) \leq -3 \chi({\mathcal S})$
and the pigeonhole principle imply that 
there exists $1 \leq j_{0} \leq q$ such that
either 
\begin{itemize}
\item $M \not \in {\mathcal C}'$ or 
\item $[f_{j_{0}}]_{M}$ is non-trivial or
\item $M \cap {\mathcal K} \cap \mathrm{Fix}_{c}\langle{\mathcal H}, f_{j_{0}}\rangle  = \emptyset$
\end{itemize}
for any $M \in {\mathcal C} \setminus {\mathcal C}_{-}$.
 Given a connected component
 $N'$ of $S \setminus \cup_{A \in {\mathfrak a}} A$ (see Definition \ref{def:c}), 
 let $A_{1}'$, $\hdots$, $A_{r}'$ be the open annuli in
$S \setminus {\mathcal R}_{e}$ bounding $N'$. 
As a consequence of the discussion, we obtain 
 that any $A_{l}' \in {\mathcal C}'$ satisfies 
 $l \in {\mathcal F}_{f_{j_0}}$. Moreover, $A_{l}' \not \in {\mathcal C}'$  implies 
 $A_{l}' \cap {\mathcal K} \cap \mathrm{Fix}_{c} (f_{j_0}) =  \emptyset$
 by the choice of ${\mathcal C}'$.  We deduce 
${\mathcal F}_{f_{j_{0}}} \cup {\mathcal E}_{f_{j_{0}}} = \{1, \hdots, r \}$ 
and hence 
 \[ {\mathcal K} \cap N'    \cap \mathrm{Fix}_{c} (f_{j_0}) 
 \neq \emptyset \]
 (see Definition \ref{def:fin_inf}) by Property   ${\bf C}$,
 contradicting 
${\mathcal K} \cap \cup_{M \in {\mathcal C}_{-}} M = \emptyset$. Finally, 
we obtain $\mathrm{Fix}_{c}(G) \neq \emptyset$ by Proposition \ref{pro:puzfitc}.

\subsection{End of the proof of Theorem \ref{teo:maina}}
\label{subsec:end}
As a consequence of Theorems \ref{teo:deriso} and \ref{teo:baux},  
Theorem \ref{teo:maina} holds if $S$ is orientable and
$G$ is finitely generated. 

Assume that $S$ is orientable (but $G$ is not necessarily finitely generated). 
Since $S$ is compact, Theorem \ref{teo:maina} is a consequence of the finite intersection property
for compact subsets of $S$. 

Finally, suppose that $S$ is non-orientable. 
Let us consider the orientable double cover $\hat{S}$ of $S$.
Since any $f \in G$ isotopic to $\mathrm{Id}$ has a unique  lift $\hat{f}$ isotopic to 
$\mathrm{Id}$, the group $\hat{G} := \{ \hat{f} : f \in G \}$ is a lift of $G$ contained in 
$\mathrm{Diff}_{0}^{1} (\hat{S})$.
The set $\mathrm{Fix}_{c}(\hat{G})$ is non-empty by Theorem \ref{teo:maina} for the case
where $\hat{S}$ is orientable and it projects to $\mathrm{Fix}_{c}(G)$.
\section{Proof of Theorem \ref{teo:mainb}}
\label{sec:mainb}
The proof has four steps.
We show Theorem \ref{teo:mainb} for nilpotent groups in the first three steps. 
Indeed, we reduce the proof to a simpler setting in section
\ref{subsec:reduction}. Finally, we complete it in sections  \ref{subsec:anosov} and \ref{subsec:dehn} for
nilpotent groups. The locally nilpotent case is treated in 
section \ref{subsec:loc_nil}.
\subsection{Reduction step}
\label{subsec:reduction}
Assume that $G$ is nilpotent.
Analogously as in subsection \ref{subsec:end} we can  
suppose that $S$ is orientable by replacing $S$ and $G$ with its orientable double cover and 
the group of orientation-preserving lifts of elements of $G$ if necessary.
Moreover, we can suppose that $G \subset \mathrm{Diff}_{+}^{1}(S)$ by replacing $G$
with its normal subgroup $G \cap  \mathrm{Diff}_{+}^{1}(S)$ of index at most $2$.

Let us consider first the case $\chi (S) >0$, i.e. $S= {\mathbb S}^{2}$. Then $G$ has a 
finite orbit with at most $2$ elements \cite[Theorem 1.2]{JR:arxivsp}. So we can suppose 
$\chi (S) <0$ from now on.

Let ${\mathcal R}$ be a minimal Thurston-reducing set of curves for the group $G$. 
We denote by $r$ the number of reducing curves and by ${\mathcal C}_T$ the set of 
connected components of $S \setminus {\mathcal R}$.
We consider the subgroup $G^{0}$ of elements of $G$
that fix every reducing curve and every set of ${\mathcal C}_T$ modulo isotopy; it is a normal 
subgroup of $G$ containing $G_{0}$.  
Moreover since $r \leq - \chi (S)$ and $\sharp {\mathcal C}_T \leq - 3 \chi (S)/2$,  
the index $|G: G_0|$ is bounded 
by a function of $\chi (S)$. Thus, up to replace $G$ with $G^{0}$, we can suppose 
$G= G^{0}$.

Given $f \in G$, there exists $\theta_{f} \in \mathrm{Homeo}_{+} (S)$
such that $f$ is isotopic to $\theta_{f}$ and $\theta_{f} (N)=N$ for any $N \in {\mathcal C}_T$. 
Fix $M \in {\mathcal C}_T$. 
Consider the group $[G]_{M}$ consisting of the classes $[f]_{M}$ of 
$\theta_{f}|_M$ in $\mathrm{MCG}(M)$ for $f \in G$. The class  $[f]_{M}$  
does not depend on the choice of $\theta_{f}$ since $\chi (M) <0$.
The group $[G]_{M}$ is clearly nilpotent. Moreover, it is finitely generated since such 
a property holds for solvable subgroups of a finitely punctured compact surface \cite{BLM}.
The group $\mathrm{Tor} [G]_M$ of finite order elements of $[G]_M$ is a finite normal
subgroup of $[G]_M$; moreover, either $[G]_M= \mathrm{Tor} [G]_M$
(and we define $\beta_M = 1 \in \mathrm{MCG}(M)$) or $[G]_M$
is generated by a pseudo-Anosov class $\beta_M$  and $\mathrm{Tor} [G]_M$
\cite[Lemma 2.1]{JR:arxivsp}. Kerckhoff's   
 realization theorem \cite{Kerck} and Hurwitz's automorphisms theorem 
 (cf. \cite[section 7.2]{Farb-Margalit:primer}) imply 
\[  |\mathrm{Tor} [G]_M |  \leq -42 \chi (M) \leq -42 \chi (S). \]
As a consequence 
$G^{0}:= \{ f \in G : [f]_M \in \langle \beta_M \rangle \ \forall M \in {\mathcal C}_T \}$
is a finite index subgroup of $G$ containing $G_0$ and such that 
$[G:G^0|$ is bounded by a function of $\chi (S)$. From now on, we can suppose $G=G^{0}$
without lack of generality. We divide the proof in two cases.

\subsection{First case: There exists $M \in {\mathcal C}_T$ such that $[G]_M \neq \{1\}$.}
\label{subsec:anosov}
Consider a pseudo-Anosov homeomorphism $g_M \in \mathrm{Homeo}_{+} (M)$ that induces
the mapping class $\beta_M$.
We can assume that $\theta_{f}|_M$ is an iterate of $g_M$, up to replace 
 $\theta_f$ with an isotopic homeomorphism, for any $f \in G$.

Let ${\mathcal F}^{s}$ be the stable foliation of $g_M$. Consider the Euler-Poincar\'{e} formula
\[ - \sharp (\mathrm{Sing} ({\mathcal F}^{s}) \cap \partial M)  +
\sum_{x \in \mathrm{Sing} ({\mathcal F}^{s}) \setminus \partial M} (2 - P_x) = 2 \chi (M), \]
where $P_x$ is the number of prongs of ${\mathcal F}^{s}$ at $x$ 
(cf. \cite[section 5.1]{Fathi-Laudenbach-Poenaru}).
Therefore, there exists $m_M$ (depending just on $\chi (M)$) such that 
$g_{M}^{m_M}$ fixes all singular points and prongs of ${\mathcal F}^{s}$. 
We can replace $g_M$, $\beta_M$ and $G$ 
with $g_M^{m_M}$,  $\beta_M^{m_M}$and $\{ f \in G : [f]_M \in \langle \beta_M^{m_M} \rangle \}$
respectively. 

Let $g \in G$ with $[g]_M = \beta_M$. 
Fix $m \in {\mathbb Z}^{*}$. Assume that  
$\mathrm{Fix}(g_M^{m})$ contains a point $z$ in $M \setminus \partial M$.
Consider a lift $\tilde{z}$ of $z$ to $\tilde{S}$.
Let $G^{0}$ be  the subgroup  of elements $f \in G$ such that 
$[f]_{M} \in \langle \beta_M^{m} \rangle$.
Since $G_0 \subset \{ f \in G: [f]_{M}= 1 \}$, 
we deduce that $G_0$ is a subgroup of $G^{0}$.
Given $f \in G^{0}$,   consider an isotopy $(f_{t})_{t \in [0,1]}$ such that 
$f_0 = \theta_{f}$ and $f_{1}=f$.
We define $\tilde{f}$ as $\tilde{f}_{1}$ where $(\tilde{f}_{t})$
is the lift of $(f_{t})$ such that $\tilde{f}_{0}(\tilde{z})=\tilde{z}$.
The definition of $\tilde{f}$ does not depend on the isotopy since $\chi (S) <0$.
Since $\chi (M) <0$, it is clear that $\tilde{f}$ is the identity lift of $f$
if $f$ is isotopic to the identity.
The group $\tilde{G}^{0}:= \{ \tilde{f}: f \in G^{0} \}$ is isomorphic to $G^{0}$ and thus
nilpotent. Since $g_M^{m}$ is pseudo-Anosov, it does not commute with deck 
transformations of $M$ and the $g_M^{m}$-Nielsen class of $z$ 
satisfies $L(z, g_M^{m}) \neq 0$ and
does not contain peripherally any end of $M$.   Thus $\widetilde{{g}^{m}}$ 
does not commute  with deck transformations of $S$ and 
$\mathrm{Fix} (\widetilde{{g}^{m}})$ is a non-empty compact set.
We obtain $\mathrm{Fix}(\tilde{G}^{0}) \neq \emptyset$ by Corollary \ref{cor:plane}.
Hence the group $G$ has a finite orbit of cardinal at most $m$ contained in the projection of
$\mathrm{Fix}(\tilde{G}^{0})$ and then in $\mathrm{Fix}_{c}(G_{0})$.

As a consequence of the previous discussion, it suffices to show that 
$\mathrm{Fix}(g_M^{m})$ contains an interior point of $M$ for some 
$1 \leq m \leq -\chi (M) +2$.
Otherwise, we deduce that $\mathrm{Sing} ({\mathcal F}^{s}) = {\mathfrak S}$, 
where ${\mathfrak S} := \mathrm{Sing} ({\mathcal F}^{s}) \cap \partial M$, 
since $\mathrm{Sing} ({\mathcal F}^{s}) \subset \mathrm{Fix}(g_M)$ by construction.
We get  $\sharp{\mathfrak S}  = -2 \chi (M)$ by the 
Euler-Poincar\'{e} formula and hence $M \neq S$. 
Since $g_M$ fixes the points in ${\mathfrak S} $
and $\mathrm{Fix} (g_M^{m}) \setminus \partial M= \emptyset$,  we get 
$L(g_M^{m}) = - \sharp {\mathfrak S}$ and then
\[ \mathrm{tr}((g_M^{m})_{*}|H_1 (M,{\mathbb Q}) )= 1 -  L(g_M^{m})=
1 + \sharp {\mathfrak S} \]
for any $1 \leq m \leq 2-\chi (M)$ (cf. \cite{Jiang-Guo}).  
Consider the characteristic polynomial $x^{a} + \sum_{j=0}^{a-1} c_j x^{j}$ of 
$(g_M)_{*}|H_1 (M,{\mathbb Q})$ where 
$a=\dim H_1 (M,{\mathbb Q})$. We have $a=1-\chi(M)$.

We get
\begin{equation}
\label{equ:rec}
 \mathrm{tr}[((g_M)_{*}|H_1 (M,{\mathbb Q}))^{a+k}]  = 
- \sum_{j=0}^{a-1} c_{j}  \mathrm{tr}[((g_M)_{*}|H_1 (M,{\mathbb Q}))^{j+k}]  
\end{equation}
for any $k \geq 0$. Since the first $a+1$ terms of 
$(  \mathrm{tr}((g_M^{k})_{*}|H_1 (M,{\mathbb Q})) )_{k \geq 1}$ are equal to 
$1 + \sharp {\mathfrak S}$, the recurrence equation implies that the sequence is constant. 
We obtain that $a$ is such a  constant by applying equation (\ref{equ:rec}) with $k=0$ and noticing that 
$c_0 \neq 0$.
  This implies 
\[ 1 - 2 \chi (M)=1 + \sharp {\mathfrak S}  = a = 1 - \chi (M), \]
contradicting $\chi (M) \neq 0$.
 
\subsection{Second case: $[G]_M  = \{1\}$ for any $M \in {\mathcal C}_T$.}
 \label{subsec:dehn}
 All the elements in $G$ are isotopic to a product of $k$ Dehn twists
along disjoint pairwise non-isotopic annuli,
a so called Dehn twist subgroup \cite{FHP-g}.
We assume that $k$ is minimal with this property.
We have a natural monomorphism $\xi: \mathrm{MCG}(G) \to {\mathbb Z}^{k}$.
Hence, the derived group $G'$ is contained in $G_{0}$.
If $k=0$ the group $G$ is contained in $\mathrm{Diff}_{0}^{1}(S)$.
The result is a consequence of Theorem \ref{teo:maina}.
We suppose $k>0$.

Consider a connected component $N$ of $S \setminus {\mathcal A}$
(Remark \ref{rem:decomp}).
We choose a connected component $\tilde{N}$ of $\pi^{-1}(N)$ in
$\tilde{S}$. Indeed $\tilde{N}$ is a universal covering of $N$.
Given $f \in G$ there exists a isotopy $f_{t}$
such that $(f_{0})_{|N} \equiv Id$ and $f_{1}=f$.
Consider the group $J$ of covering transformations of $\tilde{S}$
preserving $\tilde{N}$. Consider the covering space
$S^{\dag}$ associated to the space $S$
and the group $J$.
Let $\pi^{\dag}:S^{\dag} \to S$
and $\pi^{\flat}:\tilde{S} \to S^{\dag}$ be the covering maps.
In particular we have
$\pi_{1}(S^{\dag}) \simeq J$,
there is a natural embedding $N \hookrightarrow S^{\dag}$ and $S^{\dag}$ is homeomorphic to $N$.
We define $f_{0}^{\dag}$
as the lift of $f_{0}$ to $S^{\dag}$ such that
$(f_{0}^{\dag})_{|N} \equiv Id$.
Hence $f_{0}^{\dag}$ is isotopic to the identity in $S^{\dag}$.
We define $f^{\dag}$   as $f_{1}^{\dag}$
where $(f_{t}^{\dag})$  is the lift of $(f_{t})$.
Let $\tilde{f}_{0}$ be the lift of $f_{0}$ to $\tilde{S}$ such that
$(\tilde{f}_{0})_{|\tilde{N}} \equiv Id$.
We define $\tilde{f}$ analogously as $f^{\dag}$.
The lifts $\tilde{f}$ and $f^{\dag}$ are well-defined since $\chi (N)<0$
and satisfy $\pi^{\flat} (\mathrm{Fix}(\tilde{f})) = \mathrm{Fix}_{c} (f^{\dag})$ for any $f \in G$.
We obtain lifts $G^{\dag}$ and $\tilde{G}$ of $G$ to $S^{\dag}$
and $\tilde{S}$ respectively.

All the elements of $G^{\dag}$ admit a continuous extension to
a compact surface $\overline{S}^{\dag}$ with boundary. More precisely
$\tilde{S}$ is identified with the Poincar\'{e} disk ${\mathbb D}$ and 
$\overline{S}^{\dag}$ is obtained as a quotient of 
${\mathbb D}  \cup (\partial {\mathbb D} \setminus \overline{\tilde{N})}$
by the action of $J$.
Let $\gamma_1, \hdots, \gamma_{\ell}$ be the circles in $\partial N$. The surface
$\overline{S}^{\dag} \setminus N$ consists of $\ell$ closed annuli $C_1, \hdots, C_{\ell}$
such that $\partial C_j$ consists of two circles, namely 
$\partial_j \subset \partial \overline{S}^{\dag}$ and $\gamma_j$ for any $1 \leq j \leq \ell$.
We can consider the rotation numbers $\rho_j (\tilde{f})$
associated to the homeomorphism $f_{|\partial_j}$ and the lift $\tilde{f}$
for $1 \leq j \leq \ell$. 
We can consider the coordinate $c_j$ of $\xi [f]$ that corresponds to the annulus $A_j$
in ${\mathcal A}$ that contains $\gamma_j$. 
By Remark \ref{rem:partial_rot} and Lemma  \ref{lem:trans} we obtain 
$\rho_{j} (\tilde{f}) = \pm c_j$ where the sign depends on the choice of a generator of 
$H_1 (A_j, {\mathbb Z})$. 
Since ${\mathcal R}$ is minimal, it follows that 
$\mathrm{Fix}_{c}(G^{\dag}) \cap  \partial \overline{S}^{\dag} = \emptyset$. 

Let $H$ be a finitely generated subgroup of $G$. We get that $\phi$ is isotopic to 
$\mathrm{Id}$ rel $\mathrm{Fix} ({\mathcal C}^{(k)} (H))$ for all $k \geq 1$ and 
$\phi \in  {\mathcal C}^{(k)} (H)$ by $H' \subset H_0$ and Theorem \ref{teo:deriso}.
Thus 
$\phi^{\dag}$ is isotopic to 
$\mathrm{Id}$ rel 
$\mathrm{Fix} ({\mathcal C}^{(k)} (H^{\dag})) \cup (\overline{S}^{\dag} \setminus {S}^{\dag} )$ 
for all $k \geq 1$ and 
$\phi \in  {\mathcal C}^{(k)} (H)$. 
We obtain that $H^{\dag}$ has a contractible global fixed point in $\overline{S}^{\dag}$ by 
Theorem \ref{teo:baux}. The finite intersection property for compact sets implies that 
$G^{\dag}$ has a contractible global fixed point in $\overline{S}^{\dag}$. 
Since $k \geq 1$, the discussion in the previous paragraph implies 
$\mathrm{Fix}_{c} (G^{\dag}) \cap (\overline{S}^{\dag} \setminus {S}^{\dag} ) = \emptyset$
and hence $\mathrm{Fix}_{c} (G^{\dag}) \cap S^{\dag} \neq \emptyset$.
Since 
\[ \pi^{\dag} ( \mathrm{Fix}_{c} (G^{\dag}) \cap S^{\dag})  \subset 
\pi^{\dag} ( \mathrm{Fix}_{c} (G_{0}^{\dag}) \cap S^{\dag}) = \mathrm{Fix}_{c}(G_0), \] 
we obtain $\mathrm{Fix} (G) \cap \mathrm{Fix}_{c}(G_0) \neq \emptyset$.
This completes the proof of Theorem \ref{teo:mainb} for the case where $G$ is nilpotent.
\subsection{Locally nilpotent case}
\label{subsec:loc_nil}
A group $G$ is {\it locally nilpotent} if every finitely generated subgroup of $G$ is nilpotent.
For instance, consider a complex coordinate $w$ in the Riemann sphere. The group 
generated by $1/w$ and the rational rotations $e^{\frac{2 \pi i  a}{2^{k}}} w$ ($a \in {\mathbb Z}$, $k \in {\mathbb Z}_{\geq 0}$) 
of order power of $2$ is a locally nilpotent group of complex analytic diffeomorphisms of the sphere but is non-nilpotent.

Notice that a group of homeomorphisms such that every of its finitely generated subgroups has a finite orbit, does
not necessarily have a finite orbit. An example is the group of rational rotations of the circle. The situation is simpler in our case
since we have a uniform upper bound for the order of the smallest orbit of every finitely generated subgroup.
More precisely, 
Theorem \ref{teo:mainb} is a consequence of 
the next two lemmas and the version of Theorem \ref{teo:mainb} for nilpotent groups.
\begin{lem}
\label{lem:fin_orb} 
Let $G$ be a group of homeomorphisms of a compact space $S$. 
Suppose that there exists $m \in {\mathbb N}$ such that any finitely generated
subgroup of $G$ has a finite orbit with at most $m$ elements. Then $G$ has a finite
orbit with at most $m$ elements.
\end{lem}
\begin{lem}
\label{lem:fin_orb0} 
Let $G$ be a group of homeomorphisms of a compact surface $S$ of negative Euler characteristic. 
Suppose that there exists $m \in {\mathbb N}$ such that any finitely generated
subgroup $H$ of $G$ has a finite orbit with at most $m$ elements contained in $\mathrm{Fix}_{c}(H_0)$. 
Then $G$ has a finite orbit with at most $m$ elements contained in $\mathrm{Fix}_{c}(G_0)$.
\end{lem}
\begin{proof}[Proof of Lemmas \ref{lem:fin_orb} and \ref{lem:fin_orb0}]
Let us show Lemma \ref{lem:fin_orb0}. The proof of Lemma \ref{lem:fin_orb} is simpler.
Let  ${\mathcal T}$ be the set of finitely generated subgroups of $G$.
Given $H \in {\mathcal T}$,  
consider the set $K(H)$ of points $z \in \mathrm{Fix}_{c}(H_0)$ such that its $H$-orbit has at most $m$ elements.
It is a compact set. Moreover, it is non-empty by hypothesis. 
Given finitely many elements  $H_1, \hdots, H_k$ of ${\mathcal T}$, we have
$\langle H_1, \hdots, H_k \rangle \in {\mathcal T}$. Thus, we get 
\[ \emptyset \neq K \langle H_1, \hdots, H_k \rangle \subset \cap_{j=1}^{k} K(H_j) . \]
Since the family $\{{K(H)\}_{H \in {\mathcal T}}}$ of compact subsets of $S$
 has the finite intersection property,  there exists
$z_0 \in \cap_{H \in {\mathcal T}} K(H)$. Moreover, $z_0$ belonts to $\mathrm{Fix}_{c} (G_0)$ and its $G$-orbit has
at most $m$ elements.
\end{proof}


\bibliography{rendu}
\end{document}